\documentclass[letterpaper,reqno,11pt,final,oneside]{amsart} 
\pdfoutput=1

\usepackage{amsmath,amsthm,amsfonts,amssymb}
\usepackage[extension=pdf]{hyperref}
\usepackage{enumitem,comment}
\usepackage[LGR,T1]{fontenc}
\usepackage{times}
\usepackage{mathtools}
\usepackage{mathrsfs}
\usepackage{cases}
\usepackage{caption}
\usepackage{bm}
\usepackage[scr=boondoxo,scrscaled=1.05]{mathalfa}
\usepackage[dvipsnames]{xcolor}
\usepackage[multiple]{footmisc}
\usepackage{microtype}
\usepackage{xspace}
\usepackage{tikz}
\usetikzlibrary{matrix,backgrounds,decorations.pathreplacing}
\usepackage{rotating}
\usepackage[totalheight=9.6in,totalwidth=6.15in,centering]{geometry}
\usepackage{graphics,graphicx}
\usepackage{tocvsec2} 
\usepackage{mathtools}
\usepackage{accents}
\usepackage{xr}
\usepackage{ifthen}
\usepackage[color,notref,notcite]{showkeys}

\usepackage[style=alphabetic,sorting=nyt,natbib=true,maxnames=99,isbn=false,doi=false,url=false,giveninits=true,hyperref=auto,arxiv=abs,backend=bibtex]{biblatex}
\AtEveryBibitem{%
\clearlist{language}}
\DeclareFieldFormat[article,inbook,incollection,inproceedings,patent,thesis,unpublished]{title}{#1\isdot}
\renewbibmacro{in:}{%
\ifentrytype{article}{}{%
 \printtext{\bibstring{in}\intitlepunct}}} 
\defbibheading{apa}[\refname]{\section*{#1}}
\DeclareFieldFormat{sentencecase}{\MakeSentenceCase{#1}} \renewbibmacro*{title}{%
\ifthenelse{\iffieldundef{title}\AND\iffieldundef{subtitle}}{}
{\ifthenelse{\ifentrytype{article}\OR\ifentrytype{inbook}%
   \OR\ifentrytype{incollection}\OR\ifentrytype{inproceedings}%
   \OR\ifentrytype{inreference}} {\printtext[title]{%
     \printfield[sentencecase]{title}%
     \setunit{\subtitlepunct}%
     \printfield[sentencecase]{subtitle}}}%
 {\printtext[title]{%
     \printfield[titlecase]{title}%
     \setunit{\subtitlepunct}%
     \printfield[titlecase]{subtitle}}}%
 \newunit}%
\printfield{titleaddon}}

\usepackage[f]{esvect}
\usepackage{stmaryrd}

\DeclareFontFamily{U}{BOONDOX-calo}{\skewchar\font=45 }
\DeclareFontShape{U}{BOONDOX-calo}{m}{n}{
  <-> s*[1.05] BOONDOX-r-calo}{}
\DeclareFontShape{U}{BOONDOX-calo}{b}{n}{
  <-> s*[1.05] BOONDOX-b-calo}{}
\DeclareMathAlphabet{\mcb}{U}{BOONDOX-calo}{m}{n}
\SetMathAlphabet{\mcb}{bold}{U}{BOONDOX-calo}{b}{n}

\makeatletter
\newcommand*\bigcdot{\mathpalette\bigcdot@{.5}}
\newcommand*\bigcdot@[2]{\mathbin{\vcenter{\hbox{\scalebox{#2}{$\m@th#1\bullet$}}}}}
\makeatother

\definecolor{darkblue}{rgb}{0.13,0.13,0.39}
\hypersetup{colorlinks=true,urlcolor=darkblue,citecolor=darkblue,linkcolor=darkblue,%
  pdftitle=PNG and Toda,pdfauthor={K. Matetski, J. Quastel, D. Remenik}}

\mathtoolsset{showmanualtags,showonlyrefs}

\newtheorem{thm}{Theorem}[section] 
\newtheorem{lem}[thm]{Lemma}
\newtheorem{prop}[thm]{Proposition} 
\newtheorem{cor}[thm]{Corollary}
\theoremstyle{definition} 
\newtheorem{rem}[thm]{Remark}

\newtheorem*{notation}{Notation}

\newcommand{\ext}{{\uptext{ext}}}
\newcommand\distr{\mathrel{\overset{\makebox[0pt]{\mbox{\normalfont\tiny (d)}}}{=}}}

\newcommand{\leqdot}{\text{\rlap{\raisebox{1pt}{$\,\,\ts\cdot$}}$\leq$}}
\newcommand{\geqdot}{\text{\rlap{\raisebox{1pt}{$\ts\cdot$}}$\geq$}}
\newcommand{\SL}{\mathscr{L}}

\renewcommand{\d}{\mathrm{d}}

\newcommand{\W}{W}
\newcommand{\fh}{\mathfrak{h}}

\newcommand{\cX}{\mathcal{X}}
\newcommand{\cY}{\mathcal{Y}}
\newcommand{\cU}{\mathcal{U}}
\newcommand{\cV}{\mathcal{V}}
\newcommand{\ux}{\underline{x}}
\newcommand{\ox}{\overline{x}}
\newcommand{\y}{y}

\newcommand{\I}{{\rm i}} 
\newcommand{\pp}{\mathbb{P}} 
 
\newcommand{\ee}{\mathbb{E}} 
\newcommand{\rr}{\mathbb{R}}
\newcommand{\nn}{\mathbb{N}} 
\newcommand{\zz}{\mathbb{Z}}
\newcommand{\cc}{{\mathbb{C}}}

\newcommand{\p}{\partial}
\newcommand{\uno}[1]{\mathbf{1}_{#1}}
\newcommand{\ep}{\varepsilon}

\newcommand{\vs}{\vspace{6pt}}

\newcommand{\qand}{\quad\text{and}\quad}
\newcommand{\qqand}{\qquad\text{and}\qquad}

\newcommand{\Fep}{\FF}

\newcommand{\FF}{\mathfrak{F}}
\newcommand{\QQ}{\mathfrak{Q}}
\newcommand{\gq}{\mathfrak{q}}
\newcommand{\FFF}{\hat{F}}
\newcommand{\PPP}{\mathscr{P}}

\newcommand{\J}{\mathscr{T}}

\newcommand{\ts}{\hspace{0.1em}}
\newcommand{\tts}{\hspace{0.05em}}
\newcommand{\tsm}{\hspace{-0.1em}}
\newcommand{\ttsm}{\hspace{-0.05em}}

\makeatletter
\newcommand\RedeclareMathOperator{%
  \@ifstar{\def\rmo@s{m}\rmo@redeclare}{\def\rmo@s{o}\rmo@redeclare}%
}
\newcommand\rmo@redeclare[2]{%
  \begingroup \escapechar\m@ne\xdef\@gtempa{{\string#1}}\endgroup
  \expandafter\@ifundefined\@gtempa
     {\@latex@error{\noexpand#1undefined}\@ehc}%
     \relax
  \expandafter\rmo@declmathop\rmo@s{#1}{#2}}
\newcommand\rmo@declmathop[3]{%
  \DeclareRobustCommand{#2}{\qopname\newmcodes@#1{#3}}%
}
\@onlypreamble\RedeclareMathOperator
\makeatother
\newcommand{\uptext}[1]{\text{\upshape{#1}}}
\DeclareMathOperator{\tr}{\uptext{tr}}
\DeclareMathOperator{\argmax}{\uptext{argmax}}
\RedeclareMathOperator{\det}{\mathop{\uptext{det}}}
\RedeclareMathOperator{\exp}{\mathop{\uptext{exp}}}
\RedeclareMathOperator{\log}{\mathop{\uptext{log}}}
\RedeclareMathOperator*{\lim}{\mathop{\uptext{lim}}}
\RedeclareMathOperator*{\sup}{\mathop{\uptext{sup}}}
\RedeclareMathOperator*{\limsup}{\mathop{\uptext{lim\hspace{1pt}sup}}}
\RedeclareMathOperator*{\liminf}{\mathop{\uptext{lim\hspace{1pt}inf}}}
\RedeclareMathOperator*{\max}{\mathop{\uptext{max}}}
\RedeclareMathOperator*{\inf}{\mathop{\uptext{inf}}}
\RedeclareMathOperator*{\min}{\mathop{\uptext{min}}}
\RedeclareMathOperator*{\cos}{\mathop{\uptext{cos}}}
\RedeclareMathOperator*{\sin}{\mathop{\uptext{sin}}}
\RedeclareMathOperator*{\arg}{\mathop{\uptext{arg}}}
\RedeclareMathOperator{\Re}{\uptext{Re}}
\RedeclareMathOperator{\Im}{\uptext{Im}}
\DeclareMathOperator{\PNG}{\uptext{PNG}}

\newcommand{\twopii}[1]{\ifthenelse{#1=1}{2\pi\I}{(2\pi\I)^{#1}}}

\newcommand{\ft}{t}
\newcommand{\fx}{x}
\newcommand{\fN}{\mathsf{N}}
\newcommand{\fE}{\mathbf{E}}
\newcommand{\fK}{\mathbf{K}}
\newcommand{\fI}{\mathbf{I}}
\newcommand{\ftau}{\tau}
\renewcommand{\P}{\chi}
\newcommand{\bP}{\bar{\P}}
\newcommand\munderbar[1]{\underaccent{\bar}{#1}}

\let\Re\relax
\DeclareMathOperator{\Re}{Re}
\DeclareMathOperator{\hypo}{\uptext{hypo}}

\DeclareMathOperator{\hit}{\uptext{hit}}
\DeclareMathOperator{\nohit}{\uptext{no hit}}
\DeclareMathOperator{\UC}{\mathscr{U\!C}}
\DeclareMathOperator{\LC}{\mathscr{L\!C}}
\DeclareMathOperator{\gdet}{\uptext{det}}
\DeclareMathOperator{\gcr}{\uptext{cr}}

\def\dash---{\kern.16667em---\penalty\exhyphenpenalty\hskip.16667em\relax}

\numberwithin{equation}{section}

\let\oldmarginpar\marginpar
\renewcommand\marginpar[1]{\-\oldmarginpar[\raggedleft\footnotesize #1]%
  {\raggedright{\small\textsf{#1}}}}

\let\oldFootnote\footnote
\newcommand\nextToken\relax
\renewcommand\footnote[1]{%
    \oldFootnote{#1}\futurelet\nextToken\isFootnote}
\newcommand\isFootnote{%
    \ifx\footnote\nextToken\textsuperscript{,}\fi}

\allowdisplaybreaks[2]

\settocdepth{subsection}

\begin{document}

\title{Polynuclear growth and the Toda lattice}

\author{Konstantin Matetski} \address[K.~Matetski]{
  Department of Mathematics\\
  Michigan State University\\
  619 Red Cedar Road\\
  East Lansing, MI 48824\\
  USA} \email{matetski@msu.edu}

\author{Jeremy Quastel} \address[J.~Quastel]{
  Department of Mathematics\\
  University of Toronto\\
  40 St. George Street\\
  Toronto, Ontario\\
  Canada M5S 2E4} \email{quastel@math.toronto.edu}

\author{Daniel Remenik} \address[D.~Remenik]{
  Departamento de Ingenier\'ia Matem\'atica and Centro de Modelamiento Matem\'atico (IRL-CNRS 2807)\\
  Universidad de Chile\\
  Av. Beauchef 851, Torre Norte, Piso 5\\
  Santiago\\
  Chile} \email{dremenik@dim.uchile.cl}

\begin{abstract}  
It is shown that the polynuclear growth model is a completely integrable Markov process in the sense that its transition probabilities are given by Fredholm determinants of kernels produced by a scattering transform based on the invariant measures modulo the absolute height, continuous time simple random walks.
From the linear evolution of the kernels, it is shown that the $n$-point distributions are determinants of $n\times n$ matrices evolving according to the \emph{two dimensional non-Abelian Toda lattice}.
\end{abstract}

\maketitle

\settocdepth{section}

\tableofcontents 

\section{Introduction}

One-dimensional polynuclear growth (PNG) is a model for randomly growing interfaces which has been extensively studied as a solvable model in the Kardar-Parisi-Zhang universality class.
The KPZ class is characterized by non-standard fluctuations, which are universal for models in the class, but do remember the initial data.  Most previous work on PNG has considered the \emph{droplet} or \emph{narrow wedge} initial condition.
For this special initial condition, the one-point distributions of the model can be identified with a Poissonized version of the longest increasing subsequence of a random permutation.
Ulam's problem of determining the law of large numbers can be understood through the connection with the Hammersley process \cite{hammersley,MR1355056,seppalainenLIS,aldousDiaconis2}, and a representation as a Toeplitz, then later a Fredholm, determinant led to the identification of the fluctuations around this limit as coinciding with those of the top eigenvalue of a matrix from the Gaussian unitary ensemble \cite{gessel-P-rec, MR1618351,MR1727236, baikDeiftJohansson,johanssonShape,johanssonPlancherel}.
Multipoint distributions in this geometry can be computed based on a multilayer extension, through which the Airy process was discovered \cite{prahoferSpohn,johansson}.
Flat and stationary initial distributions were also considered in \cite{baikRainsF0,sasamotoImamuraPNG1,sasamotoImamuraPNG2,prahoferSpohn2, ferrariPNGGOE,ferrariSpohnRGM}.
Recently \cite{johanssonRahmanPNG2,johanssonRahmanPNG1} were able to produce certain multitime formulas for some of these special initial conditions.
One thing all such results had in common, for PNG as well as other related models, was the reliance on tricks particular to a few special initial conditions in order to obtain Fredholm determinant formulas suitable for asymptotics.

In \cite{fixedpt,kp} we introduced the \emph{KPZ fixed point}, a scaling invariant Markov process expected to be the universal limit of 1:2:3 scaled height functions, or their analogues, within the entire KPZ class.
The transition probabilities are given by Fredholm determinants of certain kernels based on a Brownian scattering transform of the initial data.
Furthermore, the $n$-point distributions, starting from an essentially arbitrary deterministic height function, come from matrix Kadomtsev-Petviashvilli (KP) equations: In the one-point case, they simply satisfy the KP-II equation; for higher $n$, they are given as traces of matrix solutions of such equations.

Scaling invariant fixed points are natural places to look for integrability. But it is also interesting to ask whether such an integrable structure exists already in one of the discrete or semi-discrete models in the KPZ class.
The question is particularly relevant as it may shed light on the physical origin of the integrable equations.
The KPZ fixed point formulas were derived through asymptotic analysis of similar determinantal formulas for the totally asymmetric simple exclusion process (TASEP).
However, the TASEP transition probabilities appear not to satisfy any of the standard integrable discretizations of KP.

In this article we consider PNG as a Markov process.
Its transition probabilities, again from essentially arbitrary deterministic initial data, are given by Fredholm determinants of kernels produced by a scattering transform based on continuous time symmetric random walks -- the invariant measures of the process modulo the absolute height.
For PNG, the $n$-point distributions turn out to be associated to the most appealing discretization of the matrix KP: In the one-point case, they are solutions of the \emph{two-dimensional Toda lattice}; for higher $n$, they are given as determinants of $n\times n$ matrices solving the \emph{non-Abelian Toda lattice}.
All of these are well-known \cite{MR810623} completely integrable extensions of the Toda lattice.
In this sense, PNG is shown to be a \emph{completely integrable Markov process.}

\subsection*{Outline} 

Sec. \ref{sec:main} contains the statement of our main results, including the Fredholm determinant formula for PNG and its connection to the Toda lattice.
In that section we also provide a heuristic argument showing that solutions of the Toda lattice equations in our context converge to solutions of the KP equation under the appropriate KPZ scaling, as well as a discussion of some connections to earlier work.
In Sec. \ref{sec:nonAbelianToda} we provide an abstract result which derives the non-Abelian Toda lattice equations for a general class of kernels satisfying certain structural conditions, and then show that the kernel appearing in the PNG formulas satisfies those conditions.
Sec. \ref{PNGMarkov} is devoted to the proof of the Fredholm determinant formula for PNG, which proceeds by checking directly that the determinant solves the Kolmogorov backward equation for the process.
Appdx. \ref{sec:multipt-ini} derives the initial conditions for the non-Abelian Toda equation for PNG, while Appdx. \ref{sec:tr-cl-estimate} provides some trace norm estimates which are required in the proofs.

\settocdepth{subsection}

\section{Main results} \label{sec:main}

\subsection{Model and setting}\label{sec:model}

The \emph{polynuclear growth model} (PNG) is a Markov process whose state space $\UC$ consists of upper semi-continuous height functions $h$ mapping $\mathbb{R}$ into the integers $\mathbb{Z}$ compactified at $\{-\infty\}$.
The topology on $\UC$ is that of local Hausdorff convergence of hypographs, which is natural for models of random growth (see Sec. \ref{states} for a precise definition).
The dynamics consists of two parts; a continuous, deterministic dynamics $h\mapsto \sup_{|y-x|\le t} h(y)$ and a discontinuous, stochastic dynamics.
In the stochastic part of the dynamics, there is a Poisson point process in space-time with rate $2$ (this choice of rate is arbitrary, but convenient in formulas).
If $(t^*,x^*)$ is such a point, then $h(t^*,x)=h((t^*)^-,x) + \mathbf{1}(x=x^*)$.
In other words, in each interval $[x,x + \mathrm{d}x]$, with probability $2\tts\mathrm{d}x\tts\mathrm{d}t$, a \emph{nucleation} appears in $[t,t+\mathrm{d}t]$ and the height function is increased by $1$ at $x$.
Both dynamics are running simultaneously\footnote{Note that the dynamics makes sense starting from \emph{any} function in $\UC$, as $ \sup_{|y-x|\le t} h(y)$ is finite for any function in $\UC$, at any $x$, since the supremum is achieved, and $h$ is not allowed to take the value $+\infty$ at any point.}, so the continuous part of the dynamics implies that the nucleations spread in both directions deterministically at speed $1$.

Almost equivalently, the model can be thought of as a collection of one-dimensional ‘kinks’ (down steps) and ‘antikinks’ (up steps) moving on $\mathbb R$.  They move ballistically at speed one:  kinks to the right, antikinks to the left.  They annihilate upon collision \cite{MR607609}, and are also produced, in pairs, a kink at $x^+$, an antikink at $x^-$, according to a rate $2$ space-time Poisson process.  The height $h(t,0)$ is just its value $h(0,0)$ at time $0$,  plus the number of kinks and antikinks which have crossed the spatial point $0$ up to time $t$.  One doesn't have to quibble about whether to count as crossings events like creations or annihilations exactly at $0$, as these happen with probability $0$.

\begin{figure}
  \centering
  \vspace{7pt}
  \includegraphics[width=5.25in]{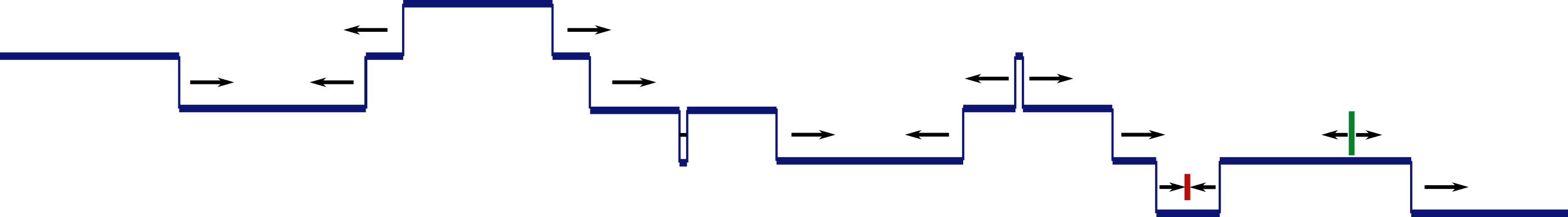}
  \caption*{\small \textsc{Figure.} ~In the PNG dynamics, flat islands with integer heights expand deterministically in both directions with speed $1$, merging when they collide (red), and with new islands (green) appearing above existing ones at the times of a space-time Poisson process with rate $2$.}
\end{figure}

If one thinks of the $\frac12(t+x)$ direction as time and the $\frac12(t-x)$ direction as space, or vice-versa, the dynamics becomes that of the Hammersley process \cite{MR1355056}.
Note, however, that this transformation does not take our initial data into the standard initial data for the Hammersley process.

We will denote the PNG height function at time $t$ by $h(t,x)$, or sometimes $h(t,x; h_0)$ if we want to indicate the initial data.
The \emph{narrow wedge at $y\in\rr$} is defined by
\begin{equation}\label{eq:def-nw}
\mathfrak{d}_y(y) = 0\qqand\mathfrak{d}_y(x) = -\infty\uptext{ for $x \neq y$}.
\end{equation}
Define
\[\mathscr{A}_t(x,y)=h(t,x;\mathfrak{d}_y),\]
i.e. the PNG height function at $x$ and at time $t$ when we start the dynamics with a narrow wedge at $y$.
The PNG dynamics (under the natural coupling where one starts with two or more different initial conditions and runs them using the same Poisson noise) preserve the max operation,
\begin{equation}
h(t,x; \max\{h_1,h_2\}) =\max\{ h(t,x;h_1),h(t,x;h_2)\},
\end{equation}
and hence one has the variational formula
\begin{equation}\label{varform}
h(t,x; h_0) =\sup_y\{\mathscr{A}_t(x,y) + h_0(y)\}.
\end{equation}
The equality holds simultaneously in $x\in \rr, t\ge 0$.  

$\mathscr{A}_t(x,y)$ can be recast in terms of a {\em Poissonian last passage percolation} model. 
Here one looks for Lipschitz-$1$ paths from $(0,x)$ to $(t,y)$ which pick up, along the way, a maximal number of points from the background space-time Poisson process. 
It is not hard to see that this maximal number, the {\em point-to-point last passage time}, coincides with $\mathscr{A}_t(x,y)$\footnote{See also the discussion leading to \eqref{varform2} below, which essentially proves this for PNG with general initial data in connection to curve-to-point Poissonian last passage times.}\footnote{This also establishes the connection between PNG with narrow wedge initial data $\mathfrak{d}_0$ and a Poissonized version of the longest increasing subsequence problem. Indeed, we have that $h(t,0)$ equals the maximal number of Poisson points visited by a Lipschitz-$1$ path going from $(0,0)$ to $(t,0)$, and all such paths lie inside the rhombus $R$ with vertices at those points and at $(t/2,\pm t/2)$. Now rotate the picture by $-45^\circ$, let $N$ denote the number of Poisson points inside the square corresponding to $R$ (so that $N$ is a Poisson$[t^2]$ random variable), and order these $N$ points according to their $x$ coordinate. The $y$ coordinates of these points define a random permutation $\sigma$ of $\{1,\dotsc,N\}$, which is clearly chosen uniformly from $S_N$, and $h(t,0)$ is then nothing but the length of the longest increasing subsequence in $\sigma$.\label{foot:poissonized}}.
From the symmetry of the background space-time Poisson process we also get
\begin{equation}\label{sym}
  \mathscr{A}_t(x,y)\distr \mathscr{A}_t(y,x) .
\end{equation}

A special property, connected to solvability, is \emph{skew time reversal invariance}:
\begin{equation}\label{eq:timerev}
\pp (h(t; g) \le -f) = \pp(h(t;f)\le -g)
\end{equation}
for all $f,g\in\UC$.
Here $h(t;g)$ is the height function starting from $g$, and $h(t;g)\le -f$ means $h(t,x;g)\le -f(x)$ for all $x\in \rr$.
The skew time reversal invariance is a simple consequence of the variational formula \eqref{varform} since, as events,
\begin{equation}
\left\{ h(t,x; f) \le -g \right\} =\left\{ \sup_{x,y} \{\mathscr{A}_t(x,y) + f(y) +g(x) \}\le 0\right\} .
\end{equation}
Switching $f$ and $g$ is the same as switching $x$ and $y$, and using \eqref{sym} we get \eqref{eq:timerev}.
In particular, skew time reversal invariance appears to come from the fact that the model has a polymer interpretation.
The skew time reversal invariance can also be stated as 
\begin{equation}\label{eq:timerev2}
\pp (h(t; g) \le f) = \pp(h^{\gets}(t; f)\ge g),
\end{equation}
where $h^{\gets}(t;f)$ is the PNG dynamics running backwards in time from $f$: Up steps move to the right, down steps move to the left, upward peaks merge on contact, and  downward troughs are created at rate $2$.  Another way to say it is that $-h$ evolves according to the usual PNG dynamics.
This last process naturally takes values in the space of lower semi-continuous functions $\LC\coloneqq\{h\!:-h\in\UC\}$, and in \eqref{eq:timerev2} we take $f\in\LC$ and $g\in\UC$.  We will show that the generator of this backward evolution is given by the adjoint $\hat{\SL}^*$ of the PNG generator $\hat{\SL}$.

Let $\Phi(g)$ be the indicator of the event that a path $g$ ever becomes less than or equal to $0$,
\begin{equation}\label{eq:def-Phi}
\Phi(g)=1-\uno{g> 0}.
\end{equation}  
The identity \eqref{eq:timerev2} also yields $\pp (h(t,x; g) \ge f(x)~\uptext{for some $x$}) = \pp(h^{\gets}(t,x; f)\le g(x)~\uptext{for some $x$})$, which we can then write as $\ee_g [\Phi(f-h(t))] = \ee_f[\Phi(h^{\gets}(t)-g ) ]$, and now taking the derivative at $t=0$ we get {\em{infinitesimal skew time reversibility}},
\begin{equation}\label{infinitesimalstr}
   \hat{\SL}_g\Phi(f-g)=  \hat{\SL}^*_f\Phi(f-g).
\end{equation}
Here the subscript indicates which variable the operator is acting on.

We record a few other key properties.  Typically for KPZ class models,
PNG is statistically invariant under spatial shifts and under reflections:
\begin{align} 
T_{-y}h(t,x; T_yh_0)&\stackrel{\uptext{dist}}{=}h(t,x;h_0), \qquad T_yf(x)=f(x-y),\\
Rh(t, x; Rh_0 )&\stackrel{\uptext{dist}}{=}h(t,x;h_0), \qquad Rf(x)=f(-x).
\end{align} 
Unusually among continuous time random growth models, PNG has a \emph{finite propagation speed}: Given everything up to time $t$, for $s>0$, $h(t+s,x)$ only depends on $h(t,y)$, $|y-x|\le s$, and on the points in the background Poisson process inside the space-time region $\{(u,y)\in\rr^2\!:t\leq u\leq t+s,\,|y-x|\le t+s-u$\}.

Modulo a non-trivial upward height shift, the invariant measures for the process consist of the following one-parameter family of two-sided continuous time random walks.  
Let $h^+(x)$, $h^-(x)$ be independent Poisson process on $\mathbb{R}$ with rates $\rho$ and $\rho^{-1}$.  
Then $h(x)=h^+(x)-h^-(x)$ is invariant modulo height shifts, in the sense that, starting PNG from such an initial state, with $h(0)=0$, at time $t>0$, $h(t,x) -h(t,0)$ will have the same distribution as the process in $x$. This is proved in Sec. \ref{sec:inv}.
A concrete choice is the symmetric invariant state with $\rho=1$, which we denote by $\fN(x)$.
In the hitting probabilities below, the determinant turns out to be unaffected by the choice of $\rho$.

The transition probabilities of a Markovian growth model such as PNG are completely determined by the finite dimensional distributions
\begin{equation}\label{mtp}
F(t,x_1,\ldots,x_n, r_1,\ldots,r_n) = \pp_{h_0}( h(t,x_i)\le r_i, i=1,\ldots,n),
\end{equation}
for an initial condition $h_0\in\UC$, $t>0$, $x_1<\cdots <x_n$, and $r_1,\ldots,r_n\in \mathbb{Z}$.

The PNG dynamics is surprisingly similar to \emph{ballistic aggregation} (sometimes called \emph{ballistic deposition}) where the height functions live over $\zz$ instead of $\rr$ and $h(t,x)$ jumps to $\max\{ h(t,x-1), h(t,x)+1, h(t,x+1)\}$ at rate one. 
One of the challenges in the field of random growth is that we know almost nothing about many natural and basic models such as ballistic aggregation (in this case, only that it grows at some rate $Ct$ \cite{MR1797390} and plenty of computer simulations).
In contrast, for PNG we are now going to explain how the entire family \eqref{mtp} comes from known integrable non-linear wave equations.

\subsection{Toda}

We need to introduce some notation.
For $a\leq b$ consider the \emph{hit operator}, with kernel
\begin{equation}\label{eq:P-hit}
P^{\hit(h_0)}_{a, b}(u,v) = \pp_{\fN(a)=u}(\ftau \in [a, b],\,\fN(b) = v),\qquad \ftau=\inf\{x\ge a : \fN(x)\in \hypo(h_0)\},
\end{equation}
where $\hypo(h_0)=\{(x,g)\in\rr\times\zz\!:g\leq h_0(x)\}$ is the \emph{hypograph} of $h_0$, closed since $h_0\in\UC$, and $\fN$ is the continuous time random walk invariant measure introduced in the paragraph before \eqref{mtp}.
If $a>b$ we set $P^{\hit(h_0)}_{a, b}\equiv0$.
We need the symmetric first and second order discrete difference operators
\begin{align}\label{eq:nabla}
\nabla f(u) &= \tfrac{1}{2}(f(u+1) - f(u-1)),\qquad
\Delta f(u) = f(u+1) + f(u-1) - 2 f(u).
\end{align}
The latter is the infinitesimal generator of the random walk $\fN$, that is, for $x \geq 0$ and $u,v\in\zz$ we have
\begin{equation}\label{eq:PNG-random-walk}
\pp_{\fN(0)=u}(\fN(x) = v) = e^{x\Delta}(u,v).
\end{equation}
The right hand side has an explicit contour integral formula, see \eqref{eq:Scontour}, which is in fact valid for all $x\in\rr$ and exhibits $(e^{x\Delta})_{x\in\rr}$ as a group of operators acting on suitably decaying functions in $\ell^2(\zz)$.
The following should be thought of as a \emph{scattering transform} of $h_0$ by random walks,
\begin{equation}\label{eq:hatKh0ab}
\J^{h_0}_{a, b} = e^{a\Delta} P^{\hit(h_0)}_{a, b} e^{-b\Delta}.
\end{equation}
Out of it, and for a fixed choice of $t\geq0$ and $x_1<\dotsm<x_n$, we build an \emph{extended kernel} acting on $\ell^2(\{x_1,\dotsc,x_n\}\times\zz)$:
\begin{align}
K^{\uptext{ext}}(x_i,\cdot;x_j,\cdot) &= - \uno{\fx_i < \fx_j} e^{(\fx_j-\fx_i) \Delta} + e^{- 2 \ft \nabla - \fx_i \Delta} \J^{h_0}_{\fx_i - \ft, \fx_j + \ft} e^{2 \ft \nabla + \fx_j \Delta}\label{eq:PNG-kernel-two-sided-hit};
\end{align} 
as before, $e^{t\nabla}$ is well defined for all $t\in\rr$ through a contour integral formula for its kernel, see \eqref{eq:Scontour}.

Now given $r=(r_1,\dotsc,r_n)\in\zz^n$, we use the extended kernel $K^{\uptext{ext}}$ to build an $n\times n$ matrix kernel $K=K(t;x_1,\dotsc,x_n;r_1,\dotsc,r_n;h_0)$ acting on $\oplus_n\ell^2(\zz_{>0})\coloneqq\ell^2(\zz_{>0})\oplus\dotsm\oplus\ell^2(\zz_{>0})$, the $n$-fold direct sum of $\ell^2(\zz_{>0})$:
\begin{equation}
(K_r)_{ij}(u,v)=K^{\uptext{ext}}(x_i,u+r_i;x_j,v+r_j).\label{eq:matrixKPNG}
\end{equation}
We stress that, although we only include the vector $r$ in the notation, $K_r$ depends on $t$, $x_1,\dotsc,x_n$, and $h_0$ as well.
Let
\begin{equation}
r_0(t,x)=\sup_{|y-x|\leq t}h_0(y)\label{eq:r0},
\end{equation}
which is simply the level reached at $x$ by the deterministic part of the dynamics, starting with $h_0$, so
\[F(t,x_1,\dotsc,x_n,r_1,\dotsc,r_n)=0\quad\uptext{if $r_i<r_0(t,x_i)~~$ for some $i\in\{1,\dotsc,n\}$}.\]
We will show that if $r_i\geq r_0(t,x_i)$ for each $i$, then $I-K_r$ is invertible; we will write 
\begin{equation}
R_r=(I-K_r)^{-1}
\end{equation}
for this inverse. 
Introduce also the variables
\begin{equation}
\zeta= \tfrac12(t-\overline x)\qand \eta = \tfrac12(t+\overline x),\qquad \overline x = x_1+\cdots+x_n,
\end{equation}
write 
\[r\pm1=(r_1,\dotsc,r_n)\pm (1,\dotsc,1)\qqand\p_{\overline x}=\p_{x_1}+\dotsm+\p_{x_n},\]
and consider the following element of $GL(n)$, depending on $\zeta,\eta$, and $r$, and obtained by setting $(u,v)=(1,1)$ in the extended kernel $R_r(u,v)$:
\begin{align}\label{eq:26}
Q_r=R_r(1,1).
\end{align}

\subsubsection{Non-Abelian 2D Toda equation}

We can now state the main result.

\begin{thm}\label{thm:nonAbelian}
Let $h_0\in\UC$ be deterministic.
For $r=(r_1,\dotsc,r_n)$ such that $r_i>r_0(t,x_i)$ for each $i$, the matrix function $Q_r$
satisfies the non-Abelian 2D Toda equation
\begin{equation}
\p_{\zeta}(\p_\eta Q_{r} Q_{r}^{-1})+  Q_{r} Q_{r-1}^{-1}-  Q_{r+1} Q_{r}^{-1}=0,\label{3h'}
\end{equation}
while
\begin{equation}
\frac{F(t, x_1,\ldots, x_n, r_1+1,\ldots, r_n+1)}{F(t, x_1,\ldots, x_n, r_1,\ldots, r_n)}=\det Q_r.\label{eq:Frs-detQr1}
\end{equation}
\end{thm}

The equation has to be supplemented by initial conditions and boundary conditions. See Sec. \ref{sec:bdc} and Appdx. \ref{sec:multipt-ini} for a discussion.

The non-Abelian 2D Toda equation \eqref{3h'} arises from the special dependence of the kernel $K_r$ on $r$, $\eta$ and $\zeta$ (or, more precisely, the $r_i$'s, $t$ and the $x_i$'s).
The theorem is proved in Sec. \ref{sec:nonAbelianToda}.  It was inspired by \cite{MR891103}, where a special case is considered in an abstract setting (see Rem. \ref{rem:bp}). 

The non-Abelian 2D Toda equation forms a well-known integrable system \cite{MIKHAILOV198173}. 
Lax and Zakharov-Shabat (zero-curvature) forms of the equation can be found in Sec. 3.1 of \cite{MR810623} (see also \cite{MR3803581}).
They arise from compatibility of the equations for 
vectors $\Phi_r$
\cite{MR1490247}
\begin{equation}
\p_\eta\Phi_r=V_r\Phi_r+\lambda_r\Phi_{r-1},\qquad \p_t\Phi_r = U_r\Phi_{r+1},
\end{equation}
with
\begin{equation}
U_r=Q_rQ_{r-1}^{-1},\qquad V_r=-\p_\eta Q_{r}Q_{r}^{-1}.\label{eq:27}
\end{equation}
The non-Abelian 2D Toda equation is sometimes written in terms of $U_r$ and $V_r$ as
\begin{equation}
\p_\eta U_r + V_{r}U_r - U_rV_{r-1} =0,\qquad \p_\zeta V_r +U_{r+1}- U_{r} =0
\label{2a'}
\end{equation}
(the first equation is an elementary consequence of the definition \eqref{eq:27} of $U_r$ and $V_r$, while the second one is equivalent to \eqref{3h'}).

When $n=1$, there is no need to take the determinant in \eqref{eq:Frs-detQr1} and the previous theorem reduces to the following more accessible statement:

\begin{cor}\label{thm:PNGToda}  
Let $h_0\in\UC$ be deterministic.
Then  
\begin{equation}\label{oneddist}
F_r(t,x) = \pp_{h_0}( h(t,x) \le r)
\end{equation} 
satisfies the scalar 2D Toda equation
\begin{equation}\label{eq:waveToda}
\tfrac14(\p_t^2-\p_x^2)\log F_{r}=\frac{F_{r+1}F_{r-1}}{F_r^2}-1
\end{equation}
for $t>0$ and $r> r_0(t,x)$ given by \eqref{eq:r0}.
\end{cor}

To see how the corollary follows from Thm. \ref{thm:nonAbelian}, note that from \eqref{eq:Frs-detQr1} and \eqref{eq:27} we have  $U_r= \frac{F_{r+1}F_{r-1}}{F_r^2}$.  On the other hand, since in this case $U_r$ and $V_r$ are scalars, the first equation in \eqref{2a'} reads $\p_\eta\log U_r+V_r-V_{r-1}=0$.
Taking the $\zeta$ derivative of this identity and using the second equation in \eqref{2a'} gives $\partial^2_{\zeta\eta} \log U_r = U_{r+1}-2U_r + U_{r-1}$.
But $\log U_r = \log F_{r+1}-2\log F_r + \log F_{r-1}$ and $\partial^2_{\zeta\eta}= \tfrac14(\p_t^2-\p_x^2)$, so this is just a second difference version of \eqref{eq:waveToda}.

\subsubsection{1D Toda lattice}

In the case of flat initial data, $h_0\equiv 0$, $F_r$ is independent of $x$, the $x$ derivative drops out of \eqref{eq:waveToda}, and taking $g_r=\log F_r-\log F_{r-1}$ and $t\longmapsto 2t$ we obtain the classic Toda lattice,
\begin{equation}\label{1ode}
\ddot{g}_r = e^{ g_{r+1}-g_r } - e^{ g_r- g_{r-1}}. 
\end{equation}

\subsubsection{Boundary conditions}\label{sec:bdc}

In the one-point case, one needs to supplement \eqref{eq:waveToda} with boundary conditions at 
$t=0$ and  $r=r_0(t,x)$.  At $t=0$ we clearly have from the definition \eqref{oneddist} of $F_r(t,x)$ that
 \begin{equation}
F_r(0,x)=\uno{r\geq h_0(x)}.\label{eq:1pt-ini-1}\end{equation}
Since we are dealing with a wave equation, we require an extra piece of data at $t=0$, which in simple cases is given by 
\begin{equation}
\p_tF_r(0,x)=-\textstyle\sum_{y}\!\left((h_0(y)-h_0(y^-))\uno{h_0(y^-)\leq r<h_0(y)}+(h_0(y)-h_0(y^+))\uno{h_0(y^+)\leq r<h_0(y)}\right)\tsm\delta_y(x).\label{eq:1pt-ini-2}
\end{equation}
This comes from the deterministic part of the dynamics: The random part does not contribute to the derivative because a jump at $x$ coming from a nucleation in a time interval of length $\ep$ has probability of order $\ep^2$.
From the dynamics of the process, $h(t,x)$ can never be smaller than $r_0(t,x)$.  Hence $F_r(t,x)=0$ for $r<r_0(t,x)$.  This leaves open the fate of $F_{r_0(t,x)}(t,x)$.  We clearly have $\frac{F_{r_0+1}F_{r_0-1}}{F_{r_0}^2}=0$, since $F_{r_0-1}=0$, and $F_{r_0}>0$, since if there are no nucleations in the backward light cone of $(t,x)$, then 
$h(t,x)=r_0(t,x)$, and this has positive probability (note however, that other scenarios can also lead to $h(t,x)=r_0(t,x)$ so $F_{r_0}$ is not trivial to compute.)  
This would appear to give $\frac14(\p_t^2-\p_x^2) \log F_{r}(t,x)\mid_{r=r_0(t,x)} = -1$, but it is not quite true because $F_{r_0}$ has jumps when $r_0$ has jumps.  These can be worked out in the following way.
First of all, $r_0(t,x)$ is computed in an elementary way from the initial data $h_0(x)$.  It has discontinuities of the first kind (jumps) on (annihilating) lines of slope either $1$ or $-1$ emerging from the initial jumps of $h_0$ and at each discontinuity point $(t,x)$ we have $ r_0(t^-,x)<r_0(t^+,x)$.  It is not hard to see that $F_{r_0(t,x)}(t,x)$ must jump from $F_{r_0(t^-,x)}(t^-,x)$ to  $F_{r_0(t^+,x)}(t^-,x)$ as we cross the discontinuity line.  In the interior of the regions bounded by the discontinuity lines, we do have $\frac14(\p_t^2-\p_x^2) \log F_{r}(t,x)\mid_{r=r_0(t,x)} = -1$.  Because the jump is actually computed using the value of $F$ at $t^-$, $F_{r_0(t,x)}(t,x)$ can now be computed everywhere.   
Again one requires an extra boundary condition which is provided by $\lim_{r\to \infty}F_r(t,x)=1$.
We believe that under these conditions the 2D Toda equation is well posed, but this is very far from anything available in the literature, and we leave it for future work.
In the special narrow wedge and flat cases, these boundary conditions are very easy to state:  

\vskip3pt
\noindent\emph{Narrow wedge}.    
$F_r(0,x)=\uno{x\neq0}+\uno{x=0,\,r\geq0}$,   $\p_tF_r(0,x)=-2\delta_0(x)$. For $|x|\leq t$,  $r_0(t,x) = 0$, while $r_0=-\infty$ otherwise, and  $F_0(t,x)=\exp\{-(t^2-x^2)\}$, $F_{k}(t,x) = 0$, $k<0$, $|x|\leq t$.

\vskip3pt
\noindent\emph{Flat}.
$F_r(0,x)= \uno{r\ge 0}$, $\p_t F_r(0,x) =0$ and $r_0\equiv 0$ with $F_0(t,x)=\exp\{-2t^2\}$ and $F_{k}(t,x) = 0$, $k<0$.

\vskip3pt
The boundary conditions for the non-Abelian Toda equation are complicated to state and will be partially addressed in Appdx. \ref{sec:multipt-ini}.

\subsubsection{Conserved quantities}

It is natural to inquire as to the physical meaning of the integrals of motion (see for example \cite{MIKHAILOV198173}) of the 2D Toda equations in the present context. However, the initial data is so singular that all the standard ones are infinite.  The same situation holds for the KP equations for the KPZ fixed point. Hope is not lost; presumably others can be found.  We are simply in a situation where the dynamics moves through states which have not been previously studied.  For example, it seems in our context, solitons are nowhere to be seen.   In an analogous situation for the KdV equation starting with white noise,  a host of supplementary integrals requiring far less regularity than the usual ones turn out to be available \cite{MR2365449,MR4145790}. We leave this problem in the  Toda context for future work.   However, integrability is far more than the collection of integrals. In the present situation we have in a sense constructed the transition probabilities via determinants out of a vast class of martingales $f(t,h(t))$ with $f$ coming from the integral kernel $\langle \phi,K(T-t,h)\psi\rangle$ and test functions $\phi,\psi$ on $\zz$ (see \eqref{kol}).
Constructability of these transition probabilities is our goal, and is what we mean in this context by complete integrability, and therefore these martingales play the role of the relevant conserved quantities. 

\subsubsection{Related work}

It was known that the narrow wedge PNG one-point distribution is connected to Toda $\tau$-functions  via Schur measures \cite{MR1856553}.  This can be partially extended to the flat case (see Sec. \ref{flat1}).  It is far from clear however, if, or how, such methods could apply to general initial data.  Because of this prevailing paradigm, the fact that general initial data lead to Toda actually comes as something of a surprise.  There is a point of view that all these integrable models lead to $\tau$-functions.  But this is refuted by the key integrable model which led to the discovery of the KPZ fixed point, TASEP, which does not appear to fit this mold so nicely.

Very recently \citet{cafassoRuzza} studied a finite temperature version of the narrow wedge solution, with a Fermi factor analogous to the deformation of the Airy kernel which gives the narrow wedge solution of KPZ \cite{MR2796514}.     Note that it is far from clear what a positive temperature deformation of PNG should look like; an interesting suggestion is studied in \cite{deformedPNG}.
Using Riemann-Hilbert methods, \cite{cafassoRuzza} show that the Fredholm determinant satisfies the same equation as ours in the one-point narrow wedge case, and also obtains refined $t\searrow0$ asymptotics there.  

The scalar 2D Toda lattice also comes up in generating functions for Hurwitz numbers \cite{MR1783622} (see also \cite{MR3349849}), ``moments'' of conformal maps of simply connected domains \cite{MR1785428},  as well as in the intriguingly related situation of the multilayer heat equation \cite{MR3439221} (see also \cite{oconnnel-quantumtoda}, where the partition function of a Brownian directed polymer model is characterized in terms of a diffusion process associated with the quantum Toda lattice). 
The 1D non-Abelian case appeared earlier, governing spin correlations in the inhomogeneous XY model \cite{PERK1978163}.
A mean field analysis of PNG in \cite{Ben_Naim_1998} somewhat surprisingly leads to a directed version of the Toda lattice.

\subsection{Fredholm determinant formula}\label{sec:fred}

The proof of Thm. \ref{thm:nonAbelian} and Cor. \ref{thm:PNGToda} is based on the following explicit Fredholm determinant formula for the multipoint distribution of the PNG height function, given in terms of the kernel $K_r$ from \eqref{eq:matrixKPNG}:

\begin{thm}\label{thm:PNG-fred}
Let $h_0\in\UC$ be deterministic and fix any $x_1,\dotsc,x_n\in\rr$ and $r_1,\dotsc,r_n\in\zz$ such that  $r_i\geq r_0(t,x_i)$ (see \eqref{eq:r0}) for each $i=1,\dotsc,n$.
The finite dimensional distributions \eqref{mtp} are given by the Fredholm determinant
\begin{equation}\label{fred} 
F(t, x_1,\ldots, x_n, r_1,\ldots, r_n)= \det(I - K_r)_{\oplus_n\ell^2(\zz_{>0})}.
\end{equation}
\end{thm}

Note that the determinant in \eqref{fred} is the Fredholm determinant of an operator acting on the $n$-fold direct sum of $\ell^2(\zz_{>0})$, while \eqref{eq:Frs-detQr1} is that of an $n\times n$ matrix.

The way we originally derived this formula was as a limit of an analogous formula for parallel TASEP with general initial data obtained in \cite{caterpillars}.
In the regime when each particle in parallel TASEP attempts to move with probability $p$ close to $1$, the height function of that model can be thought of as a discrete time version of PNG \cite{gatesWestcott}.
For parallel TASEP there is a biorthogonal ensemble representation, derived in \cite{borodFerSas}, from which a Fredholm determinant formula for its multipoint distributions can be derived based on a generalization of the method introduced in \cite{fixedpt} in the context of continuous time TASEP \cite{caterpillars}.
We will pursue this route to \eqref{fred} in an upcoming paper, where we will also show that in fact parallel TASEP converges to PNG, after appropriate rescaling, as $p\to1$.
A potential alternate approach could be to start from the formula in Thm. 2 of \cite{johanssonRahmanPNG1}, which is an analogue of Schutz's formula for TASEP in the PNG context.
In principle, one could hope to pass from this to a Fredholm determinant formula.
But it has only been accomplished for a few special initial conditions (see \cite{johanssonRahmanPNG1}).
In this paper, we will prove Thm. \ref{thm:PNG-fred} by directly checking that the right hand side of \eqref{fred} solves the Kolmogorov backward equation for PNG (see Sec.\ref{sec:whyhit1}).

\begin{rem}
The finite propagation speed of the PNG dynamics implies that if $|x_1-x_2|>2t$ then $h(t,x_1)$ and $h(t,x_2)$ are independent.
This fact can actually be seen from the Fredholm determinant formula \eqref{fred}: restricting to the two-point case for simplicity, in such a situation we have $P^{\hit(h_0)}_{x_2-t,x_1+t}=0$, so $(K_r)_{2,1}=0$, and thus the kernel is upper-triangular, which means that its Fredholm determinant equals the product of the Fredholm determinants of its diagonal entries, i.e. $\det(I-(K_r)_{1,1})_{\ell^2(\zz_{>0})}\det(I-(K_r)_{2,2})_{\ell^2(\zz_{>0})}$.
Note that this means, in particular, that there is no dependence on the $(K_r)_{1,2}$ entry of the kernel, which is the only one depending on $h_0$ outside of $[x_1-t,x_1+t]\cup[x_2-t,x_2+t]$ (which cannot affect the distribution of $h(t,x_1)$ and $h(t,x_2)$).
The same argument works when there are several clusters of $x_i$'s which are separated by distance at least $2t$.
\end{rem}

The non-Abelian Toda equations for PNG follow from \eqref{fred} and the structure of the kernel $K_r$.
A key relation to this effect which is satisfied by the kernel is 
\begin{equation}
\partial_\eta K_r(u,v) = K_{r-1}(u+1,v) - K_r(u + 1,v),\quad
\partial_\zeta K_r(u,v) = K_{r-1}(u,v+1)-K_r(u,v + 1),\label{eq:lin-intro}
\end{equation}
provided $r_i\geq r_0(t,x_i)$ for each $i$.
The two identities together yield $\p_\eta\p_\zeta K_r=K_{r+1}-2K_r+K_{r-1}$, which effectively provides a linear evolution for $K_r$ from which the PNG multipoint distributions can be recovered through the Fredholm determinant, thus presenting the model as a stochastic integrable system.
The identities in \eqref{eq:lin-intro} follow directly (and for all $r$) from the definition \eqref{eq:PNG-kernel-two-sided-hit}/\eqref{eq:matrixKPNG} if one ignores the dependence of the scattering transform $\J^{h_0}_{x_i-t,x_j+t}$ on the $x_i$'s and $t$.
Showing that this factor does not contribute to $\p_\eta K_r$ and $\p_\zeta K_r$ requires a separate argument, see Sec. \ref{sec:forcing}.

The formula \eqref{fred} for the PNG finite dimensional distributions can be expressed equivalently in terms of the kernel $K^{\uptext{ext}}$ from \eqref{eq:matrixKPNG}, which is how formulas of this type have often been written in the literature.
This way of writing the formula is convenient for some computations.
To state it, for a fixed vector $r\in\zz^n$ and indices $x_1<\dotsm<x_n$  let 
\begin{equation}\label{eq:defChis}
\P_r(x_j,u)=\uno{u>r_j}\qqand\bP_r(x_j,u)=\uno{u\leq r_j},
\end{equation}
which are regarded as multiplication operators acting on the space $\ell^2(\{x_1,\dotsc,x_m\}\times\zz)$. We will also use this notation when $n=1$ and $r\in\zz$, writing $\P_r(u)=1-\bP_r(u)=\uno{u>r}$.
Then we have, for $r_i\geq r_0(t,x_i)$, $i=1,\dotsc,n$,
\begin{equation}\label{fred2} 
F(t, x_1,\ldots, x_n, r_1,\ldots, r_n)= \det(I - \P_rK^{\uptext{ext}}\P_r)_{\ell^2(\{x_1,\dotsc,x_n\}\times\zz}.
\end{equation}

On the other hand, the basic operators which make up the kernels appearing in \eqref{eq:hatKh0ab} and \eqref{eq:PNG-kernel-two-sided-hit} can be expressed as integral operators with kernels having an explicit contour integral formula:
\begin{equation}\label{eq:Scontour}
e^{2t\nabla+x\Delta}(u_1, u_2) = e^{- 2 \fx} \frac{1}{2\pi\I}\oint_{\gamma_0} \frac{\d z}{z^{u_2 - u_1 + 1}} e^{\ft (z - z^{-1}) + \fx (z + z^{-1})},
\end{equation}
where $\gamma_0$ is any simple, positively oriented contour around the origin.
When $\ft > |\fx|$ this kernel can be expressed in terms of the Bessel function of the first kind, $J_n(x) = \frac{1}{2\pi\I}\oint_{\gamma_0} \d z\, e^{x (z - z^{-1}) / 2} / z^{n + 1}$:
\begin{equation}\label{eq:SBessel}
e^{2t\nabla+x\Delta}(u_1,u_2)
= e^{- 2 \fx} \left( \tfrac{\ft - \fx}{\ft + \fx} \right)^{(u_2 - u_1) / 2}\! J_{u_2 - u_1} \bigl(2 \sqrt{\ft^2 - \fx^2}\bigr).
\end{equation}
When $t=0$ there is another Bessel function expression for the kernel of $e^{\fx\Delta}$ for $\fx\geq0$ (i.e., for the semigroup of the walk $\fN$):
\begin{equation}\label{eq:QBessel}
e^{\fx\Delta}(u_1, u_2) = e^{- 2 x} I_{|u_2 - u_1|}(2x),
\end{equation}
where $I_n(2x) = \frac{1}{2\pi\I}\oint_{\gamma_0} \d z\, e^{x (z + z^{-1})} / z^{n + 1}$ is the modified Bessel function of the first kind. 
Formula \eqref{eq:Scontour} can be proved by differentiation in $t$ and $x$, while \eqref{eq:SBessel} and \eqref{eq:QBessel} follow from the standard contour integral formulas for the Bessel functions \cite{NIST:DLMF}.

\subsection{Why hit kernels solve the Kolmogorov equation} \label{sec:whyhit1}

Our proof of Thm. \ref{thm:PNG-fred} will proceed by showing {\em directly} that the function $\FFF$ defined as the Fredholm determinant on the right hand side of \eqref{fred} satisfies the Kolmogorov backward equation.
The proof is transparent but there are some technicalities which tend to obscure the main idea.
Here we explain the main idea, the technicalities are taken care of in Sec. \ref{sec:whyhit}.

Let $\hat\SL$ denote the infinitesimal generator of the PNG height function Markov process.
It consists of a continuous part and a jump part.  The jump part acting on the determinant turns out to produce a rank one perturbation.  This is because the only way the hit kernel changes is by the jump producing a hit where previously there was none. But in this case, we know the spatial position $x$ of the hit and the hit kernel factors into the probability for $\fN$ to go from $u$ to $h_0(x)+1$ and the probability to go from there to $v$, hence a function of $u$ times a function of $v$, i.e.  rank one.  The other part of the generator, as well as $\p_t$, acts as a first order differential operator and therefore from properties of the Fredholm determinant we get
\begin{equation}
(\p_t - \hat\SL)\FFF = \FFF\tr\!\big((I-K_r)^{-1} (\p_t -\hat\SL)K_r\big).
\end{equation}
So the reason that the Fredholm determinant $\FFF$ satisfies the Kolmogorov equation $(\p_t -\hat\SL)\FFF$ is that the kernel itself does!
To show this, i.e. that 
\begin{equation}\label{kol}
(\p_t -\hat\SL)K_r =0,
\end{equation}
note first that from the structure of $K_r$ (see \eqref{eq:hatKh0ab} and \eqref{eq:PNG-kernel-two-sided-hit}), differentiating in $t$ produces twice the commutator of $K_r$ with $\nabla$.
The generator $\hat\SL$, on the other hand, only acts on the $P^{\hit(h)}_{a,b}$ part of the kernel $K_r$.
So the special property of the hit kernel which makes \eqref{kol} hold is that
\begin{equation}\label{com}
\hat{\SL}P^{\hit(h)}_{a,b}(u,v)=2[P^{\hit(h)}_{a,b},\nabla](u,v).
\end{equation}
There is a caveat that this can only hold above the curve, i.e. for $u>h(a)+1$, $v>h(b)+1$, and also that in the computation of the $t$ derivative we may in principle get additional terms coming from differentiating the scattering transform $\J^{h}_{x_i-t,x_j+t}$.
We ignore this for now, guiding the reader to Sec. \ref{sec:whyhit} for the rigorous proof.
The point we want to make is that \eqref{com} is an elementary consequence of skew time reversibility, as we explain next.

Consider a random path $g$ in $[a,b]$ which has the same law as $\fN$ conditioned on $\fN(a) = u$ and $\fN(b) = v$.
We choose $g$ to be lower semi-continuous, and extend it as $\infty$ outside the interval $[a,b]$.
Then, recalling that $\Phi(g)$ (defined in \eqref{eq:def-Phi}) is the indicator of the event that a path $g$ ever becomes less than or equal to $0$, we have
\begin{equation}\label{eq:deltalim2}
 P_{a,b}^{\hit(h)}(u,v) = \fE_{a,u;b,v}[\Phi(g - h )] e^{(b - a)\Delta}(u, v),
\end{equation}
where $\fE_{a,u;b,v}$ denotes the expectation with respect to the law of the random path $g$, and then
\begin{equation}
 \hat{\SL}_h P_{a,b}^{\hit(h)}(u,v) = \fE_{a,u;b,v}[\hat{\SL}_h\Phi(g - h ))] e^{(b - a)\Delta}(u, v)= \fE_{a,u;b,v}[\hat{\SL}^*_g\Phi(g - h ))] e^{(b - a)\Delta}(u, v),
\end{equation}
by skew time reversibility \eqref{infinitesimalstr}.
The subscripts $h$ and $g$ indicate which height function the generators are acting on.
If the expectation were over the invariant measure, the right hand side would now vanish.
So it is not hard to guess that with the pinned measure, one gets boundary terms.
They can be computed, with Prop. \ref{prop:invm-general-adjoint} giving  the right hand side of \eqref{com}.

We feel that the proof of \eqref{fred} using the backward equation has considerable explanatory power.
Although it turns out to be somewhat technical, the key fact that makes it work is transparent, explaining why hit kernels based on the invariant measure produce transition probabilities for such models, and, we hope, opening the way to stochastic analysis of the KPZ fixed point itself.
In \cite{nqr-kolmogorov}, the backward equation was also proved for the TASEP Fredholm determinant formulas.
However, the key mechanism is not apparent there. The identification of this mechanism is one of the most important contributions of this article.
It is also important to note that a proof like this, via the backward equation, can only be achieved once one has a formula for general initial data.

\subsection{KPZ universality}  

The \emph{strong KPZ universality conjecture}  states that for any model in the class there is an analogue of the height function and if $\ep^{1/2} h_0( \ep^{-1}x)\longrightarrow\fh_0(x)$ as $\ep\to0$ then there are (non-universal)
$c$ and $C_\ep$ such that
\begin{equation}\label{eq:123}
\ep^{1/2} h(c\tts\ep^{-3/2} t, \ep^{-1}x)  -  C_\ep t\xrightarrow[\ep\to0]{}\mathfrak{h}(t,x;\fh_0),
\end{equation} 
where $\mathfrak{h}(t,x;\fh_0)$ is
the KPZ fixed point starting from $\mathfrak{h}_0$ \cite{fixedpt}.  This can alternatively be interpreted as the definition of the KPZ class.
 KPZ universality has been achieved in only a few cases, usually by showing that Fredholm determinant formulas for the transition probabilities converge to those of the KPZ fixed point.
There are two steps: 1. Pointwise convergence of the operator kernels; 2. Upgrade to trace class convergence, or analogous bounds.  
The latter tends to be tedious and difficult, and is sometimes skipped in the literature. In the case of PNG, the pointwise convergence of the kernels can be checked by observation.  It will be upgraded to trace class in a future article.
But for PNG, there is anyway a more direct route using the variational formula \eqref{varform}.  Implicit in the article \cite{DV} is the statement that under the 1:2:3 scaling, \eqref{eq:123}, $\mathscr{A}_t(x,y)$ converges to the rescaled Airy sheet $t^{1/3}\mathscr{A}( t^{-2/3} x, t^{-2/3}y)$ and therefore that if one has uniform convergence on compact sets of the rescaled initial data to a compactly supported $\mathfrak{h}_0$, then at time $t>0$ one has 
\[\ep^{1/2} h(\ep^{-3/2} t, \ep^{-1}x)  -  2\ep^{-1} t\xrightarrow[\ep\to0]{}\sup_{y\in\rr}\big\{t^{1/3}\mathscr{A}( t^{-2/3} x, t^{-2/3}y)+\fh_0(y)\}\]
in the same sense.
The right hand side is the variational description of the KPZ fixed point $\fh(t,x;\fh_0)$ (see \cite{nqr-rbm,fixedpt,DOV}).

For the KPZ fixed point $\mathfrak{h}(t,x)$, we showed in \cite{kp} that 
\[\gq=\partial_r^2 \log\pp(\mathfrak{h}(t,x)\le r)\]
satisfies the scalar Kadomtsev--Petviashvili (KP-II) equation
\begin{equation}\label{eq:KP-II}
\partial_t \gq + \tfrac12\partial_r \gq^2 + \tfrac1{12}\partial_r^3 \gq + \tfrac14\partial_r^{-1} \partial_x^2 \gq = 0.
\end{equation}
An equation, with the flavor of \eqref{3h'}/\eqref{eq:Frs-detQr1}, was derived in that case also for the multipoint case.
In order to write it down, we let $\fK=\fK(t;x_1,\dotsc,x_n;r_1,\dotsc,r_n;\fh_0)$ denote the \emph{Brownian scattering operator} introduced in \cite{fixedpt}, which plays the role of the PNG kernel $K_r$ for the KPZ fixed point, so that
\[\FF(t;x_1,\dotsc,x_n;r_1,\dotsc,r_n;\fh_0)\coloneqq\pp(\fh(t,x_i)\leq r_i,\,i=1,\dotsc,n)
=\det(\fI-\fK)_{\oplus_n L^2([0,\infty))}.\]
Consider also the variables $\bar r=r_1+\dotsm+r_n$, $\overline x=x_1+\dotsm+x_n$, and write as before $\p_{\bar r}=\p_{r_1}+\dotsm+\p_{r_n}$ and $\p_{\overline x}=\p_{x_1}+\dotsm+\p_{x_n}$.
Finally let
\[\QQ=(\fI-\fK)^{-1}\fK(0,0),\]
which is in $GL(n)$.
Then the result of \cite{kp} is that
\[\p_r\log\FF=\tr \QQ\]
while $\QQ$ and its derivative  $\gq= \p_r\QQ$ solve the matrix KP equation
\begin{equation}\label{eq:matKP}
\p_t\gq+\tfrac12\tts\p_r\gq^2+\tfrac1{12}\p_r^3\gq+\tfrac14\p_x^2\QQ+\tfrac12[\gq,\tts\p_x\QQ]=0,
\end{equation}
where $[A,B] = AB-BA$.

The following gives a twist on the convergence of PNG to the KPZ fixed point: Instead of convergence of kernels, one can check convergence of equations.
The argument which we will present is far from giving a complete proof, which would require, in addition to controlling the expansions, a uniqueness proof at the level of the KP equations for the KPZ fixed point.
This will be the subject of future research.  
At any rate, the route of \cite{DV} is preferable for a rigorous proof here.  
However, we believe that convergence of equations has considerable explanatory value as a direct and intuitive route to the universal fluctuations.
Well-known behaviours, such as the Tracy-Widom GUE and GOE one-point asymptotic fluctuations in KPZ models, are explained by their appearance as special self-similar solutions of  \eqref{eq:KP-II}, so such a derivation helps us understand the emergence of non-standard fluctuations in the KPZ class.
The idea is that their appearance is through such \emph{normal form} equations.

In the flat case $\fh_0\equiv0$, the $n=1$ $\gq$ in \eqref{eq:KP-II} does not depend on $x$ and the equation simplifies to the Korteweg-de Vries (KdV).
The derivation of the KdV equation from the Toda lattice \eqref{1ode} is a folklore result with a long history.
However, examination of the literature actually reveals a consensus that it does not actually hold unless one reduces the problem, often by considering waves traveling in only one direction.
For example,  \cite{MR891103} states that, ``it cannot be literally true'' that KdV is the limit of the Toda lattice ``since KdV is first order in time and antisymmetric'' (in $r$), but instead has to be derived from first order versions, the Langmuir, or Kac-van Moerbeke lattice.
This approach goes back to considerable work on the problem of convergence to KdV by Toda himself.
In  \cite{doi:10.1143/JPSJ.34.18} it is the Boussinesq equation that is derived from Toda, which is then reduced to KdV by considering only 
one-way waves.

The following informal derivation (in the more general setting of non-Abelian 2D Toda and matrix KP) shows that the misconception was because the unexpected, highly refined 1:2:3 scaling was missed.  It is our hope that it inspires further work.

\subsubsection{Scalar KP equation}

We start with the scalar equation \eqref{eq:waveToda}, where the derivation is easier to follow.
The 1:2:3 rescaling \eqref{eq:123} means we are interested in $\Fep_\ep$ defined by
\begin{equation}\label{one23}
 F(t,x,r) = \Fep_\ep (\ep^{3/2}t,\ep^{}x, \ep^{1/2}(r-c t)).
\end{equation}
For now we let the linear shift in the third argument depend on a constant $c>0$, to be determined.
For the derivation, we have to expand $\Fep_\ep$ in Taylor series.
We do not attempt a justification here.
Furthermore, in the $r$ variable, $\Fep_\ep$ really lives on $\ep^{1/2}\zz$, yet we will pretend it is defined on $\rr$ and still expand in Taylor series.  This type of argument is standard in such derivations, though, again, considerable work would be required to justify it.

Under this scaling, \eqref{eq:waveToda} becomes
\begin{equation}\label{two}
\tfrac14(\ep^3\partial_t^2 -2\ep^2 c\partial^2_{tr} + \ep c^2\partial^2_r -\ep^2\partial_x^2) \log \Fep = \tfrac{\Fep( \cdot +\ep^{1/2})\Fep( \cdot- \ep^{1/2})}{\Fep^2} - 1,
\end{equation}
where for notational simplicity we removed the $\ep$ from the subscript in $\Fep_\ep$.
Expanding \eqref{two} out and multiplying by $\Fep^2$, we get
\begin{multline}
\tfrac{\ep c^2}4\Fep\ts\partial_r^2\Fep- \tfrac{\ep c^2}4(\partial_r \Fep)^2
-\tfrac{\ep^2c}2\Fep\ts\partial_{tr}^2\Fep+ \tfrac{\ep^2c}2\partial_t \Fep\ts\partial_r \Fep-\tfrac{\ep^2}4\Fep\ts\partial_x^2\Fep+ \tfrac{\ep^2}4(\partial_x \Fep)^2
\\
=\ep \Fep\ts\partial_r^2\Fep -\ep (\partial_r \Fep)^2 +\ep^2\big(\tfrac1{12} \Fep\ts\partial_r^4\Fep -\tfrac1{3}\p_r \Fep\ts\partial_r^3\Fep +\tfrac1{4}(\p^2_r \Fep)^2\big)+o(\ep^2).
\end{multline}
Clearly we should choose average growth rate $c=2$ in order for the $O(\ep)$ terms to cancel. We are left with an equation at order $\ep^2$ which says that the associated bilinear form should vanish:
\begin{align}
\Fep\partial_{tr}^2\Fep- \partial_t \Fep\ts\partial_r \Fep+\tfrac{1}4\Fep\ts\partial_x^2\Fep- \tfrac{1}4(\partial_x \Fep)^2
+\tfrac1{12} \Fep\ts\partial_r^4\Fep -\tfrac1{3}\p_r \Fep\ts\partial_r^3\Fep +\tfrac1{4}(\p^2_r \Fep)^2=0.
\end{align}
This is the Hirota form of the scalar KP equation, and a simple computation shows that it means that $\gq=\p_r^2 \log\Fep$ satisfies \eqref{eq:KP-II}.

\subsubsection{Matrix KP equation}

Next we do the same thing with the non-Abelian Toda equations. Start with \eqref{3h'} in the form
\begin{equation}
\p_{\zeta\eta}^2 Q_{r} Q_{r}^{-1} - \p_\eta Q_{r}  Q_{r}^{-1} \p_\zeta Q_{r} Q_{r}^{-1} +  Q_{r} Q_{r-1}^{-1}-  Q_{r+1} Q_{r}^{-1}=0.\label{3ha}
\end{equation}
Since we already learned the average growth rate $c=2$ we write 
\begin{equation}
Q=I + \ep^{1/2} \QQ_\ep(\ep^{3/2}t,\ep^{}x, \ep^{1/2}(r-2 t)).
\end{equation}
Note the difference between $Q=R(1,1)$ in our problem and $\QQ=RK(0,0)$ at the KPZ fixed point level which is reflected in this scaling.
Then \eqref{3ha} becomes $\Omega_1+\Omega_2+\Omega_3=0$ where
\begin{equation}\label{gggg}
\begin{aligned}
\Omega_1=&\tfrac14 \ep^{1/2}( \ep^3\p_t^2-4\ep^2\p^2_{tr} +4\ep\p_r^2-\ep^2\p_x^2)\QQ(r)(I-\ep^{1/2} \QQ(r)+\ep \QQ(r)^2+o(\ep)),\\
\Omega_2= & -\tfrac14 \ep (\ep^{3/2}\p_t -2\ep^{1/2}\p_r + \ep\p_x)\QQ(r)(I-\ep^{1/2} \QQ(r)+\ep \QQ(r)^2+o(\ep))\\
&\qquad\qquad\times (\ep^{3/2}\p_t -2\ep^{1/2}\p_r - \ep\p_x)\QQ(r)(I-\ep^{1/2} \QQ(r)+\ep \QQ(r)^2+o(\ep)),\\
\Omega_3=&(I + \ep^{1/2}\QQ(r))(I+\ep^{1/2} \QQ(r-\ep^{1/2}))^{-1} - (I + \ep^{1/2}\QQ(r+\ep^{1/2}))(I+\ep^{1/2} \QQ(r))^{-1},
\end{aligned}
\end{equation}
and where, as before, we have removed the subscript $\ep$ in $\QQ_\ep$ for convenience; in this decomposition, $\Omega_1$ and $\Omega_2$ come from the first and second terms in \eqref{3ha}, respectively, while $\Omega_3$ combines the last two.
Now writing $\QQ$ for $\QQ(r)$ and $\QQ'$ for $\p_r\QQ(r)$, we have 
\begin{align}
I+\ep^{1/2}\QQ(r+\ep^{1/2})&=I+\ep^{1/2}\QQ+\ep\QQ'+\tfrac12\ep^{3/2}\QQ''+\tfrac16\ep^2\QQ'''+\tfrac1{24}\ep^{5/2}\QQ''''+o(\ep^{5/2}),\\
(I+\ep^{1/2} \QQ(r))^{-1}&=I-\ep^{1/2}\QQ+\ep\QQ^2-\ep^{3/2}\QQ^3+\ep^2\QQ^4-\ep^{5/2}\QQ^5+o(\ep^{5/2}),\\
\shortintertext{and similarly}
(I+\ep^{1/2} \QQ(r-\ep^{1/2}))^{-1}&=I+\textstyle\sum_{i=1}^5(-1)^i\QQ(r-\ep^{1/2})^i+o(\ep^{5/2})\\
&\hspace{-1.45in}=I-\ep^{1/2}\QQ+\ep(\QQ'+\QQ^2)-\ep^{3/2}(\tfrac12\QQ''+\QQ\QQ'+\QQ'\QQ+\QQ^3)\\
&\hspace{-1.4in}+\ep^2(\tfrac16\QQ'''+(\QQ')^2+\tfrac12\QQ\QQ''+\tfrac12\QQ''\QQ+\QQ'\QQ^2+\QQ\QQ'\QQ+\QQ^2\QQ'+\QQ^4)\\
&\hspace{-1.4in}-\ep^{5/2}\Big(\tfrac1{24}\QQ''''+\tfrac16(\QQ'''\QQ+\QQ\QQ''')+\tfrac12(\QQ'\QQ''+\QQ''\QQ')+\tfrac12(\QQ''\QQ^2+\QQ\QQ''\QQ+\QQ^2\QQ'')\\
&\hspace{-1.3in}-((\QQ')^2\QQ+\QQ'\QQ\QQ'+\QQ(\QQ')^2)+(\QQ'\QQ^3+\QQ\QQ'\QQ^2+\QQ^2\QQ'\QQ+\QQ^3\QQ')+\QQ^5\Big)+o(\ep^{5/2}).
\end{align}
Using these expansions in the definition of $\Omega_3$ and collecting terms carefully, one sees that only terms of order $\ep^{3/2}$ and higher remain, and after simplification one gets
\begin{align}\label{hhhh}
&\Omega_3=  -\ep^{3/2}\QQ'' + \ep^2( (\QQ')^2+ \QQ\QQ'')-\ep^{5/2}\Big((\QQ')^2\QQ+\QQ'\QQ\QQ'+\QQ(\QQ')^2 + \QQ''\QQ^2 \\
&\hspace{3.2in} + \tfrac12(\QQ''\QQ'+\QQ''\QQ) + \tfrac1{12} \QQ''''\Big) + o(\ep^{5/2}).
\end{align}
Now we note that the only term of order $\ep^{3/2}$ and lower in the first three lines of \eqref{gggg} is $\p_r^2 \QQ$, which cancels the first term in \eqref{hhhh}.  
We also have cancellation at order $\ep^2$.
Thus the equation appears as the vanishing of terms of order $\ep^{5/2}$.
From \eqref{gggg} we have
\begin{align}
\ep^{5/2}~~\uptext{coeff. of } \Omega_1=& -\p_{tr}^2\QQ-\tfrac14\p_x^2\QQ+\p_r^2\QQ \QQ^2,\\
\ep^{5/2}~~\uptext{coeff. of } \Omega_2=&-\tfrac12 \p_r\QQ\p_x\QQ + \tfrac12\p_x\QQ \p_r\QQ +\p_r\QQ\QQ\p_r\QQ+ (\p_r\QQ)^2\QQ.
\end{align}
Adding them together with the $\ep^{5/2}$ coefficient of $\Omega_3$ we get the matrix KP equation
\begin{equation}
-\p_{tr}^2\QQ-\tfrac14\p_x^2\QQ -\tfrac12 [\p_r\QQ,\p_x\QQ ]
- \tfrac12\p_r^2\QQ\tts\p_r\QQ - \tfrac12\p_r\QQ\tts\p_r^2\QQ - \tfrac1{12} \p_r^4\QQ=0.
\end{equation}

\subsection{Narrow wedge and flat examples, discrete Painlev\'e II}

\subsubsection{Narrow wedge initial data}

Consider PNG with narrow wedge initial condition $\mathfrak{d}_0$, defined in \eqref{eq:def-nw}.
The scaling limit of the height function $h(t,x)$ for fixed $t=1$ was first derived in \cite{prahoferSpohn}, where the limiting Airy$_2$ process was first described, using the transfer matrix method for the Fermi field associated to a multilayer version of the process constructed using dynamics based on the Robinson-Schensted-Knuth (RSK) correspondence.
RSK is also behind the methods employed in \cite{johansson} to study a discrete version of PNG, for which uniform convergence on compacts to the same Airy$_2$ process was 
obtained.

We begin by showing how the formulas obtained earlier in the literature can be recovered from Thm. \ref{thm:PNG-fred}.
Note that for narrow wedge initial data one has $h(t,x)=-\infty$ for all $|x|>t$.
This can be thought of alternatively as restricting the model (and, in particular, the nucleations) to the domain $\{(\ft,\fx)\in[0,\infty)\times\rr\!:|\fx|\leq\ft\}$, which is how it has been often defined in the literature in this case (see e.g. \cite{prahoferSpohn}).
The random walk $\fN$ can only hit $\hypo(\mathfrak{d}_0)$ at the origin, so
$P^{\hit(\mathfrak{d}_0)}_{x_i - t, x_j + t} = \uno{\fx_i - \ft\leq0\leq\fx_j + \ft} e^{(\ft - \fx_i) \Delta} \bP_{0} e^{(\fx_j - t) \Delta}$,
and thus $\J^{\mathfrak{d}_0}_{x_i-t,x_j+t}=\uno{\fx_i - \ft\leq0\leq\fx_j + \ft}\ts\bP_0$ and the kernel in \eqref{eq:PNG-kernel-two-sided-hit} equals
\begin{equation}\label{eq:K-ext-nw}
K^{\ext}(\fx_i, \cdot; \fx_j, \cdot) = - \uno{\fx_i < \fx_j} e^{(\fx_j-\fx_i) \Delta} + \uno{\fx_i - \ft\leq0\leq\fx_j + \ft} e^{- 2 \ft \nabla - \fx_i \Delta} \bP_{0} e^{2 \ft \nabla + \fx_j \Delta}.
\end{equation}
The indicators $\uno{\fx_i\leq\ft}$ and $\uno{\fx_j\geq-\ft}$ are natural because the model is essentially restricted to the domain $|\fx|\leq\ft$.
 In particular, if $\fx_{\ell} > \ft$ for some $\ell$ and the $x_j$'s are ordered, then $\fx_{j} > \ft$ for all $\ell \leq j \leq m$ so $K^{\uptext{ext}}(\fx_i, \cdot; \fx_j, \cdot) = - \uno{\fx_i < \fx_j} e^{(\fx_j-\fx_i)\Delta}$ for $i \geq j \geq \ell$, and thus $K^{\uptext{ext}}$ can be decomposed as 
{\footnotesize$\left[\begin{array}{c|c}
(K^{\uptext{ext}})(\fx_i,\cdot;\fx_j,\cdot)_{1\leq i,j\leq\ell-1} & \star\\\hline
0 & {\sf U}
\end{array}\right]$},
where the value of the $\star$ is irrelevant and ${\sf U}$ is strictly upper triangular.
Thus in this case we have $F(t,x_1,\dotsc,x_n,r_1,\dotsc,r_n)=\det \bigl(I-\P_{r}  K^{\PNG}_{\ft, \ext} \P_{a} \bigr)_{\ell^2(\{\fx_1, \dotsc, \fx_{\ell-1}\} \times \zz)}$, which equals $F(t,x_1,\dotsc,x_\ell,r_1,\dotsc,r_\ell)$.
An analogous statement holds in the case $\fx_\ell<-t$ for some $\ell$.

We focus now on \eqref{eq:K-ext-nw} in the case $|\fx_i|<\ft$ for all $i$ (we throw away the case when some $\fx_i=\pm\ft$ in order to simplify the presentation; at any rate this case is physically uninteresting because by definition of the process we know that $h(\ft,\pm\ft)=0$ almost surely).
Assume first that $\fx_i\geq\fx_j$.
Using \eqref{eq:SBessel}, the kernel \eqref{eq:K-ext-nw} turns into 
\begin{equation}\label{bes1}
K^{\uptext{ext}}(\fx_i, u_i; \fx_j, u_j) 
= e^{2 (\fx_i - \fx_j)} \sum_{\ell \leq 0} \left( \tfrac{\ft + \fx_i}{\ft - \fx_i} \right)^{\frac{u_i - \ell}{2}}\!\! J_{u_i - \ell} \bigl({\scriptstyle{2 \sqrt{\ft^2 - \fx_i^2}}}\bigr) 
\left( \tfrac{\ft - \fx_j}{\ft + \fx_j} \right)^{\frac{u_j - \ell}{ 2} }J_{u_j - \ell} \bigl({\scriptstyle{2 \sqrt{\ft^2 - \fx_j^2}}}\bigr).
\end{equation}
If instead we have $\fx_i < \fx_j$, we write $e^{(\fx_j-\fx_i)\Delta}=e^{-2t\nabla-x_i\Delta}e^{2t\nabla+x_j\Delta}$, and obtain \eqref{bes1} with a minus sign in front.
This shows that $K^{\uptext{ext}}(\fx_j, u_j; \fx_i, u_i)=e^{2(x_i-x_j)}B^*_\ft(u_i, -\fx_i; u_j, -\fx_j)$, where $B_t$ is the \emph{extended discrete Bessel kernel} derived in \cite[Eq.~3.52]{prahoferSpohn}\footnote{After correcting a minor typo in that paper: $\ell+1/2$ inside the Heaviside function in (3.52) there should be $\ell-1/2$.}. 
Having the adjoint of $B_t$ does not change the value of the Fredholm determinant, nor does the conjugation of the kernel by $e^{2(x_i-x_j)}$, while flipping the sign of the $x_i$'s corresponds to reflecting the model about the $x$-axis, which again does not change anything because the dynamics and the initial condition are symmetric.
Hence, after setting $t=1$, this recovers the formula derived in \cite{prahoferSpohn}.

Consider now the one-point case.
Using the above formula we have 
\[K^{\ext}(x,u;x,v)=\left( \tfrac{\ft + \fx}{\ft - \fx} \right)^{(u - v) / 2}B_s(u,v),\qquad B_s(u,v) = \sum_{\ell \leq 0} J_{u - \ell} \bigl(2s\bigr) J_{v - \ell} \bigl(2s\bigr)\]
with $s=\sqrt{t^2-x^2}$.
$B_s$ is the usual \emph{discrete Bessel kernel} \cite{borodinOkounkovOlshanski,johanssonPlancherel}, which is \emph{integrable} in the sense of \cite{IIKS}:
\[B_s(u,v)=\frac{s}{u-v}(J_{u-1}(2s)J_v(2s)-J_u(2s)J_{v-1}(2s))\quad\uptext{for}~ u\neq v\] and $B_s(u,u)=s(\frac{\d}{\d u}J_{u-1}(2s)J_u(2s)-J_{u-1}(2s)\frac{\d}{\d u}J_u(2s))$.
From \cite{borodinOkounkov}, removing the conjugation $((\ft+\fx)/(\ft-\fx))^{(u-v)/2}$ from $K^{\ext}$, we can rewrite 
\begin{equation}\label{eq:Fs-def}
F_r(s) \coloneqq F(t,x,r)=\det (I-\P_rB_s\P_r)_{\ell^2(\zz)}=e^{-s^2}D_{r-1}(\varphi_s),
\end{equation}
where $D_n(\varphi_s)$ is the $n\times n$ Toeplitz determinant associated with the weight $\varphi_s(\zeta)=e^{s(\zeta+\zeta^{-1})}$ (setting $D_{-1}(\varphi_s)=1$); more explicitly, we have (for $r\geq1$)
\begin{equation}
F_r(s) = e^{-s^2}\det(I_{i-j}(2s))_{i,j=0,\ldots, r-1},
\end{equation}
where the $I_n$ are modified Bessel functions of the first kind (see \eqref{eq:QBessel}).
This formula goes back to \cite{gessel-P-rec}, and was the starting point of the analysis in \cite{baikDeiftJohansson}.

The Toeplitz determinants $D_n(\varphi_s)$ are intimately related to the orthogonal polynomials on the unit circle associated with the weight $\varphi_s$.
In the Szeg\H{o} three term recurrence for such orthogonal polynomials on the circle with respect to a general weight \cite{szego-book}, only one coefficient $\alpha_n$ depends on the choice of weight.
Thus the $\alpha_n$'s, known as the \emph{Verblunsky coefficients}, ``encode'' the weight, and can be thought of as the analogue of reflection coefficients in this context.
In the case of our weight $\varphi_s$ they are real, and related to $F_r$ through
\begin{equation}
\alpha_r^2 = 1- F_rF_{r+2}/{F_{r-1}^2}.
\end{equation}
They are known (see e.g. \cite[Sec.~3]{vanAssche-book}) to satisfy the discrete Painlev\'e II equation
\begin{equation}\label{eq:dpII}
-s(1-\alpha_r^2) (\alpha_{r+1}+\alpha_{r-1}) = (r+1)\alpha_r
\end{equation}
(see also \cite{borodin-discretegap}) as well as the Ablowitz-Ladik lattice
\begin{equation}\label{allatt}
\p_s\alpha_r=(1-\alpha_r^2)(\alpha_{r+1}-\alpha_{r-1})
\end{equation}
(see also \cite{MR1794352}).
The Ablowitz-Ladik hierarchy is a reduction of the 2D Toda hierarchy \cite{MR3803581}.  However, it is not clear if \eqref{allatt} can be obtained directly from our equation with the given symmetry,
\begin{equation}\label{rtdt}
\tfrac14 (\p_s^2+ \tfrac1{s}\p_s) \log F_r = \tfrac{F_{r+1}F_{r-1}}{F_r^2} -1.
\end{equation}
We mention also that in Thm.~7.12 of \cite{BDS} it is shown, using the Riemann-Hilbert problem for the Toeplitz determinant $D_n(\varphi_s)$, that the functions $g_{m} = - \log (\tfrac{F_{-2m+1}}{F_{-2m}} (1 - \alpha_{-2m - 1}))$, defined for $m \leq 0$, satisfy the Toda lattice \eqref{1ode}.  

\subsubsection{Flat initial data} \label{flat1}

Consider now PNG with flat initial condition $h_0\equiv0$.
For fixed $t=1$, this case was analyzed in \cite{borodFerSas}.
Clearly it is enough to consider $r_i\geq0$ for each $i$.
In view of \eqref{fred2}, this means that we are interested in the kernel evaluated at values of $u_i>0$ for each $i$.
In that region, the scattering kernel $e^{- 2 \ft \nabla - \fx_i \Delta} \J^{0}_{\fx_i - \ft, \fx_j + \ft} e^{2 \ft \nabla + \fx_j \Delta} = e^{- 2 \ft \nabla + t \Delta} P^{\hit(0)}_{x_i - t, x_j + t} e^{2 \ft \nabla - t \Delta}$ equals
\begin{equation}\label{eq:flatex}
 \uno{2 \ft \geq \fx_i - \fx_j} \big(e^{- 2 \ft \nabla + t \Delta} \bP_0 e^{2 \ft \nabla + (x_j - x_i + t) \Delta} + e^{- 2 \ft \nabla + t \Delta} \P_0 P^{\hit(0)}_{x_i - t, x_j + t} e^{2 \ft \nabla - t \Delta}\big).
\end{equation}
A simple computation using \eqref{eq:Scontour} shows that the first term inside the parenthesis can be represented by the contour integral
$\frac{e^{2 (\fx_i - \fx_j)}}{(2\pi\I)^2}\oint_{\ts\gamma_0} \d z \oint_{\ts\gamma_0}\tsm \d w \frac{z^{u_i-1}}{w^{u_j} (w-z)} e^{2 \ft (w - z) + (\fx_j - \fx_i) (w + w^{-1})}$
with the contours chosen so that $|w|>|z|$.  But in the case of interest, $u_i>0$,  the integrand is analytic in $z$  and hence this term vanishes.
 The second term on the right hand side of \eqref{eq:flatex} is computed using the reflection principle for $\fN$: assuming $x_j-x_i+2t\geq0$ and $u>0$, and writing $\ux=x_i-t$, $\ox=x_j+t$,
\[P^{\hit(0)}_{\ux,\ox}(u,v)=e^{(\ox-\ux)\Delta}(u,-v)\uno{v>0}+e^{(\ox-\ux)\Delta}(u,v)\uno{v\leq0}.\]
$P^{\hit(0)}_{\ux,\ox}e^{2 \ft \nabla - t \Delta}(u,v)$ equals the sum of two integrals, 
$\frac{e^{2(\ox-\ux)}}{(2\pi\I)^2}\ttsm\oint_{\ts\gamma_0}\tsm\d z\tsm \oint_{\ts\gamma_0'}\tsm \d w\, \frac{z^{u}}{w^{v} (1-wz)} e^{(\ox-\ux)(z+\frac1z)-\frac{2t}w}$ with contours chosen so that $|wz|<1$, and 
$\frac{e^{2(\ox-\ux)}}{(2\pi\I)^2}\oint_{\ts\gamma_0} \d z \oint_{\ts\gamma_0}'\tsm \d w \frac{z^{u-1}}{w^{v} (w-z)} e^{(\ox-\ux)(z+\frac1z)-\frac{2t}w}$
with contours chosen so that  $|w|>|z|$.
The two integrals can be brought together if we choose contours with $|z| < 1$ and $|z| < |w| < 1/|z|$, and simplifying we get
\begin{equation}\label{flatcomp}
\textstyle P^{\hit(0)}_{\ux,\ox} e^{2 \ft \nabla - t \Delta} (u, v) 
=\frac{e^{2(\ox-\ux)}}{(2\pi\I)^2}\oint_{\ts\gamma_0} \d z \, \frac{1 - z^2}{z^{1 - u}} e^{(\ox-\ux) (z + z^{-1})}\oint_{\ts\gamma_0'}\tsm \d w\, \frac1{w^{v} (w - z)(1 - zw)} e^{- 2 \ft w^{-1}}.
\end{equation}
Expanding the $w$ contour to a contour $\gamma_0''$ with $|w|>1/|z|$ and picking up the residue at $w=z^{-1}$, the $w$ integral can be written as $\frac{1}{2 \pi\I} \oint_{\gamma_0''} \frac{\d w}{w^{v} (w - z)(1 - zw)} e^{- 2 \ft w^{-1}} - \frac{z^{v}}{z^2 - 1} e^{- 2 \ft z}$, and changing variables $w\longmapsto1/w$, this expression becomes
$ \frac{1}{2 \pi\I} \oint_{\gamma_0'''}\d w\, \frac{w^v}{(w - z)(1 - zw)} e^{- 2 \ft w} - \frac{z^{v}}{z^2 - 1} e^{- 2 \ft z}$,
where the contour includes only the singularity at $0$. 
We are interested only in the case $v>0$, for which the $w$ integral vanishes, and thus using this for such $v$ in the above expression we get that \eqref{flatcomp} is given by 
$ e^{2 (\ox-\ux)} \frac{1}{2 \pi\I}\oint_{\gamma_0} \frac{\d z}{z^{u + v + 1}} e^{2\ft z + (\fx_j - \fx_i) (z + z^{-1})}$,
where in the last integral we changed $z$ to $z^{-1}$.
Continuing the computation in the same way we get, for $u_i,u_j>0$, $e^{- 2 \ft \nabla + t \Delta} \P_0 P^{\hit(0)}_{\ux,\ox} e^{2 \ft \nabla - t \Delta}(u_i, u_j)=e^{2 (\fx_i - \fx_j)} \frac{1}{2\pi\I}\oint_{\gamma_{0}} \frac{\d w}{w^{u_i}} \frac{1}{2\pi\I}\oint_{\gamma_0} \frac{\d z}{z^{u_j + 1} (z-w)} e^{2 \ft (z - w^{-1}) + (\fx_j - \fx_i) (z + z^{-1})}$ with contours chosen so that $|z|>|w|$.
Now we enlarge the $w$ contour to a circle of radius $r$, $r>|z|$.
With this new contour the integral is clearly bounded by a constant times $r\tts e^{-2t/r}/r^{u_i+1}$, which goes to $0$ as $r\to\infty$.
Hence we are only left with the residue at $w=z$, which finally yields, for $u_i,u_j>0$,
\begin{align}
  \textstyle K^{\uptext{ext}}(\fx_i, u_i; \fx_j, u_j) &\textstyle = - e^{2 (\fx_i - \fx_j)} I_{|u_i - u_j|}(2(\fx_j-\fx_i)) \uno{\fx_i < \fx_j} \\
  &\textstyle \qquad + \uno{2 \ft \geq \fx_i - \fx_j} e^{2 (\fx_i - \fx_j)} \frac{1}{2 \pi\I}\oint_{\gamma} \frac{\d w}{w^{u_i + u_j + 1}} e^{(\fx_j - \fx_i) (w + w^{-1})} e^{2\ft (w - w^{-1})}.
\end{align}
When $x_j - x_i \leq 2t$, the kernel coincides with the one obtained in \cite[Prop.~4]{borodFerSas}.
When $x_j-x_i<-2t$ this does not coincide exactly with the expression in that result, which has an absolute value in the indicator function, but in terms of the determinant it does not make any difference.

In the one-point case the flat kernel is independent of the spatial variable $x$, and simplifies to 
\[\textstyle K^\uptext{ext}(u,v)=\frac{1}{2 \pi\I}\oint_{\gamma} \frac{\d w}{w^{u+v + 1}} e^{2\ft (w - w^{-1})}=J_{u+v}(4\ft).\]
The kernel is Hankel, and in fact the associated Fredholm determinant is known to be equal to a Toeplitz-plus-Hankel determinant \cite{baikRainsDuke1},
\begin{equation}\label{eq:Fs-flat}
F_r(s) \coloneqq F(s,0,r) = e^{-s^2}\det(I_{i-j}(2s) - I_{i+j+2}(2s))_{i,j=0,\ldots, r-1}
\end{equation}
where $I_n$ are modified Bessel functions of the first kind (see \eqref{eq:QBessel}).  By analogy with the narrow wedge case, one expects a connection with discrete Painlev\'e.
Such a connection has appeared implicitly in the literature and we include it here for completeness; we thank Jinho Baik for explaining it to us.

From Eqn.~4.18 in \cite{baikRainsDuke2}, the determinant on the right hand side of \eqref{eq:Fs-flat} can be written as $\prod_{j \geq r} N_{2j+2}(s)^{-1} (1 - \pi_{2j+2}(0; s))$ where $\pi_k(\zeta; s)$ is the monic orthogonal polynomial of degree $k$ on the unit circle with respect to the weight $\varphi_s(\zeta)$ defined below \eqref{eq:Fs-def} and $N_k(s)$ is its norm, defined in \cite[Eq.~4.10]{baikRainsDuke2}.
Writing $b_k(s) = -\pi_k(0; s)$ and $a_k(s) = -N_{k-1}(s)^{-1}$, we get from this that
\[
{F_r(s)}/{F_{r+1}(s)} = -a_{2r + 3}(s) (1 + b_{2r + 2}(s)).
\]
The relation to Painlev\'e is provided by Eqn.~1.6 and Thm.~5.1 of \cite{Baik-RH}, which show that the $b_k$'s satisfy the discrete Painlev\'e II equation \eqref{eq:dpII} with $\alpha_r=b_{r-1}$ while the $a_k$'s can be recovered from the $b_k$'s through
$a_k = (1-b_k^2) a_{k+1}$.
Using arguments similar to those in \cite[Sec. 7.3]{BDS}, it should also be possible to derive the Toda lattice \eqref{1ode} for flat PNG from these formulas.

\section{Derivation of the non-Abelian Toda equation}\label{sec:nonAbelianToda}

Our goal in this section is to prove Thm. \ref{thm:nonAbelian}, using the result of Thm.~\ref{thm:PNG-fred}. 
The proof is divided in three parts.
In Sec. \ref{sec:genkerres}, we abstract away details particular to PNG and prove the non-Abelian Toda equations for a general class of matrix kernels $K_r$ satisfying certain structural conditions.
In Sec. \ref{sec:pfMain} we explain how the general result is used to derive the non-Abelian Toda equations for a simplified version of the PNG kernel where the dependence on the interval $[x_i-t,x_j+t]$ in the scattering transform $\J^{h_0}_{x_i-t,x_j+t}$ in \eqref{eq:PNG-kernel-two-sided-hit} is dropped.
Finally in Sec. \ref{sec:forcing} we show that dropping this dependence does not matter in the region above $r_0(t,x)$.

\subsection{General kernel}\label{sec:genkerres}

\begin{thm}\label{thm:nonAbelianToda-general}
Fix two integers $\munderbar r<\bar r-1$ and some open domain $O\subseteq\rr^2$.
Let $K_{\eta,\zeta,r}$ be a family of kernels defined for $r\in\{\munderbar r, \munderbar r+1,\dotsc,\bar r\}$ and $(\eta,\zeta)\in O$, which acts on $\oplus_n\ell^2(\zz_{>0})$, the $n$-fold direct sum of $\ell^2(\zz_{>0})$. 
Write $K_r=K_{\eta,\zeta,r}$ and suppose that its matrix kernel $K_r(u,v)$ satisfies, for all $u,v>0$,
\begin{enumerate}[label=\uptext{(\arabic*)}]
\item $K_{r+1}(u,v)=K_{r}(u+1,v+1)$,~~ $r\in\{\munderbar r,\dotsc,\bar r-1\}$,
\item $\partial_\eta K_r(u,v) = K_{r-1}(u+1,v) - K_r(u + 1,v)$, ~~$r\in\{\munderbar r+1,\dotsc,\bar r\}$,
\item $\partial_\zeta K_r(u,v) = K_{r-1}(u,v+1)-K_r(u,v + 1)$, ~~$r\in\{\munderbar r+1,\dotsc,\bar r\}$.
\end{enumerate}
Assume that, for all $r\in\{\munderbar r,\munderbar r+1,\dotsc,\bar r\}$ and fixed $(\eta,\zeta)\in O$, $K_r$ is trace class as an operator acting on $\oplus_n\ell^2(\zz_{>0})$, and that for such $r$ the derivatives in (2) and (3) above exist in trace norm.
Assume furthermore that $I-K_r$ is invertible for all $r\in\{\munderbar r,\dotsc,\bar r\}$ and define the invertible $n\times n$ matrix
\begin{align}
Q_r=(I-K_r)^{-1}(1,1).
\end{align}
Then for $r\in\{\munderbar r+1,\dotsc,\bar r-1\}$ and $(\eta,\zeta)\in O$, $Q_r$
satisfies the non-Abelian 2D Toda equation
\begin{equation}
\p_{\zeta}(\p_\eta Q_{r} Q_{r}^{-1})+  Q_{r} Q_{r-1}^{-1}-  Q_{r+1} Q_{r}^{-1}=0.\label{eq:2-2}
\end{equation}
\end{thm}

We stress that, although we used the same notation, $K_r$ and $Q_r$ in this result are general and not necessarily related to PNG (in particular, the index $r$ in these objects denotes an integer rather than an element of $\zz^n$).

Before starting the proof, we need to introduce some notations.
For $r\in\{\munderbar r,\munderbar r+1,\dotsc,\bar r\}$ let
\[R_r= (I-K_r)^{-1},\qquad\uptext{so}\quad Q_r=R_r(1,1),\]
and define, for $f\in\ell^2(\zz_{>0})$,
\[\sigma f(u)=f(u+1),\qquad\sigma^*f(u)=f(u-1)\uno{u>1},\]
which have explicit kernels given by $\sigma(u, v) = \uno{v = u+1}$ and $\sigma^*(u, v) = \uno{v = u-1>0}$.

\begin{rem}\label{rem:bp}
In \cite{MR891103} it is shown, in this language, that if one has a scalar kernel $K_r(t)$ satisfying $K_{r+1}= \sigma K_r = K_r \sigma^*$ and $\p_tK_r= K_{r+1}-K_{r-1}$, then $ p_r=\det(I+K_r)$ satisfies the one-sided bilinear Toda lattice equation $\dot{p}_{r+1} p_r- p_{r+1}\dot{p}_r = p_{r+2}p_{r-1}-p_{r+1}p_r$, which can be transformed into the classic Toda lattice \eqref{1ode} (see the introduction of \cite{MR891103}).
\end{rem}
We will use the same notation $\sigma$ for the diagonal matrix kernel with $\sigma$ in each diagonal entry, which acts on $f\in\oplus_n\ell^2(\zz_{>0})$ as above.
Then (1) in the assumptions of the theorem becomes
\begin{equation}\label{eq:sigmaKsigma}
 K_{r+1}=\sigma K_r\sigma^*,
\end{equation}
while 
\[\sigma\sigma^*=I\qqand P\coloneqq I-\sigma^*\sigma\quad\uptext{is rank $n$}.\]
$P$ is a diagonal matrix kernel with diagonal entries $\Pi$ which have an explicit kernel given by
\begin{equation}\label{eq:Pi}
\Pi (u, v) = \uno{u = v = 1}.
\end{equation}
On the other hand, (2) and (3) in the assumptions of the theorem can be expressed in terms of $\sigma$ and $\sigma^*$ as
\begin{equation}
\p_\eta K = \sigma(K_{r-1}-K_r),\qqand\p_\zeta K_r = (K_{r-1}-K_r)\sigma^*.\label{eq:petazeta1}
\end{equation}
By taking the $\eta$ derivative of the first or the $\zeta$ derivative of the second, we also have
\begin{equation}
\p_{\eta}\p_{\zeta}K_r = K_{r+1}-2K_r+ K_{r-1}.\label{eq:petazeta2}
\end{equation}

The proof of Thm. \ref{thm:nonAbelianToda-general} is based on two identities which we collect in the next result:

\begin{lem}\label{lemmafortoda}
\leavevmode
\begin{enumerate}[label=\uptext{(\roman*)}]
\item $(I+PR_rK_rP)^{-1}=(I-P)+(I-K_r)(I+\sigma^*R_{r+1}\sigma K_r)P$.
\item $(I+PR_rK_rP)(I+PR_{r-1}K_{r-1}P)^{-1}= I+ PR_r(K_r-K_{r-1})(I+\sigma^*R_r\sigma K_{r-1})P$.
\end{enumerate}
\end{lem}

\begin{proof}
The first step of the proof is to show the following two simple identities:
\begin{align}
(I+PR_rK_r)^{-1}&=(I-K_r)(I+\sigma^*R_{r+1}\sigma K_r),\label{eq:lv1}\\
(I+PR_rK_rP)^{-1}&=(I-P)+(I+PR_rK_r)^{-1}P.\label{eq:lv2}
\end{align}
For the first one, use the fact that $R_rK_r=K_rR_r=R_r-I$
to write $I+PR_rK_r=I-P+PR_r=(I-\sigma^*\sigma K_r)R_r$, and then multiply by $(I-K_r)(I+\sigma^*R_{r+1}\sigma K_r)$ to get
\begin{align}
&(I+PR_rK_r)(I-K_r)(I+\sigma^*R_{r+1}\sigma K_r)=(I-\sigma^*\sigma K_r)(I+\sigma^*R_{r+1}\sigma K_r)\\
&\hspace{0.8in}=I-\sigma^*\sigma K_r+\sigma^*R_{r+1}\sigma K_r-\sigma^*\sigma K_r\sigma^*R_{r+1}\sigma K_r=I,
\end{align}
where in the last equality we used $\sigma K_r\sigma^*=K_{r+1}$ together with $R_rK_r=K_rR_r=R_r-I$ again.

To get \eqref{eq:lv2}, decompose $I+PR_rK_rP$ as $(I+PR_rK_r)-PR_rK_r(I-P)$ and then multiply by $(I-P)+(I+PR_rK_r)^{-1}P$ on the right.
The result can be simplified using the fact that $P$ is a projection and we get
\[I-PR_rK_r(I-P)(I+PR_rK_r)^{-1}P=I,\]
the last equality following from the easily checked fact that $(I+PR_rK_r)^{-1}P=P(I+R_rK_rP)^{-1}$.

The identity in (i) now follows directly from \eqref{eq:lv1} and \eqref{eq:lv2}.
For (ii), write the left hand side as 
\[I+P(R_rK_r-R_{r-1}K_{r-1})P(I+PR_{r-1}K_{r-1}P)^{-1}\] 
and then use (i) to rewrite this as
\begin{equation}
I+P(R_rK_r-R_{r-1}K_{r-1})P(I-K_{r-1})(I+\sigma^*R_r\sigma K_{r-1})P,\label{eq:ccc}
\end{equation}
where we used again the fact that $P$ is a projection.
Now use \eqref{eq:sigmaKsigma} to write
\begin{align}
(I-P)(I-K_{r-1})(I+\sigma^*R_r\sigma K_{r-1})P&=\sigma^*\sigma(I-K_{r-1})(I+\sigma^*R_r\sigma K_{r-1})P\\
&=(\sigma^*\sigma(I-K_{r-1})+\sigma^*\sigma K_{r-1})P=\sigma^*\sigma P=0,
\end{align}
and use this to rewrite \eqref{eq:ccc} further as 
\begin{equation}
 I+P(R_rK_r-R_{r-1}K_{r-1})(I-K_{r-1})(I+\sigma^*R_r\sigma K_{r-1})P.\label{eq:cccc}
\end{equation}
Since $R_{r-1}K_{r-1}=K_{r-1}R_{r-1}$ and $R_rK_r=R_r-I$ we have 
\[(R_rK_r-R_{r-1}K_{r-1})(I-K_{r-1})=R_rK_r(I-K_{r-1})-K_{r-1}=R_r(K_r-K_{r-1})\]
and using this in \eqref{eq:cccc} shows that it is equal to the desired right hand side.
\end{proof}

\begin{proof}[Proof of Thm. \ref{thm:nonAbelianToda-general}] 
Let
\begin{equation}
U_r=Q_rQ_{r-1}^{-1},\qquad
V_r=-\p_\eta Q_{r}Q_{r}^{-1}.
\end{equation}
Then \eqref{eq:2-2} can be rewritten as
\begin{equation}
\p_\zeta V_r +U_{r+1}- U_{r}=0.\label{2}
\end{equation}
Our goal then will be to prove this last equation.

In order to achieve this, it is convenient to lift \eqref{2} to an equivalent equation for operators acting on $\oplus_n\ell^2(\zz_{>0})$.
To this end, we introduce the operators
\[\cX_r=PR_rK_rP,\qquad\cY_r=I+\cX_r=I+PR_rK_rP.\]
From the previous lemma, $\cY_{r}$ is invertible.
On the other hand, since $K_r$ is differentiable in $\eta$ and $\zeta$ in trace norm, its inverse $R_r$ is also differentiable, and since $P$ is finite rank then $\cX_r$, and thus also $\cY_r$ and $\cY_r^{-1}$, are differentiable in $\eta$ and $\zeta$ in trace norm.
In particular, we may introduce the following analogs of $U_r$ and $V_r$ in terms of the operator $\cY_r$:
\begin{align}
\mathcal{U}_r=\cY_r\cY_{r-1}^{-1},\qquad
\mathcal{V}_r&=-\p_\eta \cY_{r}\cY_{r}^{-1}.
\end{align}
We claim that \eqref{2} is equivalent to the same equation at the level of $\mathcal{U}_r$ and $\mathcal{V}_r$:
\begin{equation}
\p_\zeta \mathcal{V}_r +\mathcal{U}_{r+1}- \mathcal{U}_{r} =0.\label{2a}
\end{equation}
To see this note first that, in view of \eqref{eq:Pi}, $P$ is a projection onto an $n$-dimensional subspace of $\oplus_n\ell^2(\zz_{>0})$, call it $Z$.
Then we may think of $\cY_r$ as acting on $Z\oplus Z^\perp$ as a block matrix operator
$\cY_r=\left[\begin{smallmatrix}
I_Z+\cX_r & 0 \\
0 & I_{Z^\perp}
\end{smallmatrix}\right]$,
and we have $\cY_r^{-1}=\left[\begin{smallmatrix}(I_Z+\cX_r)^{-1}&0\\0&I_{Z^\perp}\end{smallmatrix}\right]$.
In this way we thus have
\[\cU_r=\begin{bmatrix}
(I_Z+\cX_r)(I_Z+\cX_{r-1})^{-1} & 0 \\
0 & I_{Z^\perp}\end{bmatrix},\qquad
\cV_r=\begin{bmatrix}
\p_\eta\cX_r(I_Z+\cX_{r-1})^{-1} & 0 \\
0 & I_{Z^\perp}\end{bmatrix}.\]
In particular, \eqref{2a} is completely trivial for the second row in this block representation.
For the first row we simply note that $Z$ is isomorphic to $\rr^n$ and the matrix of the linear transformation $I_Z+\cX_r$ on $\rr^n$ is $Q_r$, so the equation is equivalent to \eqref{2}.

Hence our goal is to show \eqref{2a}.
Using $R_rK_r=K_rR_r=R_r-I$ we have
\begin{gather}
\p_\eta\tts\cY_r=P\p_\eta R_rP=PR_r\p_\eta K_rR_rP,\qquad
\p_\zeta\tts\cY_r=P\p_\zeta R_rP=PR_r\p_\zeta K_rR_rP,\\
\p_\zeta\p_\eta\tts\cY_r=P\p_\zeta\p_\eta R_rP=PR_r(\p_\zeta\p_\eta K_r+\p_\zeta K_rR_r\p_\eta K_r+\p_\eta K_rR_r\p_\zeta K_r)R_rP,
\end{gather}
so
\begin{align}
\p_\zeta\cV_r&=-\p_\zeta\p_\eta \cY_{r}\cY_{r}^{-1}+\p_\eta \cY_{r}\cY_{r}^{-1}\p_\zeta \cY_{r}\cY_{r}^{-1}\label{eq:pzetaVrp}\\
&=-PR_{r}\Big[\p_\zeta\p_\eta K_{r} + \p_\zeta K_{r}R_{r}\p_\eta K_{r}+\p_\eta K_{r}R_{r}\p_\zeta K_{r}-\p_\eta K_{r}R_{r}P\cY_{r}^{-1}PR_{r}\p_\zeta K_{r}\Big]R_{r}P\cY_{r}^{-1}.
\end{align}
Now from (i) of Lem. \ref{lemmafortoda}  we have $R_{r}P\cY_{r}^{-1}=R_{r}P(I-K_{r})(I+\sigma^*R_{r+1}\sigma K_{r})P$, while
\begin{align}
R_{r}P(I-K_{r})(I+\sigma^*R_{r+1}\sigma K_{r})
&=I+\sigma^*R_{r+1}\sigma K_{r}-R_{r}\sigma^*\sigma(I-K_{r})(I+\sigma^*R_{r+1}\sigma K_{r})\\
&\hspace{-40pt}=I+\sigma^*R_{r+1}\sigma K_{r}-R_{r}(\sigma^*\sigma(I-K_{r})+\sigma^*(I-K_{r+1})R_{r+1}\sigma K_{r})\\
&\hspace{-40pt}=I+\sigma^*R_{r+1}\sigma K_{r}-R_{r}\sigma^*\sigma=I+\sigma^*R_{r+1}\sigma K_{r}-R_{r}(I-P),
\end{align}
where in the second equality we simplified using \eqref{eq:sigmaKsigma}, and therefore $R_{r}P\cY_{r}^{-1}=(I+\sigma^*R_{r+1}\sigma K_{r})P$.
Using this in \eqref{eq:pzetaVrp} gives
\begin{equation}
\p_\zeta\cV_r = -PR_r\Omega(I+\sigma^*R_{r+1}\sigma K_r)P,\label{eq:pzetaVr}
\end{equation}
where $\Omega$ equals (using again the formula for $R_{r}P\cY_{r}^{-1}$)
\begin{align}
& \p_\zeta\p_\eta K_r + \p_\zeta K_rR_r\p_\eta K_r+\p_\eta K_rR_r\p_\zeta K_r-\p_\eta K_rR_rP\cY_r^{-1}PR_r\p_\zeta K_r\label{eq:crbr}\\
& \;\; =  (K_{r+1}-2K_r+K_{r-1})+(K_r- K_{r-1})\sigma^*R_r\sigma(K_r- K_{r-1})\\
& \quad + \sigma(K_r- K_{r-1})R_r(K_r- K_{r-1})\sigma^*-\sigma(K_r- K_{r-1})(I+\sigma^*R_{r+1}\sigma K_r)PR_r(K_r- K_{r-1})\sigma^*
\end{align}
\begin{align}
&\;\; =  (K_{r+1}-2K_{r}+K_{r-1})+(K_{r}- K_{r-1})\sigma^*R_{r}\sigma(K_{r}- K_{r-1})\\
&\quad -\sigma(K_{r}- K_{r-1})\sigma^*R_{r+1}\sigma K_{r}PR_{r}(K_{r}- K_{r-1})\sigma^*+\sigma(K_{r}- K_{r-1})\sigma^*\sigma R_{r}(K_{r}- K_{r-1})\sigma^*\\
& \;\;= (K_{r+1}\!-\!2K_{r}\!+\!K_{r-1})+(K_{r}\!-\! K_{r-1})\sigma^*R_{r}\sigma(K_{r}\!-\! K_{r-1}) + (K_{r+1}\!-\!K_{r})R_{r+1}(K_{r+1}\!-\!K_{r}),
\end{align}
and where we used \eqref{eq:sigmaKsigma} again together with the identity $R_{r+1}\sigma K_{r}P-\sigma=R_{r+1}\sigma K_{r}(I-\sigma^*\sigma)-\sigma=-R_{r+1}\sigma(I-K_{r})$.
On the other hand, (ii) of Lem. \ref{lemmafortoda} gives
\[\cU_{r+1} -\cU_{r} = PR_{r+1}(K_{r+1}-K_{r})(I+\sigma^*R_{r+1}\sigma K_{r})P-PR_{r}(K_{r}-K_{r-1})(I+\sigma^*R_{r}\sigma K_{r-1})P,\]
and using $R_r(K_{r+1}-K_{r})R_{r+1}=R_{r+1}-R_r$ we have $R_{r+1}(K_{r+1}-K_{r})=R_{r}(K_{r+1}-K_{r})(I+R_{r+1}(K_{r+1}-K_{r}))$.
Then 
\begin{multline}
\cU_{r+1} -\cU_{r} = PR_{r}(K_{r+1}-K_{r})(I+\sigma^*R_{r+1}\sigma K_{r})P
PR_{r}(K_{r}-K_{r-1})(I+\sigma^*R_{r}\sigma K_{r-1})P
\\
\qquad + PR_{r}(K_{r+1}-K_{r})R_{r+1} (K_{r+1}-K_{r})(I+\sigma^*R_{r+1}\sigma K_{r})P.
\end{multline}
Comparing with \eqref{eq:pzetaVr}, the last term on the right hand cancels with the piece coming from the third term in \eqref{eq:crbr}, 
so we get
\begin{align}
\p_\zeta\cV_r + \cU_{r+1} -\cU_{r}
&= PR_{r}(K_{r}-K_{r-1})(I+\sigma^*R_{r+1}\sigma K_{r})P\\
&\quad\qquad-PR_{r}(K_{r}-K_{r-1})\sigma^*R_{r}\sigma(K_{r}-K_{r-1})(I+\sigma^*R_{r+1}\sigma K_{r})P\\
&\quad\qquad-PR_{r}(K_{r}-K_{r-1})(I+\sigma^*R_{r}\sigma K_{r-1})P\\
&= PR_{r}(K_{r}-K_{r-1})\sigma^*(R_{r+1}\sigma K_{r}-R_{r}\sigma K_{r-1} )P
\\& \quad\qquad-PR_{r}(K_{r}- K_{r-1})\sigma^*R_{r}\sigma(K_{r}- K_{r-1})(I+\sigma^*R_{r+1}\sigma K_{r})P
\\&= PR_{r}(K_{r}-K_{r-1}) \sigma^*[R_{r+1}   -  R_{r} 
-  R_{r}K_{r+1}R_{r+1}  +R_{r}K_{r}R_{r+1} ]\sigma K_{r}P.
\end{align}
The term in brackets vanishes as $R_rK_r=K_rR_r=R_r-I$, yielding \eqref{2a} and completing the proof.
\end{proof}

\subsection{PNG kernel}\label{sec:pfMain}

Recall our definition $(K_r)_{ij}(u,v)=K^{\uptext{ext}}(u+r_i,v+r_j)$.
In Sec. \ref{sec:main} we used $r$ to denote the vector $(r_1,\dotsc,r_n)\in\zz^n$.
Here it will be more convenient to regard $r_1,\dotsc,r_n$ as fixed parameters and $r$ as an auxiliary variable taking values in $\zz$, and redefine
\begin{equation}
(K_r)_{ij}(u,v)=K^{\uptext{ext}}(u+r_i+r,v+r_j+r).\label{eq:KrKext}
\end{equation}
This notation is consistent with our usage of $K_{r\pm1}$ in Sec. \ref{sec:main}, and then \eqref{3h'} coincides with \eqref{2} with $r=0$.
Therefore, to show that \eqref{3h'} holds at the given choice of parameters $(r_1,\dotsc,r_n)$ satisfying $r_i>r_0(t,x_i)$ for each $i$, it is enough to show that $K_r$ satisfies the assumptions of Thm. \ref{thm:nonAbelianToda-general} with $\munderbar r=-1$ and $\bar{r}=1$.

That $K_r$ satisfies (1) of Thm. \ref{thm:nonAbelianToda-general} is straightforward by definition, so we turn to assumptions (2) and (3) of that result.
We fix $i$ and $j$ and focus on the $i,j$ entry of $K_r$, which we denote by $K^{ij}_r$.
We need to check that (2) and (3) hold for $r\in\{\munderbar{r}+1,\dotsc,\bar{r}\}$, which in our setting means $r\in\{0,1\}$.
From \eqref{eq:PNG-kernel-two-sided-hit} we have that
$\p_t K^{ij}_r(u, v)$ equals $K^{ij}_r(u - 1, v) - K^{ij}_r(u + 1, v) + K^{ij}_r(u, v - 1) - K^{ij}_r(u, v + 1)$ plus a term coming from differentiating the scattering operator, and $\p_{\overline x} K^{ij}_r(u, v) $ equals $-K^{ij}_r(u - 1, v) - K^{ij}_r(u + 1, v) + K^{ij}_r(u, v - 1) + K^{ij}_r(u, v + 1)$ plus a similar term.
Using this we get, making now the scattering terms explicit,
\begin{equation}
\begin{aligned}
\p_\eta K^{ij}_r(u, v) &= K^{ij}_r(u, v - 1) - K^{ij}_r(u + 1, v) + \W^{(\eta)}_{t,x_i,x_j}(u+r_i+r,v+r_j+r),\\
\p_{\zeta} K^{ij}_r(u, v) &= K^{ij}_r(u-1,v)-K^{ij}_r(u, v + 1) + \W^{(\zeta)}_{t,x_i,x_j}(u+r_i+r,v+r_j+r)
\end{aligned}\label{eq:petazeta3}
\end{equation}
with
\begin{equation}\label{eq:Wetazeta0}
 \W^{(\zeta)}_{t,x,x'}=e^{-2\ft\nabla-\fx\Delta}\Big(\p_a\J^{h_0}_{a,b}\big|_{\substack{a=x-t\\b=x'+t}}\Big)e^{2\ft\nabla+x'\Delta},\quad
 \W^{(\eta)}_{t,x,x'}=e^{-2\ft\nabla-\fx\Delta}\Big(\p_b\J^{h_0}_{a,b}\big|_{\substack{a=x-t\\b=x'+t}}\Big)e^{2\ft\nabla+x'\Delta}.
\end{equation}
Hence (2) and (3) will follow (since $K_r(u, v - 1)=K_{r-1}(u+1, v)$,  $K_r(u-1, v)=K_{r-1}(u, v+1)$) if we show that the $\W^{(\eta)}_{t,x_i,x_j}$ and $\W^{(\zeta)}_{t,x_i,x_j}$ terms on the right hand side of \eqref{eq:petazeta3} vanish for $u,v>0$ or, what is the same (since we are assuming $r_i>r_0(t,x_i)$, $r_j>r_0(t,x_j)$, and $r\in\{0,1\}$), that
\begin{equation}\label{eq:Wetazeta}
W^{(\eta)}_{t,x_i,x_j}(u,v)=W^{(\zeta)}_{t,x_i,x_j}(u,v)=0\quad\forall u>r_0(t,x_i)+1,\,v>r_0(t,x_j)+1.
\end{equation}
We do this in the next subsection.

$K_r$ itself is not trace class in general, but there is a multiplication operator $\vartheta$ (independent of $t$, the $x_i$'s and the $r_i$'s) such that the conjugated kernel $\vartheta K_r\vartheta^{-1}$ is trace class.
This is proved in Prop. \ref{prop:trcl}, and since the above argument remains valid if we replace $K_r$ by this conjugation, and since both \eqref{3h'} and \eqref{eq:Frs-detQr1} also do not change under this conjugation, we may perform this replacement to prove Thm. \ref{thm:nonAbelian}.
The necessary differentiability in trace norm of the conjugated kernel appears in Appdx. \ref{sec:trcl-diff}.

It only remains to show that $I-K_r$ is invertible for $r\in\{-1,0,1\}$.
In fact, this is true for all $r\geq-1$, which is equivalent (by \eqref{eq:KrKext}) to showing that $K^{\uptext{ext}}$ is invertible whenever $r_i\geq r_0(t,x_i)$ for each $i$.
We prove this by appealing to Thm. \ref{thm:PNG-fred} and a simple probabilistic argument.
Consider the event that there are no nucleations inside the region $C_{t,x_1,x_m}=\{(s,y)\!:s\in[0,t],\,x_1-s\leq y\leq x_m+s\}$, which obviously has positive probability.
By definition of the PNG dynamics we necessarily have in that case that $h(t,x_i)=r_0(t,x_i)$ for each $i$, and then since we are taking $r_i\geq r_0(t,x_i)$ we have
\[\pp_{h_0}(h(t,x_i)\leq r_i)\geq\pp_{h_0}(\uptext{no nucl. in }C_{t,x_1,x_m})>0.\]
But by \eqref{fred2} the left hand side equals the Fredholm determinant of the identity minus the (trace class, after conjugation) kernel $K^{\uptext{ext}}$, so $I-K^{\uptext{ext}}$ is invertible.

To finish the proof of the theorem we need to prove \eqref{eq:Frs-detQr1}.
By \eqref{fred} we have $F(t, x_1,\ldots, x_n, r_1+1,\ldots, r_n+1)=\det(I-K_{r+1})_{\oplus_n\ell^2(\zz_{>0})}$, where we keep thinking of $r$ as a scalar variable.
The kernel $K_r$ satisfies the hypotheses of Thm. \ref{thm:nonAbelian}, so in particular the formalism employed there also applies here.
In particular, by \eqref{eq:sigmaKsigma} and the cyclic property of the Fredholm determinant we have
\begin{align}
F(t, x_1,\ldots, x_n, r_1\tsm+\tsm1,\ldots, r_n\tsm+\tsm1)&=\det(I-\sigma K_r\sigma^*)_{\oplus_n\ell^2(\zz_{>0})}=\det(I-(I-P)K_r)_{\oplus_n\ell^2(\zz_{>0})}\\
&\hspace{-0.6in}=\det(I+PR_rK_r)_{\oplus_n\ell^2(\zz_{>0})}\det(I-K_r)_{\oplus_n\ell^2(\zz_{>0})}\label{eq:FdetF}\\
&\hspace{-0.6in}=\det(I+PR_rK_r)_{\oplus_n\ell^2(\zz_{>0})}F(t, x_1,\ldots, x_n, r_1,\ldots, r_n).
\end{align}
Now $P$ is a projection onto an $n$-dimensional subspace of $\oplus_n\ell^2(\zz_{>0})$, and the matrix of the linear transformation $P+PR_rK_r$ on $\rr^n$ is given by $\big(\uno{i=j}+(R_rK_r)_{ij}(1,1)\big)_{i,j=1}^n=\big((R_r)_{ij}(1,1)\big)_{i,j=1}^n$ (since $R_rK_r=R_r-I$).
Hence
\[\det(I+PR_rK_r)_{\oplus_n\ell^2(\zz_{>0})}=\det((I-P)+(P+PR_rK_r))_{\oplus_n\ell^2(\zz_{>0})}=\det(R_r(1,1)),\]
the last being a determinant of an $n\times n$ matrix.
Using this and the definition $Q_r=R_r(1,1)$ in \eqref{eq:FdetF} yields \eqref{eq:Frs-detQr1}.

\subsection{The forcing term}\label{sec:forcing}

Our goal now is to prove \eqref{eq:Wetazeta}.
In fact, we will prove a more general result, which will also play an important role when we prove in Sec. \ref{PNGMarkov} that our Fredholm determinant satisfies the Kolmogorov backward equation for PNG.
To state it we introduce the following notational convention: when an object $A_a$ depends on a parameter $a\in\rr$, then the notations $A_{a^-}$ and $A_{a+}$ will always mean
\[A_{a^+}=\lim_{s\searrow a}A_s\qqand A_{a^-}=\lim_{s\nearrow a}A_s,\]
assuming these limits are well defined.
We also define $P^{\nohit(h)}_{a,b}=e^{(b-a)\Delta}-P^{\hit(h)}_{a,b}$.
In the following lemma and its proof the above notational convention will be used to denote by $P^{\hit\tsm/\tsm\nohit(h)}_{a^+,b}$ and $P^{\hit\tsm/\tsm\nohit(h)}_{a,b^-}$ the transition probabilities for $\fN$ hitting/not hitting $\hypo(h)$ on the half-open intervals $(a,b]$ and $[a,b)$, respectively.

\begin{lem}\label{lem:petazeta4}
Let $t\geq0$ and $x,x'\in\rr$ and let $\W^{(\zeta)}_{t,x,x'}$ and $\W^{(\eta)}_{t,x,x'}$ be defined as in \eqref{eq:Wetazeta0}.
Then
\begin{equation}\label{eq:petazeta4gen}
\W^{(\zeta)}_{t,x,x'}(u,v)=0\quad\forall\tts u>h_0(x-t)+1,\qquad
\W^{(\eta)}_{t,x,x'}(u,v)=0\quad\forall\tts v>h_0(x'+t)+1.
\end{equation}
Moreover,
\begin{equation}\label{eq:petazeta5gen}
\begin{aligned}
\W^{(\zeta)}_{t,x,x'}(h_0(x-t)+1,v)&=-\tfrac12P^{\nohit(h_0)}_{(x-t)^+,x'+t}e^{2t+2t\nabla-t\Delta}(h_0(x-t),v),\\
\W^{(\eta)}_{t,x,x'}(u,h_0(x'+t)+1)&=-\tfrac12e^{2t-2t\nabla-t\Delta}P^{\nohit(h_0)}_{x-t,(x'+t)^-}(u,h_0(x'+t)).
\end{aligned}
\end{equation}
\end{lem}

Since $r_0(t,x)>h_0(x\pm t)$ for all $x$, \eqref{eq:petazeta4gen} clearly implies \eqref{eq:Wetazeta}, finishing the proof of Thm. \ref{thm:nonAbelian}.
Note that in this use of the lemma in the proof Thm. \ref{thm:nonAbelian} we have only employed the condition $r_k>h_0(t,x_k)$; the stronger condition $r_k>r_0(t,x_k)$ assumed in the theorem is required to ensure that $K_{r}$ is invertible where needed.
The identities in \eqref{eq:petazeta5gen}, on the other hand, will be used in Sec. \ref{sec:whyhit}.

\begin{rem}\label{rem:1ptforcing}
In the one-point case there is a transparent (though partial) argument that the extra derivatives coming from  the dependence of $\J^{h_0}_{x-t,x+t}$ on $t$ and $x$ do not contribute to the computation (i.e. \eqref{eq:Wetazeta}).
Fix $t\geq0$, $x\in\rr$.
For $a\leq b$, let
$
K^{a,b}_r(u,v)=e^{- 2 \ft \nabla - \fx \Delta} \J^{h_0}_{a,b} e^{2 \ft \nabla + \fx \Delta}(u+r,v+r)$, the kernel from \eqref{eq:matrixKPNG} in the one-point case except that in the scattering transform appearing in \eqref{eq:PNG-kernel-two-sided-hit} we remove the dependence on $x-t$ and $x+t$ and replace it by arbitrary parameters $a,b$.
If we let
\[F_r(t,x)=F_{r,x-t,x+t}(t,x),\qquad F_{r,a,b}(t,x)=\det(I-K^{a,b}_r)_{\ell^2(\zz_{>0})}\]
 for $r\geq r_0(t,x)$, then a simple adaptation of the arguments from the last subsection  tells us that $\tfrac14(\p_t^2-\p_x^2)\log F_{r}-\frac{F_{r+1}F_{r-1}}{F_r^2}$ is given by 
\begin{equation}\label{eq:TodaForcing}
-\big(\p_{x+t}\p_a\log(F_{r,a,b})-\p_{x-t}\p_b\log(F_{r,a,b})-\p_a\p_b\log(F_{r,a,b})\big)\big|_{a=x-t,b=x+t}.
\end{equation}
To prove that the 2D Toda equation \eqref{eq:waveToda} holds, we need to show that this \emph{forcing term} vanishes for $r>r_0(t,x)$.
The argument we present next will show that the left derivative $\p_{b^-}\log(F_{r,a,b})|_{b=x+t}$ vanishes; the same argument shows $\p_{a^+}\log(F_{r,a,b})|_{a=x-t}=0$.
 Introduce  truncated initial data
$h_0^{[a,b]}(x) = h_0(x) \uno{x \in [a,b]} - \infty \cdot \uno{x \notin [a,b]}$.
If $x-t\leq a\leq b\leq x+t$ then  $P^{\hit(h_0^{[a,b]})}_{x-t,x+t}=e^{(a-x+t)\Delta}P^{\hit(h_0)}_{a, b}e^{(x+t-b)\Delta}$ because $\fN$ cannot hit $\hypo(h_0^{[a,b]})$ outside $[a,b]$.  Then we may write $\J^{h_0}_{a,b}=e^{(x-t)\Delta} P^{\hit(h_0^{[a,b]})}_{x-t,x+t} e^{-(x+t)\Delta}=\J^{h_0^{[a,b]}}_{x-t,x+t}$.
In other words, when $[a,b]\subseteq[x-t,x+t]$, the scattering transform for $h_0$ in the interval $[a,b]$ is the same as for the truncated initial data $h_0^{[a,b]}$ in the interval $[x-t,x+t]$, and thus
\begin{equation}
  F_{r,a,b}(t,x) = F_r(t,x;h_0^{[a,b]}) = \pp_{h_0^{[a,b]}} \bigl(h(t,x) \leq r\bigr)= \pp_{\mathfrak{d}_x^{-r}} \bigl(h(t,\cdot)\leq -h_0^{[a,b]}\bigr).
\end{equation}
The third equality comes from the skew time reversal invariance of PNG \eqref{eq:timerev}, with $\mathfrak{d}_x^{-r}$ a shifted narrow wedge given by $\mathfrak{d}_x^{-r}(y)=-r$ if $x = y$ and $-\infty$ otherwise.
In view of this we may compute $\p_{b^-} F_{r, a, b}\big|_{b=x+t}= \lim_{\delta \to 0} \delta^{-1} \left(\pp_{\mathfrak{d}_x^{-r}} \bigl(h(t,\cdot)\leq -h_0^{[a,x+t]} \bigr) - \pp_{\mathfrak{d}_x^{-r}} \bigl(h(t,\cdot)\leq -h_0^{[a,x+t-\delta]}\bigr) \right)$ as
\[-\lim_{\delta \to 0} \delta^{-1} \pp_{\mathfrak{d}_x^{-r}} \bigl(\forall\y\in [a, x+t-\delta],~h(t,y) \leq -h_0(y); ~\exists z\in (x+t-\delta, x+t], h(t,z) > -h_0(z) \bigr).\]
The key is that the derivative is being computed at the edge of the forward light cone, and that, in view of the initial condition, $h(t,x+t)=-r$.
Suppose first that $x+t$ is not a jump point for $-h_0$ so that, in particular, $h_0$ is constant, and equal to $h_0(x+t)$, on $[x+t-\delta,x+t]$ if $\delta$ is small enough.
Then on the event inside the probability we have $h(t,x+t-\delta)\leq -h_0(x+t)$ while $h(t,x+t)=-r\leq-h_0(x+t)$.
Hence, for the event to occur, $h(t,y)$ has to jump up \emph{and then down} in the interval $[x+t-\delta,x+t]$, which has probability $\mathcal{O}(\delta^2)$.
The same holds if $h_0$ has a down jump at $x+t$, because $h_0$ is upper semi-continuous, so it is still constant on $[x+t-\delta,x+t]$.
The only relevant possibility then is that $h_0$ has an up jump at $x+t$.
In this case, and up to terms of order $\delta^2$, the event inside the probability will occur if $h(t,y)$ stays below $-h_0$ on $[a,x+t-\delta]$ and $-r=h(t,x)>-h_0(x+t)$, but we are assuming $r\geq h_0(x+t)$.  This explains why the forcing terms \eqref{eq:TodaForcing} vanish. 
\end{rem}

We turn now to the full proof of Lem. \ref{lem:petazeta4}.
The first step is to prove the following preliminary result:

\begin{lem}\label{lem:pbPhit}
For $a<b$ and $h\in\UC$,
\begin{subequations}
\begin{align}
\p_a P^{\hit(h)}_{a,b}(u,v)&=\Delta P^{\hit(h)}_{a^+,b}(u,v)+\tfrac12\uno{u=h(a)+1}P^{\nohit(h)}_{a^+,b}(h(a),v)\quad\forall\tts u>h(a),\label{eq:pbPhita}\\
\p_b P^{\hit(h)}_{a,b}(u,v)&=P^{\hit(h)}_{a,b^-}\Delta(u,v)+\tfrac12P^{\nohit(h)}_{a,b^-}(u,h(b))\uno{v=h(b)+1}\quad\;\forall\tts v>h(b).\label{eq:pbPhitb}
\end{align}
\end{subequations}
\end{lem}

Note that the second term on the right hand side of each identity only appears when $h$ has a jump at the edges of $[a,b]$ (more precisely, a down jump at $a$ for \eqref{eq:pbPhita} and an up jump at $b$ for \eqref{eq:pbPhitb}), because if $h(b^-)=h(b)$ then $P^{\nohit(h)}_{a,b^-}(u,h(b))=0$, and similarly for the other edge.
Note also that the derivatives may become singular without the restrictions on $u$ and $v$.
For example, take $h(x)$ to be $-\infty$ for $x<1$ and $0$ for $x\geq1$.
Then $P^{\hit h}_{0,b}(0,0)$ is not even continuous at $b=1$.

\begin{proof}
We will only prove \eqref{eq:pbPhitb}, \eqref{eq:pbPhita} is completely analogous (and can also be derived by considering the adjoint of the kernel in the second line and reversing the direction of the random walk).
Decomposing $P^{\hit(h)}_{a,b+\ep}$ as $P^{\hit(h)}_{a,b-\ep}e^{2\ep\Delta}+P^{\nohit(h)}_{a,b-\ep}P^{\hit(h)}_{b-\ep,b+\ep}$, one checks directly that
\begin{equation}
\p_{b}P^{\hit(h)}_{a,b}=\lim_{\ep\searrow0}(2\ep)^{-1}\big(P^{\hit(h)}_{a,b+\ep}-P^{\hit(h)}_{a,b-\ep}\big)=P^{\hit(h)}_{a,b^-}\Delta+P^{\nohit(h)}_{a,b^-}\lim_{\ep\searrow0}(2\ep)^{-1}P^{\hit(h)}_{b-\ep,b+\ep},\label{eq:dbP1}
\end{equation}
where $P^{\hit(h)}_{a,b^-}$ is the transition probability for $\fN$ without hitting $\hypo(h)$ on the half-open interval $[a,b)$.

Consider the limit on the right hand side of \eqref{eq:dbP1}.
Since $P^{\nohit(h)}_{a, b^-}(u,\eta)=0$ if $\eta\leq h(b^-)$,  we only need to consider $(2\ep)^{-1}P^{\hit(h)}_{b-\ep,b+\ep}(\eta,v)$ for $\eta>h(b^-)$.
By upper semi-continuity we may assume that $\ep$ is small enough so that $h(s)\leq h(b)$ for all $s\in[b-\ep,b+\ep]$.
If $\eta>h(b)$ then $P^{\hit(h)}_{b-\ep,b+\ep}(\eta,v)=\mathcal{O}(\ep^2)$, because to go from $\eta$ to $v$ hitting $\hypo(h)$ involves $\fN$ jumping at least twice in the interval $[b-\ep,b+\ep]$ (first down to go below level $h(b)$, and then up to get to $v>h(b)$).
Now suppose $h(b^-)<\eta\leq h(b)$.
If $v-\eta>1$ then $P^{\hit(h)}_{b-\ep,b+\ep}(\eta,v)=\mathcal{O}(\ep^2)$.
The only other alternative is $\eta=h(b)$, $v=h(b)+1$, in which case $\fN$ can go from $\eta$ to $v$ hitting $\hypo(h)$ in $[b-\ep,b+\ep]$ by staying put up to time $b$ and then jumping up by one inside the interval $(b,b+\ep]$.
The conclusion is that for $\eta>h(b^-)$ and $v>h(b)$,
\[P^{\hit(h)}_{b-\ep,b+\ep}(\eta,v)=e^{-\ep}e^{\ep\Delta}(0,1)\uno{\eta=h(b)>h(b^-),v=h(b)+1}+\mathcal{O}(\ep^2).\]
Using this in \eqref{eq:dbP1} and computing the limit we deduce that, for $v>h(b)$,
\[\p_bP^{\hit(h)}_{a,b}(u,v)=P^{\hit(h)}_{a,b^-}\Delta(u,v)+\tfrac12P^{\nohit(h)}_{a,b^-}(u,h(b))\uno{v=h(b)+1}\uno{h(b^-)<h(b)}.\]
The last indicator function can be removed (because $P^{\nohit(h)}_{a,b^-}(u,h(b))=0$ if $h(b)=h(b^-)$) and we get the claimed identity.
\end{proof}

We also state the following simple result, which we will use often:

\begin{lem}\label{lem:ett}
The kernel $e^{2t\nabla-t\Delta}$ is lower triangular, i.e. $e^{2t\nabla-t\Delta}(u,v)=0$ for all $v-u\geq1$.
\end{lem}

\begin{proof}
From \eqref{eq:Scontour} we have $e^{2 t \nabla -t \Delta}(u,v) = e^{2\ft} \frac{1}{2\pi\I}\oint_{\gamma_r} \frac{\d z}{z^{v - u + 1}} e^{-2\ft z^{-1}}$, 
which vanishes if $v-u\geq1$ since then there is no residue at infinity.
\end{proof}

\begin{proof}[Proof of Lem. \ref{lem:petazeta4}]
We will only consider the $\eta$ the derivatives, corresponding to the right edge of the interval $[x-t,x'+t]$; the $\zeta$ derivatives follow in the same way.
Given $a,b\in\rr$, write
\begin{equation}\label{eq:L0}
\W=e^{- 2 t \nabla - x \Delta} \J^{h_0}_{a,b} e^{2 t \nabla + x' \Delta},
\end{equation}
so that what we are trying to compute is $\p_b\W(u,v)|_{a=x-t,b=x'+t}$ for $v>h_0(x'+t)$.

Start by assuming that $x-t<x'+t$, which means that in $\W$ we are taking $a<b$.
By definition of $\J^{h_0}_{a,b}$ we have $\p_{b}\J^{h_0}_{a,b}=e^{a\Delta}\big(-P^{\hit(h_0)}_{a,b}\Delta+\p_{b}P^{\hit(h_0)}_{a,b}\big)e^{-b\Delta}$, so
\begin{equation}
\p_{b}\W=e^{- 2 t \nabla +(a-x)\Delta}\big(-P^{\hit(h_0)}_{a,b}\Delta + \p_{b}P^{\hit(h_0)}_{a,b}\big)e^{2 t \nabla + (x'-b)\Delta}.\label{eq:dbminus0}
\end{equation}
Evaluating at $b=x'+t$, the last factor on the right becomes $e^{2 t \nabla -t \Delta}$, and then by Lem. \ref{lem:ett}, 
if we want to compute $\p_bW\big|_{b=x'+t}(u,v)$ for $v>h_0(x'+t)$ then we only need to consider the factor $\p_bP^{\hit(h_0)}_{a,b}\big|_{b=x'+t}(u,\eta)$ for $\eta>h_0(x'+t)$.
In this case Lem. \ref{lem:pbPhit} shows that $\p_bP^{\hit(h_0)}_{a,b}\big|_{b=x'+t}(u,\eta)=P^{\hit(h_0)}_{a,(x'+t)^-}\Delta(u,\eta)+\frac12P^{\nohit(h_0)}_{a,(x'+t)^-}\Delta(u,h_0(x'+t))\uno{\eta=h_0(x'+t)+1}$, so going back to \eqref{eq:dbminus0} and using $P^{\hit(h_0)}_{x-t,x'+t}-P^{\hit(h_0)}_{x-t,(x'+t)^-}=-P^{\nohit(h_0)}_{x-t,(x'+t)^-}$ we deduce that
\begin{multline}\label{eq:peta23}
\p_bW\big|_{\substack{a=x-t\\b=x'+t}}(u,v)=-e^{- 2 t \nabla-t\Delta}P^{\nohit(h_0)}_{x-t,(x'+t)^-}\bP_{h_0(x'+t)}\Delta e^{2t\nabla-t\Delta}(u,v)\\
 +\tfrac12e^{-2t\nabla-t\Delta}P^{\nohit(h_0)}_{x-t,(x'+t)^-}(u,h_0(x'+t))e^{2t\nabla-t\Delta}(h_0(x'+t)+1,v).
\end{multline}
From Lem. \ref{lem:ett}, the last factor on the second line vanishes if $v>h_0(x'+t)+1$, so by our assumption on $v$ this term only contributes to the last expression in the case $v=h_0(x'+t)+1$.
Now we focus on the first term on the right hand side, and consider the factor $\bP_{h_0(x'+t)}\Delta e^{2 t \nabla -t \Delta}(\xi,v)$ ($\xi$ is an intermediate variable in the kernel compositions), which involves $\bP_{h_0(x'+t)}e^{2 t \nabla -t \Delta}(\xi+i,v)$ for $i\in\{-1,0,1\}$.
Using Lem. \ref{lem:ett} again together with the assumption $v>h_0(x'+t)$, all these terms vanish except in the case $i=1$, $\xi+1=v=h_0(x'+t)+1$, and we conclude that $\bP_{h_0(x'+t)}\Delta e^{2 t \nabla -t \Delta}(\xi,v)=\uno{\xi+1=v=h_0(x'+t)+1}e^{2t\nabla-t\Delta}(h_0(x'+t)+1,v)$, and hence that 
\begin{equation}\label{eq:peta24}
\begin{multlined}
    e^{- 2 t \nabla-t\Delta}P^{\nohit(h_0)}_{x-t,(x'+t)^-}\bP_{h_0(x'+t)}\Delta e^{2t\nabla-t\Delta}(u,v)\\
=e^{-2t\nabla-t\Delta}P^{\nohit(h_0)}_{x-t,(x'+t)^-}(u,h_0(x'+t))e^{2t\nabla-t\Delta}(h_0(x'+t)+1,v)\uno{v=h_0(x'+t)+1}.
\end{multlined}
\end{equation}
Using this, together with the identity $e^{2t\nabla-t\Delta}(\eta,\eta)=e^{2t}$ coming from \eqref{eq:Scontour}, in \eqref{eq:peta23} yields
\begin{equation}
\p_bW\big|_{\substack{a=x-t\\b=x'+t}}(u,v)=-\tfrac12e^{-2t\nabla-t\Delta}P^{\nohit(h_0)}_{x-t,(x'+t)^-}(u,h_0(x'+t))e^{2t}\uno{v=h_0(x'+t)},\label{eq:peta25}
\end{equation}
which proves the second identity in each of \eqref{eq:petazeta4gen} and \eqref{eq:petazeta5gen} under the assumption $x+t<x'+t$.

Now if $x-t>x'+t$, then $\J^{h_0}_{x-t,b}=0$ for $b$ in a neighborhood of $x'+t$, and the identities are trivial.
What is left is to consider the case $x-t=x'+t$, for which we need to compute 
\[e^{- 2 t \nabla-t\Delta}\big(\p_{b}P^{\hit(h_0)}_{a,b}\big|_{a=b=x'+t}\big)e^{2 t \nabla -t\Delta}=\lim_{\ep\searrow0}(2\ep)^{-1}e^{- 2 t \nabla-t\Delta} (P^{\hit(h_0)}_{x'+t,x'+t+\ep}-P^{\hit(h_0)}_{x'+t,x'+t-\ep}) e^{2 t \nabla -t\Delta}.\]
We have $P^{\hit(h_0)}_{x'+t,x'+t-\ep}=0$ (because $\ep>0$) while proceeding as above we have, for $\xi>h_0(x'+t)$, $P^{\hit(h_0)}_{x'+t,x'+t+\ep}(\eta,\xi)=\ep\uno{\eta=h_0(x'+t),\,\xi=h_0(x'+t)+1}+\mathcal{O}(\ep^2)$ and hence furthermore that the above limit equals $\frac12e^{- 2 t \nabla-t\Delta}(u,h_0(x'+t))e^{2 t \nabla -t\Delta}(h_0(x'+t)+1,v)$ when evaluated at $u>h_0(x-t)$, $v>h_0(x'+t)$.
Using this in \eqref{eq:dbminus0} and proceeding as above show that $\p_bW\big|_{\substack{a=x-t\\b=x'+t}}(u,v)$ equals
\[-e^{- 2 t \nabla-t\Delta}\bP_{h_0(x'+t)}\Delta e^{2 t \nabla-t\Delta}(u,v)+\tfrac12e^{2t- 2 t \nabla-t\Delta}(u,h_0(x'+t))\uno{v=h_0(x'+t)+1}.\]
Exactly the same argument we used to get to \eqref{eq:peta24} and \eqref{eq:peta25} shows that the first term equals $e^{2t- 2 t \nabla-t\Delta}(u,h_0(x'+t))\uno{v=h_0(x'+t)+1}$, so the above expression equals $-\frac12e^{2t- 2 t \nabla-t\Delta}(u,h_0(x'+t))\uno{v=h_0(x'+t)+1}$, which yields the $\eta$ derivatives in \eqref{eq:petazeta4gen} and \eqref{eq:petazeta5gen} in the remaining case $x-t=x'+t$. 
\end{proof}

\section{PNG as a Markov process}\label{PNGMarkov}

We will need to develop some basic Markov machinery for PNG in order to show that the Fredholm determinant expressions \eqref{fred} uniquely solve the backward (Kolmogorov) equations. To our surprise, there does not seem to be any previous reference, except  \cite{MR607609} where the equation is written down and invariant measures computed from a physics viewpoint.

Let us briefly describe the plan.  First of all, PNG is constructed directly as a Markov process, with a generator $\hat\SL$ coming from general Markov process theory.  Continuous time simple random walks are shown to be invariant modulo the height shift and  the adjoint $\hat\SL^*$ is computed.  It is exactly the generator of the backward dynamics for PNG.  Note that one can also think of this as the forward process ``upside down'', i.e. $h\mapsto -h$.  

The transition probabilities should satisfy the backward equation, but there is a technical issue which permeates the entire proof: $\varphi(h)=\uno{h(x_i)\le r_i, i=1,\ldots,n}$ is \emph{not} a continuous function on $\UC$   (for example, if $h_\ep(x)=\log \uno{x\ge \ep}$ then 
$h_\ep\longrightarrow h_0$ in $\UC$, but $\uno{h_\ep(0)\le -1}=1$ for $\ep>0$ and $0$ for $\ep=0$).  But the  formulas are only for exactly $\ee_{h}( \varphi(h(t)))$.  In fact, the backward equations only hold in a weak sense and one cannot just appeal to general theory.  For the Fredholm determinants the key behind the backward equations is the identity \eqref{eq:LP}, which shows that the time evolution of the scattering kernels corresponds to the action of the generator.  

Once one knows that the determinantal formulas satisfy the backward equation in an appropriate sense, uniqueness is provided by the existence of a large enough class of solutions of the adjoint equation (see Prop. \ref{prop:4.7}).  This is now automatic, since the Fredholm determinant formulas can be ``flipped on their head'' to produce solutions of the adjoint equation, which is just the backward equation for the solution upside down.  

\subsection{Semigroup}\label{states}

The state space $\UC$ consists of upper semi-continuous height functions $h:\mathbb{R}\to\mathbb{Z}\cup\{-\infty\}$.  A function is upper semi-continuous if and only if its hypograph $\hypo(f)=\{ (x,y): y\le f(x)\}$ is closed. We compactify at $-\infty$ by introducing a metric on $\mathbb{Z}\cup\{-\infty\}$ so that $[-\infty, 0]$ has finite measure.  Then $f_n\longrightarrow f$ if there exist\footnote{In \cite{fixedpt} this is mistakenly written as ``for all $M\ge 1$''.} $M_\ell\to \infty$ such that $\hypo(f_n) \cap [-M_\ell,M_\ell] \longrightarrow \hypo(f) \cap [-M_\ell,M_\ell]$ in the Hausdorff topology. Equivalently, $f_n\longrightarrow f$ if for each $x$, $\limsup_{n\to \infty} f_n(x_n) \le f(x)$ whenever $x_n\to x$  and there exists $x_n\to x$ with $\liminf_{n\to \infty} f_n(x_n)\ge f(x)$.

The first order of business is to construct the Markov semigroup acting on continuous functions $\mathscr C(\UC)$ on $\UC$.  Note that because of the finite propagation speed, unlike models such as the exclusion process, there is an elementary construction.  

Start with a space-time set $\mathcal N$ of nucleation points which is locally finite.  Let $\mathcal N_A$ denote the restriction to a subset $A$ of space-time.
Define a map $\mathcal T_{s,t,x}$ taking a $\UC$ function $h$ on $[x-(t-s),x+t-s]$ ``at time $s$'' and $\mathcal N_A$, $A={\{(u,y): |y-x|\le t-u, s\le u \le t\}}$  into $h(t,x)$ as follows. Let the points of $\mathcal N_A$
be denoted $(u_1,y_1), \ldots, (u_M,y_N)$ with ordered times $u_1< \cdots <u_M$.  $M\le N$ since several points may have the same $u_i$.  On  the time interval $(s, u_1)$ we evolve $h$ to $h((u_1)_-,z)= \sup_{|\tilde z-z|\le u_1-s} h(\tilde z)$.  At time $u_1$ there are nucleation points $y_1,\ldots,y_{m_1}$ and we add $1$ to  $h$ at each to get  $h(u_1,z)=h((u_1)_-,z)+\sum_{i=1}^{m_1} \mathbf{1}_{z=y_i}$.
Now on the time interval $( u_1,u_2)$ we evolve $h$ to
$h((u_2)_-,z)= \sup_{|\tilde z-z|\le u_2-u_1} h(u_1,\tilde z)$ and then, at time $u_2$ there are nucleation points $y_{m_1+1},\ldots,y_{m_2}$ and we add $1$ to  $h$ at each to get  $h(u_2,z)=h((u_2)_-,z)+\sum_{i=m_1+1}^{m_2} \mathbf{1}_{z=y_i}$, etc.  Since there are finitely many  points in $\mathcal N_A$, the process is well-defined.  Since $\mathcal T_{s,t,x}$ knows already to only use the points of $\mathcal N$ in $A$ and the part of $h$ in $[x-(t-s),x+t-s]$, we can just think of it as a map 
$
(h,\mathcal N)\mapsto  \mathcal T_{s,t,x}(h,\mathcal N)
$.
It has the property that for $\mathfrak d_y$ defined in \eqref{eq:def-nw},
\begin{equation}\label{varform0}
\mathcal T_{s,t,x}(h,\mathcal N)  =\max_{|y-x|\le t-s}\{\mathcal T_{0,t,x}(\mathfrak d_y,\mathcal N)+ h(y)\},
\end{equation}
as well as the semigroup property that for $s<u<t$,
 \begin{equation}\label{semiform}
\mathcal T_{s,t,x}(h,\mathcal N)  =\mathcal T_{u,t,x}(
\mathcal T_{s,u,\cdot}(h,\mathcal N),\mathcal N).
\end{equation}
In other words, given the background nucleations, PNG is actually a well-defined, \emph{deterministic} Markov semigroup on $\UC$.  
We then choose this background $\mathcal N$ as a rate $2$ space-time Poisson point process, and this constructs our PNG model as
\begin{equation}\label{varform2}
h(t,x)=h(t,x;h_0)=\mathcal T_{0,t,x}(h_0,\mathcal N).
\end{equation}

$\UC$ is a complete, separable metric space, though as it is infinite dimensional, it is not locally compact. 
If we define, for $f\in \mathscr C (\UC)$, 
\begin{equation}
\PPP_tf(h) = \ee[ f(\mathcal T_{0,t,\cdot}(h,\mathcal N))],
\end{equation}
then $\{\PPP_t, t\ge 0\}$ is a semigroup.  To see this, note first that $\PPP_0$ is the identity.  By \eqref{semiform},
\begin{equation}
\PPP_{s+t}f(h)=  \ee[ f(\mathcal T_{s,s+t,\cdot}(\mathcal T_{0,s,\cdot}(h,\mathcal N_{0,s}), \mathcal N_{s,s+t})],
\end{equation}
where the subscripts on the $\mathcal N$ indicate which time interval is being used.  Since the Poisson processes on the two time intervals are independent, we can take the expectation over $\mathcal N_{s,s+t}$ first to get 
\begin{equation}
  \ee[ \PPP_tf(\mathcal T_{0,s,\cdot}(h,\mathcal N_{0,s})]= \PPP_s\PPP_tf(h).
\end{equation}
 $ \PPP_t f$ are also strongly  continuous. 
Recall that such a semigroup is Markov if, in addition, $\PPP_t 1=1$ and $\PPP_t f\ge 0$ for all nonnegative $f\in \mathscr C(\UC)$.  Both are immediate from the definition. 
The generator $\hat \SL$ is defined for $f\in \mathscr{D}(\hat\SL)$ by
\begin{equation}
\hat\SL f=\lim_{t\searrow 0}\tfrac1{t} (\PPP_t f- f),
\end{equation}
where $\mathscr{D}(\hat\SL)$ is just the set of those $f\in \mathscr C(\UC)$ for which the limit exists strongly (i.e. in  $\mathscr C(\UC)$, with the sup topology).
The Hille-Yosida theorem \cite[Thm. 1.5 and 2.6]{ethierKurtz}  tells us that $\mathscr{D}(\hat\SL)$ is dense in $\mathscr C(\UC)$ and that for $f\in  \mathscr C(\UC)$ and $t>0$, $\PPP_t\in \mathscr{D}(\hat\SL)$ with  
\begin{equation}\label{HYbackeq}
\p_t \PPP_t f=\hat\SL \PPP_t f.
\end{equation}

There is also a version of PNG on $[-L,L)$ with periodic boundary conditions.
We call the corresponding generator $\hat\SL_L$.  Clearly, everything we have described has an analogue in that case.

\begin{rem}  Our situation is unusual.  Usually one starts with a generator and attempts to build from it a semigroup/Markov process.  In the case of PNG, the semigroup/Markov process is explicit, so we can just read off the generator. 
\end{rem}

\subsection{Generator}

The above description of the generator is not explicit and 
in order to do calculations, one has to introduce coordinates.   A $\UC$ function essentially consists of unit steps up and down, together with an absolute height at some point, say $h(0)$. 
We consider cases where the locations of the steps are an at most countably infinite vector of points $\y_i\in \rr$, $i\in \zz$, which for simplicity we will assume to be ordered by their label, $\y_i\le \y_{i+1}$.  
To each step $\y_i$ corresponds a $\pm1$ variable $\sigma_i$; we set $\sigma_i=+1$ if the step is up and $\sigma_i=-1$ if the step is down (i.e. $\lim_{\ep\searrow 0} h(\y_i-\sigma_i\ep)=h(\y_i)-1$).  We interpret a pair $(\y_i,\sigma_i)$ as a particle with spin $\sigma_i$ located at $\y_i$.
Note that this allows for steps of an arbitrary size because we are allowing for jump locations to coincide.

It is certainly not the case that all $\UC$ functions fit this description.  
For example, $\mathfrak d_0$ ``wants'' to correspond to $h(0)=0$ and $(\y_i,\sigma_i)=(0^-,+1)$ for $i<0$ and $(0^+,-1)$ for $i\ge 0$, but this is not allowed by the above description.  Another example is the indicator function of the Cantor set.
It is in $\UC$, but it certainly cannot be written in this manner.
Furthermore, the valid $\UC$ function $h(x)=-\infty$, $x<0$, $h(x)=0$, $x\ge 0$ corresponds, for example, to $(\y_i,\sigma_i)=(0,+1)$, $i<0$ and $(\infty, \pm 1)$, $i\ge 0$.
And a configuration like $(\y_i,\sigma_i)=(i,+1)$, $i<0$ and $(1,+1)$, $i\ge 0$, corresponds to a \emph{non}-$\UC$ function taking the illegal value $+\infty$ on $[1,\infty)$.

On the other hand, these cases are extreme.  If one looks within one of the invariant measures for the process (see Prop. \ref{prop:invm}), the point process $\y_i$ is locally finite and simple (no coincidences) with probability one.
If we restrict to locally finite configurations, the map $h(x) \longleftrightarrow (h(0), (\y_i,\sigma_i)_{i\in\zz})$ is one-to-one up to shifts in numbering if we make the convention that unneeded particles are placed at either $+\infty$ with a plus sign, or $-\infty$ with a minus sign.
Furthermore, it is not hard to see that the dynamics stays within this class.

So we will work with generators on this restricted subspace 
\begin{multline}
\hat{\mathcal{X}}=\zz\times\mathcal{X},\quad\mathcal{X}=\big\{(\vec y,\vec \sigma)  \in((\rr\cup\{-\infty,\infty\}) \times\{-1,1\})^\zz\!:\vec y\uptext{ is locally finite},\,y_i\leq y_{i+1}\,\forall i,\\ 
y_i=y_{i+1}\uptext{ only if }(\sigma_i,\sigma_{i+1}) = (+1,-1)\big\},\label{eq:calX}
\end{multline}
and we say $h\in\hat{\mathcal{X}}$ whenever the pair $(\vec y,\vec \sigma)$ corresponding to $h$ is in $\mathcal{X}$.
At the end, we extend the results to $\UC$ by continuity from above.
It is reminiscent of entrance laws in diffusion theory.  Here one is misled because the narrow wedges $\mathfrak d_y$ appear to lead to the simplest evolution, but they do not fit nicely within the explicit computations with the generator. 

The dynamics of the vector $(\y_i,\sigma_i)_{i\in\zz}\in\mathcal{X}$ is also Markovian, and we let $\SL$ denote its generator.
The second Markov process has a slightly reduced description, as we lose the absolute height, but we will see that it is essentially enough for our purposes.
The state space for the $(\y_i,\sigma_i)_{i\in\zz}$ process is $\mathcal X$ with the topology inherited from the $\UC$ topology on the associated height functions (with, say, $h(0)=0$).

To describe the generator $\SL$, we first need to describe a core for its domain.
This will consist of \emph{local} functions $F$, which are regular in some sense.  
\emph{Local} means there is a finite interval $(a,b)$ so that $F$ depends only on $(\y_i,\sigma_i)$ with $\y_i\in (a,b)$.
We are assuming that the configuration is locally finite, so there is a finite number $n$ of $\y_i$'s in the interval.  
We may as well relabel them $\y_1 \leq \y_2 \leq \cdots \leq \y_n$, and let $\sigma_1,\ldots,\sigma_n$ be their corresponding spins. 
Let $E^{a,b}_n$ be the set of such configurations.
For $n = 0$ we set $E^{a,b}_0 = \{\emptyset\}$. We define $E^{a,b} = \bigcup_{n =0}^\infty E^{a,b}_n$ to be the set of all configurations of finitely many particles in $(a,b)$. 
We will use interchangeably the two notations $F( \dotsc, (\y_{i},\sigma_{i}), (\y_{i+1},\sigma_{i+1}),\dotsc) = F(\vec \y, \vec \sigma)$, where $\vec \y = (\dotsc, \y_i, \y_{i+1}, \dotsc)$ and $\vec \sigma = (\dotsc, \sigma_i, \sigma_{i+1}, \dotsc)$.  Denote by $F_n$ the restriction\footnote{We always mean that $F$ only depends on variables in the interval $(a,b)$, so this just means we are considering the restriction of $F$ to the situation where there are $n$ particles in that interval. }  of $F$ to $E^{a,b}_n$.

An important consequence of the $\UC$ topology is that configurations with $\y_i=\y_{i+1}$ and $\sigma_i=-1$, $\sigma_{i+1}=+1$, are identified with the configuration where the two points are removed from the list.
Thus continuity of $F$ with respect to the $\UC$ topology corresponds to \emph{matching conditions} between its restrictions,
\begin{equation}\label{eq:boundary}
\lim_{\y_{i+1}-\y_i\searrow 0}F_n( \ldots,(\y_i,-1),(\y_{i+1},+1),\ldots)= F_{n-2}( \ldots, (\y_{i-1},\sigma_{i-1}), (\y_{i+2},\sigma_{i+2}),\ldots).
\end{equation}
 
The dynamics of the up/down steps consists of two pieces, a deterministic part, with generator $\SL_{\gdet}$, and a stochastic part, with generator $\SL_{\gcr}$, for \emph{creation},
\begin{equation}\label{eq:L}
\SL = \SL_{\gdet} + \SL_{\gcr}.
\end{equation}
This generator acts on local functions $F$, smooth in the interior of each $ E_n^{a,b}$.
The deterministic part $\SL_{\gdet}$ has each up-step moving to the left with unit speed, and each down-step moving to the right with unit speed, so  that at interior points,
\begin{equation}\label{eq:L-det}
\SL_{\gdet} F_n(\vec{y},\vec{\sigma}) = - \sum_{i=1}^n\sigma_i \partial_{\y_i} F_n(\vec{y},\vec{\sigma}).
\end{equation}
The stochastic part of the dynamics has new  $(+1,-1)$ pairs created randomly uniformly with rate $2$,
\begin{equation}\label{eq:L-cr}
\SL_{\gcr} F_n(\vec{y},\vec{\sigma}) = 2 \int_{-\infty}^\infty \bigl(F_{n+2}( C_z(\vec{y},\vec{\sigma}) )- F_n(\vec{y},\vec{\sigma} )\bigr) \d z,
\end{equation}
where the configuration $C_z(\vec{y},\vec{\sigma}) $ is obtained by inserting $(z, +1)$ and $(z, -1)$ into $(\vec{y},\vec{\sigma}) $ to get an element of $E^{a,b}_{n+2}$. Since $F$ is local the integral is really only over $(a,b)$.
Since we only consider $F\in\mathscr{C}_0$ as candidates for the domain of the generator, $F_n$ are necessarily all bounded and \eqref{eq:L-cr} always makes sense.  The domain of $\SL$ therefore consists of functions for which the right hand side of \eqref{eq:L-det} is in $\mathscr{C}_0$ and our core consists of all local functions in the domain.

The analogous generator in the periodic case, on $[-L,L)$, without the heights, will be called $\SL_L$.
The periodic case differs from the whole line case in that many configurations $(y_i,\sigma_i)$ do not correspond to valid height functions.

In order to write the full generator $\hat\SL$ of the PNG height function in this language we would need to take into account the absolute height.
To this end one would need to count the number of up/down steps which have crossed a fixed point, for example the origin, to the right/left up to time $t$.
Adding this part of the dynamics to get the full generator amounts to adding an additional term on the right hand side of \eqref{eq:L}.
It would not be so hard to write the resulting $\hat\SL$, but we will never  need its explicit form in the computations.

\subsection{Invariant measures}\label{sec:inv}

The following measures 
$\{\mu_\rho\}_{ \rho\in (0,\infty)}$
are invariant for the $(\y_i,\sigma_i)_{i\in\zz}$ process on the line:  Fix $\rho\in (0,\infty)$ and let $\Xi^+$ and $\Xi^-$ be independent Poisson point processes on $\rr$ with intensities $\rho$ and $\rho^{-1}$, respectively (so, for instance, $\pp(|\Xi^+ \cap [a, b]| = n) = \frac{(\rho (b-a))^n}{n!} e^{- \rho(b-a)}$ for any $a<b$).
Superpose $\Xi^+$ and $\Xi^-$, order the resulting points $(\y_i)_{i\in \zz}$, and assign $\sigma_i=+1$ to points coming from $\Xi^+$ and $\sigma_i=-1$ to points coming from $\Xi^-$.
This results in a configuration $(\vec\y,\vec\sigma)$ which, in simple words, has $+1$ points and $-1$ points as independent Poisson point processes with intensities $\rho$ and $\rho^{-1}$, respectively.
$\mu_\rho$ denotes the law of this configuration.

We also have the analogous measures $\mu_{L,\rho}$ which are just the restrictions of $\mu_\rho$ to $[-L,L)$.
Note that while it is not always possible to build a periodic height function  out of a configuration $(y_i,\sigma_i)$ on $[-L,L)$, if one starts PNG with a periodic height function one can run the process of up and down steps and build a periodic height function out of the result at a later time.

\begin{prop}\label{prop:invm}
For any $\rho>0$, the measure $\mu_\rho$ is invariant for the $(\vec{\y},\vec\sigma)$ system.
\end{prop}

In order to prove the result we need to  show that $\int\tsm\SL F \d\mu_\rho = 0$.
We will actually perform a more general computation, which will be useful later on, and from which the proposition will follow.
Let
\[H^m_{a,b}(\vec y,\vec\sigma)=\uno{\sum_{i\in\zz:\ts y_i\in[a,b]}\sigma_i=m},\]
$m\in\zz$.
In terms of the associated height function $h$, $H^m_{a,b}$ is just the indicator function of the event that $h(b)-h(a)=m$.

\begin{prop}\label{prop:invm-general}
Fix $a<b$ in $\rr$ and $u\in\zz$, and suppose $F$ is a function which is continuous in $\UC$, which depends only on the $(y_i,\sigma_i)$ where $y_i\in (a,b)$ and on the value $u=h(a)$, and which is such that each $F_n$ is continuously differentiable in $y_i$ in the interior of $E_n^{a,b}$.
Then
\begin{multline}\label{invstat}
\int \SL F(\vec{\y},\vec\sigma;u)H^m_{a,b}(\vec y,\vec\sigma)\,\d\mu_\rho(\vec{\y},\vec\sigma)=-\int (F(\vec{\y},\vec\sigma;u)-F(\vec{\y},\vec\sigma;u+1))H^{m-1}_{a,b}(\vec y,\vec\sigma)\,\d\mu_\rho\\
+\int (F(\vec{\y},\vec\sigma;u)-F(\vec y,\vec\sigma;u-1)H^{m+1}_{a,b}(\vec y,\vec\sigma)\,\d\mu_\rho.
\end{multline}
Furthermore, in the periodic case with period $L$, for any $F\in \mathscr{D}(\SL_L)$,   $\int \SL_L F \d\mu_{L,\rho}=0$.
\end{prop}

\begin{proof} 
Throughout the proof it will be convenient to think of a configuration $(\vec y,\vec\sigma)$ chosen from $\mu_\rho$ as follows: sample a single Poisson point process $\Xi$ with intensity $r$, order the points $y_i$, $i\in \zz$, and then for each $i$ flip a coin to give $\sigma_i=1$ with probability $p$ and $\sigma_i=-1$ with probability $1-p$.
For this to match the above description we need to choose $r$ and $p$ so that $rp=\rho$ and $r(1-p)=\rho^{-1}$.
The parameters $r$ and $p$ are dependent, but we will keep them in because it clarifies the computation.  

Since $F$ is supported on $(a,b)$, we only need to focus on the particles of $\Xi$ inside that interval.
The number $n$ of those particles has distribution Poisson$[r(b-a)]$ and their positions $y_1 \leq \cdots \leq y_n$ are the ordering of independent uniformly distributed variables on  $[a,b]$.
The event in the indicator function defining $H^m_{a,b}$ means that there are $k$ pluses and $\ell$ minuses with $n=k+\ell$, $m=k-\ell$.
From this it follows after a computation, writing $D_{k+\ell}$ for the domain $\{a \leq \y_1 < \cdots < \y_{k+\ell} \leq b\}$, that 
\begin{equation}\label{eq:intFmu}
\int F(\vec{\y},\vec\sigma;u)H^{m}_{a,b}(\vec y,\vec\sigma)\,\d\mu_\rho(\vec{\y},\vec\sigma) = \sum_{\substack{k,\ell\geq0:\\k-\ell=m}}\upsilon_{k,\ell} \sum_{\substack{\vec\sigma \in \{-1,1\}^{k+\ell}:\\|\{i:\ts\sigma_i=1\}|=k}}\int_{D_{k+\ell}} F_{k+\ell}(\vec{y},\vec{\sigma};u)\,\d {\y_1} \dotsm \d {\y_{k+\ell}}
\end{equation}
where $\upsilon_{k,\ell} =  e^{- r(b-a)}r^{k+\ell}p^k(1-p)^{\ell}$ (and where we used that the volume of $D_{k+\ell}$ is $\frac{(b-a)^{k+\ell}}{(k+\ell)!}$).
To simplify notation below, we write $F_{k+\ell}$ for $F_{k+\ell}(\vec{y},\vec{\sigma};u)$, omit the $\d {\y_1} \cdots \d {\y_{k+\ell}}$, and write the $k,\ell$ and $\vec\sigma$ sums  as $\sum_{k,\ell,\vec\sigma}$.
Using this we have
\begin{equation}\label{eq:stoch1}
\int \SL_{\gcr} F(\vec{\y},\vec\sigma;u)H^{m}_{a,b}(\vec y,\vec\sigma)\,\d\mu_\rho(\vec{\y},\vec\sigma)=2\sum_{k,\ell,\vec\sigma}\upsilon_{k,\ell} \int_a^b\d z\int_{D_{k+\ell}}\bigl(F_{k+\ell+2}(C_z(\vec{\y},\vec{\sigma});u)- F_{k+\ell}(\vec{\y},\vec{\sigma};u)\bigr)
\end{equation}
for the stochastic part of the dynamics \eqref{eq:L-cr}, while for the deterministic part \eqref{eq:L-det} we have
\begin{equation}
\int\SL_{\gdet}F (\vec\y,\vec\sigma;u)H^{m}_{a,b}(\vec y,\vec\sigma)\,\d\mu_\rho(\vec\y,\vec\sigma)= \sum\nolimits_{k,\ell,\vec\sigma}\upsilon_{k,\ell} \sum_{i=1}^{k+\ell}(-\sigma_i)\int_{D_{k+\ell}}\tsm\p_{\y_i}F_{k+\ell}.\label{eq:bozz}
\end{equation}
We will rewrite each of these two expressions, starting with the deterministic part.

Integrating each $\p_{y_i}F$ in \eqref{eq:bozz} from $y_{i-1}$ to $y_{i+1}$ we get two terms, one from the interior points (whenever $k+\ell\geq2$),
\begin{equation}
\sum\nolimits_{k,\ell,\vec\sigma}\upsilon_{k,\ell} \Big(-\sum_{i=1}^{k+\ell-1}\sigma_i\int_{D_{k+\ell}}\tsm F_{k+\ell}\cdot\delta_{\y_i=\y_{i+1}}+\sum_{i=2}^{k+\ell}\sigma_i\int_{D_{k+\ell}}\tsm F_{k+\ell}\cdot\delta_{\y_i=\y_{i-1}}\Big),\label{eq:Sdet-pre}
\end{equation}
and one from the outer limits of the two boundary integrals,
\begin{equation}\label{eq:Bdet}
\sum\nolimits_{k,\ell,\vec\sigma}\upsilon_{k,\ell}\Big(-\sigma_{k+\ell}\int_{D_{k+\ell}}\tsm F_{k+\ell}\cdot\delta_{\y_{k+\ell}=b}+\sigma_{1}\int_{D_{k+\ell}}\tsm F_{k+\ell}\cdot\delta_{\y_{1}=a}\Big).
\end{equation}
Changing variables $i\longmapsto i+1$ in the second sum inside the brackets on the right hand side of \eqref{eq:Sdet-pre}, the whole expression in brackets becomes $\sum_{i=1}^{k+\ell-1}(\sigma_{i+1}-\sigma_i)\int_{D_{k+\ell}}\tsm F_{k+\ell}\cdot\delta_{\y_i=\y_{i+1}}$, 
which equals
\begin{equation}
-2\sum_{1\leq i\leq k+\ell-1}\uno{\sigma_i=-\sigma_{i+1}=1}\int_{D_{k+\ell}}\tsm F_{k+\ell}\cdot\delta_{\y_i=\y_{i+1}}
+2\sum_{1\leq i\leq k+\ell-1}\uno{-\sigma_i=\sigma_{i+1}=1}\int_{D_{k+\ell}}\tsm F^{\wedge i}_{k+\ell-2}\cdot\delta_{\y_i=\y_{i+1}}\label{eq:cancellat}
\end{equation}
by the continuity condition \eqref{eq:boundary}, with $F^{\wedge i}_{k+\ell-2}$ referring to $F_{k+\ell-2}$ evaluated at $(\vec{\y},\vec\sigma)$ with the $i$-th and $(i+1)$-th entries of $\vec{\y}$ and $\vec\sigma$ removed.
Now we focus on the second sum in \eqref{eq:cancellat}.
In the $i$-th term, the $\y_i$ and $\y_{i+1}$ integrals only involve the delta function, so they yield a factor $\y_{i+2}-\y_{i-1}$, where we set $y_0=a$ and $y_{k+\ell+1}=b$ for convenience.
In other words, the second sum in \eqref{eq:cancellat} equals
\[2\sum_{1\leq i\leq k+\ell-1}\uno{-\sigma_i=\sigma_{i+1}=1}\ts(y_{i+2}-y_{i-1})\int_{D^{\wedge i}_{k+\ell}}\tsm F^{\wedge i}_{k+\ell-2},\]
where the domain $D^{\wedge i}_{k+\ell}$ is now $\{a\leq y_1<\dotsm<y_{i-1}<y_{i+2}<\dotsm<y_{k+\ell}\leq b\}$.
If we now perform the summation over $\vec\sigma$ on this term, the sums over $\sigma_i$ and $\sigma_{i+1}$ disappear (the integrand does not depend on these variables) and then, after relabeling variables, the remaining sum over $i$ is telescopic and the result is (recalling that this term is present only for $k+\ell\geq2)$
\[(b-a)\uno{k+\ell\geq2}\sum\nolimits_{k-1,\ell-1,\vec\sigma}\int_{D_{k+\ell-2}}\tsm F_{k+\ell-2}.\]
Using this above shows that \eqref{eq:Sdet-pre} is given by
\begin{equation}
-2\sum\nolimits_{k,\ell,\vec\sigma}\upsilon_{k,\ell}\sum_{\substack{1\leq i\leq k+\ell-1:\\\sigma_i=-\sigma_{i+1}=1}}\int_{D_{k+\ell}}\tsm F_{k+\ell}\cdot\delta_{\y_i=\y_{i+1}}
+2r^2p(1-p)(b-a)\sum\nolimits_{k,\ell}\upsilon_{k,\ell}\int_{D_{k+\ell}}\tsm F_{k+\ell},\label{eq:exprLdet}
\end{equation}
where in the second term we have changed variables $(k,\ell)\longmapsto(k+1,\ell+1)$ in the sum, giving the $r^2p(1-p)$ prefactor. 
We conclude that 
\begin{equation}\label{eq:exprLdetcr}
\int\! \SL_{\gdet} F(\vec y,\vec\sigma;u)H^{m}_{a,b}(\vec y,\vec\sigma)\,\d\mu_\rho(\vec y,\vec\sigma)=\eqref{eq:Bdet}+\eqref{eq:exprLdet}.
\end{equation}

Now consider the right hand side of \eqref{eq:stoch1}, and split it as the difference of two expressions.
In the second one the $z$ integral simply yields an extra factor $b-a$, giving $\frac{1}{r^2p(1-p)}$ times the second term in \eqref{eq:exprLdet}.
The first one can be rewritten according to the position $z$ of the newly created pair among the $\y_i$'s as 
\begin{equation}
    2\sum\nolimits_{k,\ell,\vec\sigma}\upsilon_{k,\ell}\sum_{\substack{1\leq i\leq k+\ell+1\\\ts\sigma_i=-\sigma_{i+1}=1}}\int_{D_{k+\ell+2}}F_{k+\ell+2}\cdot\delta_{\y_i=\y_{i+1}},\label{eq:newpair}
\end{equation}
and after changing variables $(k,\ell)\longmapsto(k-1,\ell-1)$ we get $r^2p(1-p)$ times the first term in \eqref{eq:exprLdet}.
Using now $r^2p(1-p)=1$ we deduce from this that $\int\! \SL_{\gcr} F(\vec y,\vec\sigma;u)H^{m}_{a,b}(\vec y,\vec\sigma)\,\d\mu_\rho(\vec y,\vec\sigma)$ equals minus \eqref{eq:exprLdet}, and hence from \eqref{eq:exprLdetcr} that
\begin{equation}
\int\! \SL F(\vec y,\vec\sigma;u)H^{m}_{a,b}(\vec y,\vec\sigma)\,\d\mu_\rho(\vec y,\vec\sigma) = \eqref{eq:Bdet}.
\end{equation}
 
Although we did not introduce all the necessary notation for the periodic case, it should be clear from the computations that in that case we would get $\int \SL_L F(\vec{\y},\vec\sigma;u)\,\d\mu_{L,\rho}(\vec{\y},\vec\sigma)=0$.

It only remains to rewrite \eqref{eq:Bdet}.
Consider the sum corresponding to the first term in the brackets in that expression.
Here $ F_{k+\ell}\cdot\delta_{\y_{k+\ell}=b}$ is being integrated against the measure $\upsilon_{k,\ell}\uno{a\leq \y_1<\dotsm<\y_{k+\ell}\leq b}$ over $k,\ell\geq0$ with $k-\ell=m$ and over choices of $\vec\sigma\in\{-1,1\}^{k+\ell}$ with $k$ of the jumps being up.
This measure simply encodes the jump times and steps of the random walk $\fN$ (with up jumps at rate $\rho$ and down jumps at rate $\rho^{-1}$) on the event that it goes from $0$ at time $a$ to $m$ at time $b$.
In the case $\sigma_{k+\ell}=1$,  the delta function in the integrand corresponds to the last jump being up, and taking place at time $b$.
We may think of this as the walk going to $m-1$ instead of $m$, with $k-1$ up jumps (note here $k\geq1$ because otherwise $\sigma_{k+\ell}$ cannot take the value $-1$).
Hence (and since $F$ is supported on the open interval $(a,b)$) this term can be written as $\int F(\vec y,\vec\sigma;u)H^{m-1}_{a,b}(\vec y,\vec\sigma)\,\d\mu_\rho(\vec y,\vec\sigma)$.
If $\sigma_{k+\ell}=-1$, a similar thing happens, except that now we need to multiply by the indicator function  $H^{m+1}_{a,b}$.
This gives the first terms in each of the integrals on the right hand side of \eqref{invstat}.
For the other term on the right hand side of \eqref{eq:Bdet}, which involves a $\delta_{\y_1=a}$ factor, we proceed similarly and get terms which involve the walk going from $\pm1$ to $m$.
In order to express them in terms of integrals containing indicator functions $H^{m\pm1}_{a,b}$ we need to shift the path associated to $u$ and $(\vec{\y},\vec\sigma)$ to paths going from $0$ to $m\mp1$, and to this end we need to adjust the absolute height parameter $u$ to $u\pm1$, yielding the other two terms on the right hand side of \eqref{invstat}.
\end{proof}

\begin{proof}[Proof of Prop. \ref{prop:invm}]
First of all, because of the explicit construction described at the beginning of this section, the martingale problem for $\SL_L$ is well-posed (Thms. 4.1 and 2.6 of \cite{ethierKurtz}) and therefore by Echeverria's theorem (Prop. 9.2 of \cite{ethierKurtz}), a measure is invariant for the process if for all $f$ in a core for $(\SL_L,\mathscr{D}(\SL_L))$,
$\int \SL_L f \d\mu =0$.
This is the second statement of Prop. \ref{prop:invm-general}, so we know that each $\mu_{L,\rho}$ is invariant for the process on $[-L,L)$ with periodic boundary conditions.  Now start the full line $(\vec{y},\vec{\sigma})$ process with measure $\mu_\rho$ and let $F$ be any local function and $t>0$.  We run the periodic process on $[-L,L)$ using the same noise as the full line process and using the same initial data, all restricted to $[-L,L)$.
If $L$ is chosen large enough depending on the range of $F$ and $t$, then the configuration of the two processes at time $t$ coincides within the range of $F$, because of the finite speed of propagation.   
But we have shown that the distribution of the periodic process is $\mu_{L,\rho}$.
Therefore for the full line process $\ee[ F(\vec{y}(t),\vec{\sigma}(t))] =\ee_{\mu_\rho}[ F]$.
Since local functions $F$ are enough to determine the measure $\mu_\rho$, we have proven that $\mu_\rho$ is invariant. 
\end{proof}

\subsection{Adjoint generator}\label{sec:adjoint}

Our next goal is to compute the adjoint of $\SL$ with respect to the  $ L^2(E^{a,b},\mu_\rho)$ inner product (for some fixed $\rho>0$).
Recall that we are interpreting elements of $E^{a,b}$ as encoding the jumps of an upper semi-continuous function, and endowing it with the topology that it inherits from $\UC$.
Let us now interpret instead the elements of $E^{a,b}$ as the jumps of a lower semi-continuous function $g\in\LC=\{f\!:-f\in\UC\}$,
with the topology of local Hausdorff convergence of epigraphs. 
A function $G\!:E^{a,b}\longrightarrow\rr$ is therefore continuous  if its restrictions to $E^{a,b}_n$ satisfy the \emph{dual matching conditions}
\begin{equation}\label{eq:boundary2}
\lim_{\y_{i}-\y_{i-1}\searrow 0}G_n( \ldots,(\y_{i-1},+1),(\y_{i},-1),\ldots)= G_{n-2}( \ldots, (\y_{i-2},\sigma_{i-2}), (\y_{i+1},\sigma_{i+1}),\ldots).
\end{equation}

Define
\begin{equation}\label{eq:L-cr-star}
\SL^*_{\gcr} G_n(\vec{x},\vec{\sigma}) = 2 \int_{-\infty}^\infty \bigl(G_{n+2}( C^*_x(\vec{\y},\vec{\sigma}) )- G_n(\vec{\y},\vec{\sigma} )\bigr) \d\vec y,
\end{equation}
where the configuration $C^*_z(\vec{\y},\vec{\sigma}) $ is obtained by inserting $(z, -1)$ and $(z, +1)$ into $(\vec{\y},\vec{\sigma})$.
We will also use $\nu_\rho$ to denote the product measure of counting measure on the height $h(0)\in\zz$ at the origin with the invariant measure $\mu_\rho$, $\rho>0$.

\begin{prop}\label{prop:adjoint-gen}
\leavevmode
\begin{enumerate}[label=\uptext{\arabic*.}]
\item  In the above setting, the adjoint $\SL^*$ of $\SL$ on $L^2(\mu_\rho)$  is given by
\begin{equation}\label{eq:L-adj}
\SL^* = - \SL_{\gdet} + \SL^*_{\gcr}.
\end{equation}
\item  Let $\hat{\SL}^*$ be the adjoint of $\hat{\SL}$ on $L^2(\nu_\rho)$.  
Then $\hat{\SL}^*$ is the generator of the backwards PNG process $h^{\gets}(t,x)$ introduced after \eqref{varform2}.
\end{enumerate}
\end{prop}

\begin{proof}
The second statement is direct, so we focus on the first one.
Assume $F$ and $G$ have common support $(a,b)$.
Proceeding analogously to \eqref{eq:intFmu}, integration by parts shows that $\int (\SL_{\gdet} F) G\,\d\mu_\rho+\int F (\SL_{\gdet} G)\,\d\mu_\rho$ equals
\begin{multline}
\textstyle\sum_{k,\ell=0}^\infty \sum_{\vec\sigma \in \{-1,1\}^{k+\ell}}\upsilon_{k,\ell}\left[\sum_{i = 2}^{k+\ell} \sigma_i \int_{\y_{i-1} = \y_{i}} F_{k+\ell} G_{k+\ell}-\sum_{i = 1}^{k+\ell-1} \sigma_i \int_{\y_i = \y_{i+1}} F_{k+\ell} G_{k+\ell}\right]\\
\textstyle=\sum_{k,\ell=0}^\infty \sum_{\vec\sigma \in \{-1,1\}^{k+\ell}}\upsilon_{k,\ell}\sum_{i = 1}^{k+\ell-1} (\sigma_{i+1}-\sigma_i) \int_{\y_i = \y_{i+1}} F_{k+\ell} G_{k+\ell}.
\end{multline}
The terms on the right hand side with $\sigma_i=\sigma_{i+1}$ vanish, while those with $\sigma_i\neq\sigma_{i+1}$ correspond to a $(\sigma_i,-\sigma_i)$ pair at $\y_i$, and thus using \eqref{eq:boundary} and \eqref{eq:boundary2} we get that the above equals
\begin{equation}
\textstyle\sum_{k,\ell=1}^\infty \sum_{\vec\sigma \in \{-1,1\}^{k+\ell}}\upsilon_{k,\ell}
\sum_{i = 1}^{k+\ell-1}\!\int_{\y_i = \y_{i+1}}\! \big(F^{\wedge i}_{k+\ell-2} G_{k+\ell}\tts\uno{\sigma_{i+1}=-\sigma_i=1}-F_{k+\ell} G^{\wedge i}_{k+\ell-2}\tts\uno{\sigma_i=-\sigma_{i+1}=1}\big),
\end{equation}
where we are using the notation from the proof of Prop. \ref{prop:invm-general}.
We have restricted the sum in $k$ and $\ell$ to start from $1$, because the indicator functions inside the integral imply that there has to be at least one up jump and at least one down jump. 
Now we may reinterpret the integral $\sum_{i = 1}^{k+\ell-1}\int_{\y_i = \y_{i+1}}F^{\wedge i}_{k+\ell-2} G_{k+\ell}\uno{\sigma_{i+1}=-\sigma_i=1}$ as $\int_a^b\left(\int F_{k+\ell-2}G_{k+\ell}(C^*_z(\vec{\y},\vec\sigma))\right)\tsm\d z$, where the inner integral is over $k+\ell-2$ variables.
The analogous thing can be done for the other term in the integral and then, after changing $(k,\ell)\longmapsto(k-1,\ell-1)$ and using that the change of measure $r^2p(1-p)=1$, we get that the above expression equals
\[\textstyle2\sum_{k,\ell=0}^\infty \sum_{\vec\sigma \in \{-1,1\}^{k+\ell}}\upsilon_{k,\ell}\int\left(\int\big(F_{k+\ell}G_{k+\ell+2}(C^*_z(\vec{\y},\vec \sigma))-F_{k+\ell+2}(C^*_z(\vec{\y},\vec \sigma))G_{k+\ell})\right)\d z.\]
From \eqref{eq:L-cr} and \eqref{eq:L-cr-star}, the right hand side is exactly
$-\int\tsm \big((\SL_{\gcr}^* F) G - F (\SL_{\gcr}^* G)\big) \d\mu_\rho$.
\end{proof}

The computation made in Prop. \ref{prop:invm-general} for integration of $\SL$ against the pinned invariant measure can be repeated for the adjoint generator $\SL^*$.
This is the version of the result we will actually need to use below:

\begin{prop}\label{prop:invm-general-adjoint}
In the context of Prop. \ref{prop:invm-general}, suppose now that $F$ is continuous in $\LC$.
Then
\begin{multline}\label{invstat-adjoint}
\int \SL^* F(\vec{\y},\vec\sigma;u)H^m_{a,b}(\vec y,\vec\sigma)\,\d\mu_\rho(\vec{\y},\vec\sigma)=\int (F(\vec{\y},\vec\sigma;u)-F(\vec{\y},\vec\sigma;u+1))H^{m-1}_{a,b}(\vec y,\vec\sigma)\,\d\mu_\rho(\vec{\y},\vec\sigma)\\
-\int (F(\vec{\y},\vec\sigma;u)-F(\vec y,\vec\sigma;u-1)H^{m+1}_{a,b}(\vec y,\vec\sigma)\,\d\mu_\rho(\vec{\y},\vec\sigma).
\end{multline}
\end{prop}

The proof is the same as that of Prop. \ref{prop:invm-general}, using the formula $\SL^* = - \SL_{\gdet} + \SL^*_{\gcr}$ and the fact that $F$ now satisfies the dual conditions \eqref{eq:boundary2}, which together lead to the same expression as in the right hand side of \eqref{invstat} with a minus sign in front.

\subsection{Backward equation}\label{sec:backw}

It would be convenient if we could conclude from \eqref{HYbackeq} that our transition probabilities $F(t,h)$, defined by 
\begin{equation}\label{mtp2}
F(t,h; \vec{x},\vec{r}) = \pp_{h}( h(t,x_i)\le r_i, i=1,\ldots,n)= \PPP_t \varphi(h),\quad \varphi= \uno{h(x_i)\le r_i, i=1,\ldots,n},
\end{equation}
satisfy the backward equation
$\partial_t F= \hat\SL F$.
Unfortunately, this function $\varphi$ is \emph{not} continuous on $\UC$, and it does not follow.  In fact, even for $t>0$, $F(t,h)\not\in \mathscr{D}(\hat\SL)$ and the backward equation cannot hold in the classical sense.  
However, we are forced to deal with them as it is only this particular class of functions for which we have exact formulas.
We now show that the backward equation is satisfied in a weak sense.

It is not too hard to see that if $h$ corresponds to $(y_i,\sigma_i)$ then the transition probabilities \eqref{mtp2}, for given $x_1,\dotsc,x_n$, can have step discontinuities on the \emph{shock set},
\begin{equation}
\mathcal S (\vec{x})= \{(\vec{y},\vec{\sigma}, t)\!:\, y_i- \sigma_i t = x_j\uptext{ for some }i,j\}.
\end{equation}
Note that this set depends on $\vec{x}$.
It is essentially just a finite collection of lines of slope $\pm 1$ and its complement is open. We will refer to any such object below, for some finite $\vec{x}$, as ``a shock set''.
One can also think of it in terms of the height function $h$.
It is the collection of $h$ with an up jump at some point $x_k+t$ or a down jump at some point $x_k-t$, for some $k$.
At such points, $\p_t F$ and $\p_{y_i}F$ are ill-defined for particular values of the $r_i$'s and it is impossible for $\p_tF=\hat{\SL}F$ to hold in the classical sense.
On the other hand, it is always true at such points that the Dirac delta functions in $(\p_t - \p_{y_i})F$ cancel, leaving at most a bounded singularity.
In other words, $F$ satisfies $(\p_t-\hat{\SL})F=0$ off a co-dimension $1$ set $\mathcal S(\vec{x})$ where it is still well behaved, and this can be used to give meaning to the backward equation for $F$.
For our purposes, the following weak version of the backward equation is the most convenient:

\begin{prop}\label{prop:4.5} 
$F$ given by \eqref{mtp2} is a weak solution of the backward equation in the following sense:
for any $\varphi=\varphi(s,h)$ satisfying 
1. $\varphi$ is continuously differentiable off a shock set,
2. $\varphi(s,h)=0$ for all $s\in[0,t]$ if $\max_ir_0(t,x_i)\le -K$ for some $K>0$,
3. $|\varphi(s,h)|$ and $|(\p_s+\hat{\SL}^*)\varphi(s,h)|$ are bounded uniformly in $s\in[0,t]$ and $h$ in the support of $\nu_\rho$,
we have
\begin{equation}\label{ibp3}
 \int_0^t\int (-\p_s - \hat{\SL}^*)\varphi F\ts\d\nu_\rho\ts\d s = -\int \varphi (t) F(t)\ts\d\nu_\rho +\int \varphi (0) F(0)\ts\d\nu_\rho.
\end{equation}
\end{prop}

\begin{proof}
Fix $\vec{x},\vec{r}$ and recall $\nu_\rho$ is the product of counting measure on $h(0)\in\zz$ with the invariant measure $\mu_\rho$. The function $f_\ep(h)=(2\ep)^{-n}\int_{|z_i- x_i|<\ep,\tts i=1,\dotsc,n}\uno{h(z_i)\le r_i,\tts i=1,\ldots,n}\ts\d\vec{z}$ is in $\mathscr{C}_0$ and so, by the general theory described at the beginning of Sec. \ref{PNGMarkov},  $\PPP_t f_\ep= \ee_h[f_\ep(h(t))]$ does satisfy $(\p_t-\hat\SL) \PPP_t f = 0$.   Integrating by parts, \eqref{ibp3} holds with $F$ replaced by $\PPP_t f_\ep$. Furthermore, from the explicit description of the process at the beginning of Sec. \ref{PNGMarkov}, it is clear that $\PPP_t f_\ep(h) \longrightarrow F(t,h)= F(t,h; \vec{x},\vec{r})$ as $\ep\searrow0$ uniformly in $h$ and bounded sets of $t$.
Therefore, taking $\ep\searrow0$ in the (weak) backward equation \eqref{ibp3} for $\PPP_t f_\ep$, we deduce that the same equation holds for $F$ provided that we can pass the limit through the integration.

In order to justify this, the only real issue is the sum over $h(0)\in\zz$ implicit in $\nu_\rho$.
Let $I\subseteq\rr$ be the smallest (finite) interval containing the origin, each $x_i$,
and every point at distance at most $t+1$ from any of these points. Let $\overline h$ and $\underline h$ be the supremum and infimum of $h$ over this interval.
The event $A_k= \{\overline h- \underline h = k\}$ satisfies $\nu_\rho(A_k) \le \tilde C^k/k!$ with a $\tilde C$ depending on $\rho$, $t$ and the $x_i$'s.
On the other hand, on $A_k$, $\varphi(s,h)=0$ for all $s\in[0,t]$ and $h$ with $h(0)+k\leq-K$, since $\overline h\geq\max_ir_0(t,x_i)$.
Similarly, on $A_k$ we have $r_0(s,x_i)$ is bounded below by $h(0)-k$ for each $i$ and each $s\in[0,t]$, and hence that the height function at time $s\in[0,t]$ is also bounded below by $h(0)-k$, so that (by definition) $F(s,h)$ vanishes for $h(0)-k>\overline{r}\coloneqq\max\{r_1,\dotsc,r_n\}$.
This last fact is also valid for $\PPP_tf_\ep(h)$ if $\ep<1$ in view of our choice of dependent region $I$.
Now we consider the left hand side of \eqref{ibp3} with $\PPP_tf_\ep(h)$.
Splitting the $\nu_\rho$ integral  according to the value of $k$, using these facts to truncate the $h(0)$ sum, bounding $F$ by $1$ in the remaining region, and using the assumption to bound $|(\p_s+\hat{\SL}^*)\varphi|$ uniformly, the absolute value of the integrand can be bounded uniformly in $\ep\in(0,1)$ in a way which makes the whole integral be bounded by
\[\textstyle C\tts t\sum_{k\geq0}\sum_{h(0)=-K-k}^{\underline{r}+k}\frac1{k!}\tilde C^k<\infty.\]
The same holds for the right hand side of \eqref{ibp3}, and this shows that we may pass the $\ep\searrow 0$ limit through the integration in \eqref{ibp3} as needed.
\end{proof}

The next theorem says that (a truncated version of) the Fredholm determinant satisfies the backward equation in the same sense.

\begin{thm}\label{thm:Kolmogorov}
Fix $r_1,\dotsc,r_n\in\zz$, $x_1,\dotsc,x_n\in\rr$, and $h\in\UC$, and let $K_r$ be the kernel defined in \eqref{eq:matrixKPNG} using these fixed parameters and $h_0=h$.
For $t>0$ let
\begin{equation}
\FFF(t,h)=\det(I-K_r)_{\oplus_n\ell^2(\zz_{>0})}\uno{r_i\geq r_0(t,x_i),\ts i=1,\dotsc,n}\label{eq:Fth}
\end{equation}
  with $r_0$ defined using the given $h$.
Then:
\begin{enumerate}[label=\uptext{(\roman*)}]
\item $\lim_{t\searrow0}\FFF(t,h)=\uno{r_1\geq h(x_1),\dotsc,r_n\geq h(x_n)}$.
\item $\FFF$ satisfies 
  1. $(\p_t-\hat \SL)\FFF=0$ off ${\mathcal S}(\vec{x})$;
  2. $\FFF(s,h)=0$ for all $s\in[0,t]$ in a finite interval if $\inf_{i=1,\dotsc,n}\inf_{|z-x_i|\leq t}h(z)$ is sufficiently large;
  3. $|\FFF(s,h)|$ and $|(\p_s-\hat{\SL})\FFF(t,h)|$ are bounded uniformly in $s\in[0,t]$ and $h$ in the support of $\nu_\rho$.
\item  For any \emph{admissible} $\varphi$ (meaning $\varphi$ satisfies 1, 2, and 3 of Prop. \ref{prop:4.5}), 
\begin{equation}\label{ibp4}
 \int_0^t\int (-\p_s - \hat{\SL}^*)\varphi\hat F \ts\d\nu_\rho\tts\d s = -\int \varphi (t) \hat F(t)\ts\d\nu_\rho +\int \varphi (0) \hat F(0)\ts\d\nu_\rho.
\end{equation}
\end{enumerate}
\end{thm}

Note that we have set $\FFF(t,h)=0$ in \eqref{eq:Fth} when $r_i<r_0(t,x_i)$ for some $i$.
That $\FFF$ should be zero in this region is clear physically: our ultimate goal is to show that $\FFF(t,h)$ corresponds to the transition probabilities for the PNG height function (with initial data $h$), and these probabilities are $0$ whenever $r_i<r_0(t,x_i)$ for some $i$.
We need to introduce this explicit truncation because some of the computations on $K_r(u,v)$ which we will perform to prove the backward equation become singular for $u$ or $v$ near or below one of the $r_0(t,x_i)$'s.
We conjecture, however, that the Fredholm determinant $\det(I-K_r)_{\oplus_n\ell^2(\zz_{>0})}$ is already $0$ if $r_i<r_0(t,x_i)$ for some $i$.
Proving such a statement in general for a Fredholm determinant is hard, and we do not attempt it here.
Alternatively, one could attempt to prove this by taking a scaling limit of formulas for parallel TASEP, which we will do in an upcoming paper.

On the other hand, it is easy to see that the Fredholm determinant $\det(I-K_r)_{\oplus_n\ell^2(\zz_{>0})}$ as $t\searrow 0$ satisfies the initial condition.
From \eqref{eq:hatKh0ab}, we have
$e^{-x_i\Delta}\J^{h}_{x_i,x_j}e^{x_j\Delta}=P^{\hit(h)}_{x_i,x_j}$
(recall here that, by convention, the right hand side is zero if $x_i>x_j$), and 
using this in \eqref{eq:PNG-kernel-two-sided-hit} yields
\begin{equation}\label{eq:KPNG0}
K^{\uptext{ext}}\big|_{t=0}(x_i,\cdot;x_j,\cdot)=-P^{\nohit(h)}_{x_i,x_j}\uno{x_i<x_j}+\bP_{h(x_i)}\uno{x_i=x_j}.
\end{equation}
By Prop.~\ref{prop:trcl}, $K_r(\ep,h)\longrightarrow K_r(0,h)$ as $\ep \to 0$ in trace norm in $\oplus_n\ell^2(\zz_{>0})$, after conjugation, while from \eqref{eq:matrixKPNG} and \eqref{eq:KPNG0} we have that $K(0,h)$ is upper triangular with the $i$-th entry in the diagonal equal to $\bP_{h(x_i)-r_i}$.
Hence, since the Fredholm determinant is continuous in trace norm, after conjugating as in Appdx. \ref{sec:trcl-ker} we deduce that
\begin{equation}
    \textstyle\det(I-K(\ep,h))\xrightarrow[\ep\to0]{}\det(I-K(0,h))=\prod_{i=1}^n\det(I-\bP_{h(x_i)-r_i})_{\ell^2(\zz>0)}=\prod_{i=1}^n\uno{h(x_i)\leq r_i}.\label{eq:inidet}
\end{equation}
This yields the initial condition (i) for $\hat F$: if some $r_i<h(x_i)$ then $r_i<r_0(t,x_i)$ for all $t>0$ and the condition holds trivially, while if $r_i\geq h(x_i)$ for all $i$ then by upper semi-continuity of $h$ we have $r_i\geq r_0(\ep,x_i)$ for small enough $\ep$, and hence for such $\ep$ we have $\FFF(\ep,h)=\det(I-K_r(\ep,h))$ and the initial condition is provided by \eqref{eq:inidet}.

Before we turn to the proof of \eqref{ibp4} in the next subsection, let us see how it proves Thm. \ref{thm:PNG-fred}.  

\begin{prop}\label{prop:4.7}
For all finite $\vec{x}$, $\vec{r}$ and all $t>0$, $F(t,h)=\hat F(t,h)$ $\nu_\rho$-a.e.
\end{prop} 

\begin{proof} 
From \eqref{ibp3} and \eqref{ibp4}, for any $\varphi$ which is admissible (i.e. satisfying 1, 2, 3 of Prop.  \ref{prop:4.5}), 
\begin{equation}\label{ibp5}
\textstyle\int_0^t\int (-\p_s - \hat{\SL}^*)\varphi (F-\hat F)\,\d\nu_\rho\,\d s = -\int \varphi (t) (F(t)- \hat F(t))\,\d\nu_\rho. 
\end{equation}  
Let $m>n$.
 For any $y_1,\ldots,y_m$ containing $x_1,\ldots,x_n$, and $q_1,\ldots,q_m$ we can construct an admissible $\varphi$ with 
$(\p_s + \hat{\SL}^*)\varphi =0$ off its shock set and  $\varphi(t,h)=\uno{h(y_i)= q_i, i=1,\ldots, m}$.
This is done by solving the model backwards in time using Thm. \ref{thm:Kolmogorov} and skew time reversal invariance \eqref{eq:timerev2} to produce a Fredholm determinant solution of $(\p_s + \hat{\SL}^*)\tilde\varphi =0$ off the shock set with final condition $\uno{a_1\leq h(y_1),\dotsc,a_m\leq h(y_m)}$ for any $a_1,\dotsc,a_m$.
Thm. \ref{thm:Kolmogorov} ensures that these $\tilde\varphi$'s are admissible, and since the class of admissible $\varphi$ is closed under finite sums and differences, we can use them to produce such a $\varphi$ with final data of the form $\uno{h(y_i)= q_i, i=1,\ldots, m}$.
The reason to include $x_1,\ldots,x_n$ among the $y_i$'s is to ensure that condition 2 of Prop. \ref{prop:4.5} holds. 
For this choice of $\varphi$ \eqref{ibp5} becomes
\begin{equation}
\textstyle\int_{h(y_i)= q_i, i=1,\ldots, m}\big(F(t,h)-\hat F(t, h)\big)\d\nu_\rho=0,
\end{equation}
and now the fact that this holds for any $m>n$ and any $y_1,\dotsc,y_m$ (including the $x_i$'s) and $q_1,\dotsc,q_m$ implies the theorem.   
\end{proof}

\begin{proof}[Proof of Thm.  \ref{thm:PNG-fred}]
The previous proposition shows that \eqref{fred} holds whenever $h$ is in the support of $\nu_\rho$.
To go to general $h\in\UC$ we will approximate by functions in this class.
The probabilities \eqref{mtp2} are not continuous with respect to the $\UC$ topology, but $\UC$ has a partial order, $h_1\prec h_2$ if $h_1(x)\le h_2(x)$, and both sides of \eqref{fred} are continuous with respect to local hypograph convergence with decreasing limits (for the Fredholm determinant on the right hand side of \eqref{fred} this is proved in Prop. \ref{prop:trcl}).
Since every $\UC$ function is the decreasing limit of functions in the support of $\nu_\rho$, the result follows.
\end{proof}

\subsection{Hit kernel} \label{sec:whyhit}

The goal here is to provide the details for the argument sketched in Sec. \ref{sec:whyhit1}. 
The precise statement is Thm. \ref{thm:Kolmogorov}.
The key to its proof is the following result, which is mainly a consequence of skew time reversibility together with our invariant measure computations from Secs. \ref{sec:inv} and \ref{sec:adjoint}: 

\begin{prop}\label{prop:LP}
Fix $h\in\hat{\cX}\subseteq\UC$.
For any $a\leq b$ and for $u>h(a)$, $v>h(b)$,
\begin{multline}
\hat{\SL}P^{\hit(h)}_{a,b}(u,v)=2[P^{\hit(h)}_{a,b},\nabla](u,v)-\tfrac12P^{\nohit(h)}_{a^+,b}(h(a),v)\uno{u=h(a)+1}\\
-\tfrac12P^{\nohit(h)}_{a,b^-}(u,h(b))\uno{v=h(b)+1},\label{eq:LP}
\end{multline}
where the bracket denotes the commutator, i.e. $[A,B]=AB-BA$.
\end{prop}

Note that the two boundary terms on the right hand side of \eqref{eq:LP} only arise when $h$ has jump at the boundary of $[a,b]$ (more precisely, a down jump at $a$ in the first case or an up jump at $b$ in the second case), since otherwise the no hit probabilities vanish by upper semi-continuity of $h$.
A key point will be that these terms match those which we obtained in Lem. \ref{lem:pbPhit}.

\begin{proof}[Proof of Prop. \ref{prop:LP}]
Assume first that $u>h(a)+1$ and $v>h(b)+1$.
Let $g$ denote a random path in $[a,b]$ which has the same law as $\fN$ conditioned on $\fN(a) = u$ and $\fN(b) = v$.
We choose $\fN$ to be lower semi-continuous, and extend it as $\infty$ outside the interval $[a,b]$.
Let also $\fE_{a,u;b,v}$ denote the expectation with respect to the law of $g$, and recall the notation $\Phi(g)=1-\uno{g>0}$ from \eqref{eq:def-Phi}.
Then we may write
\begin{equation}
\begin{aligned}
\hat{\SL}_{h}P^{\hit(h)}_{a,b}(u,v)&=\lim_{\delta\searrow 0} \frac{1}{\delta}\ee\Big(P^{\hit(h(\delta,\cdot))}_{a,b}(u,v)-P^{\hit(h(0,\cdot)}_{a,b} (u,v)\Big)\\
&=\lim_{\delta\searrow 0} \frac{1}{\delta}\fE_{a,u;b,v}\Big(\ee\big(\Phi(g-h(\delta,\cdot))-\Phi(g-h(0,\cdot))\big)\Big) e^{(b - a)\Delta}(u,v),
\end{aligned}\label{eq:eefe}
\end{equation}
where $\ee$ denotes the expectation with respect to the PNG dynamics and the PNG height function $h(\delta,\cdot)$ is taken to start, slightly abusing notation, at $h(0,\cdot)=h$, and where we have written $\hat\SL_{h}$ to stress that the generator acts on $h$.
Here we have used Fubini to change the order of expectations, justified by the definite sign of the integrands.
By definition of $\Phi$, the inner expectation equals the probability that the fixed path $g$ goes below $h(\delta,\cdot)$ but not below $h(0,\cdot)=h$.
But since $u>h(a)+1$ and $v>h(b)+1$ and since both $g$ and $h$ have finitely many jumps inside, say, the interval $[a-1,b+1]$, if $\delta$ is small enough then for this to happen there necessarily has to be a nucleation inside that interval before time $\delta$, and since this is an event with probability $\mathcal{O}(\delta)$, by the dominated convergence theorem we can take the limit inside the first expectation to get
\begin{equation}\label{ellhit}\hat{\SL}_{h}P^{\hit(h)}(u,v)=\fE_{a,u;b,v}\Big(\hat\SL_{h}\Phi(g-h)\Big)e^{(b - a)\Delta}(u,v)=\fE_{a,u;b,v}\Big(\hat\SL^*_{g}\Phi(g-h)\Big)e^{(b - a)\Delta}(u,v),\end{equation}
the second equality following from infinitesimal skew time reversal invariance \eqref{infinitesimalstr}.

Here the generator $\hat{\SL}^*$ is the (reversed) generator of the full height process.  We claim that it makes no difference if in the last expression we replace
$\fE_{a,u;b,v}\hat\SL^*_{g}\Phi(g-h)$ by $\fE_{a,u;b,v}\SL^*_{g}\Phi(g-h)$
where 
$\SL^*$ is the (reversed) generator of the height differences, i.e. it does not take account of the absolute height.  This is because to go from $\SL$ to the generator of the full height process, one only has to track the height at a single point $z$, which can be chosen arbitrarily.  If we choose $z$ far from the support of any of our functions, by finite speed of propagation we have  $(\hat{\SL}^*-\SL^*)\Phi=0$.

Thus the last term of \eqref{ellhit} can be rewritten, using the notation from Props. \ref{prop:invm-general} and \ref{prop:invm-general-adjoint}, as
\[\int \SL^*_gF^h(\vec y,\vec\sigma;u)H^{v-u}_{a,b}(\vec y,\vec\sigma)\,\d\mu_1(\vec y,\vec\sigma),\]
where $F^h(\vec y,\vec\sigma;u)$ equals $\Phi$ applied to $g-h$ with $g$ given as the path in $[a,b]$ obtained from the configuration of jump locations and signs $(\vec y,\vec\sigma)$ there and which takes the value $u$ at $a$, continued as $\infty$ outside the interval ($\mu_1$ is the invariant measure for the $(\vec y,\vec\sigma)$ process with a density of pluses and minuses equal to $1$).
We are thus precisely in the setting of Prop. \ref{prop:invm-general-adjoint}, and hence the last integral equals
\begin{multline}
\int (F^h(\vec\y,\vec\sigma;u)-F^h(\vec\y,\vec\sigma;u+1))H^{v-u-1}_{a,b}(\vec y,\vec\sigma)\,\d\mu_1(\vec y,\vec\sigma)\\
\quad - \int (F^h(\vec\y,\vec\sigma;u)-F^h(\vec\y,\vec\sigma;u-1))H^{v-u+1}_{a,b}\,\d\mu_1(\vec y,\vec\sigma).
\end{multline}
In view of our choice of $F^h$, the above expression equals
\[P^{\hit(h)}_{a,b} (u,v-1)-P^{\hit(h)}_{a,b} (u+1,v)-P^{\hit(h)}_{a,b} (u,v+1)+P^{\hit(h)}_{a,b} (u-1,v)=2[P^{\hit(h)}_{a,b},\nabla](u,v),\]
proving \eqref{eq:LP} in the case $u>h(a)+1$, $v>h(b)+1$.

If $u=h(a)+1$ or $v=h(b)+1$ then we need to be a bit more careful.
The issue arises when $h$ has a jump at the edge of the interval (just as in Lem. \ref{lem:pbPhit}).
Assume for instance that $v=h(b)+1$ and that $h$ has an up jump at $b$ (it is not hard to see that the above argument applies if there is no such jump).
Then if the fixed path $g$ inside the first expectation is such that $g(s)-h(s)>0$ for $s\in[a,b)$ and $g(s)=h(b)$ for some $s\in[b-\delta,b)$, we clearly have $\Phi(g-h(\delta,\cdot))-\Phi(g-h(0,\cdot))=1$ due to the leftward movement of the up jump in the PNG dynamics, so interchanging the limit and the expectation as we did after \eqref{eq:eefe} becomes more delicate.
The solution in this case is simply to compute the action in $\hat{\SL}$ of the derivative term coming from $\SL_{\gdet}$ corresponding to the movement of that jump of $h$ separately.
This can be done as in the proof of Lem. \ref{lem:pbPhit}, noting that, in this computation, moving the up jump at $b$ slightly to the left has the same effect as shifting $b$ slightly to the right; as in the lemma, this derivative then yields a term of the form
\begin{equation}
P^{\hit(h)}_{a,b^-}\Delta(u,h(b))+\tfrac12P^{\nohit(h)}_{a,b^-}(u,h(b))
=P^{\hit(h)}_{a,b}\Delta(u,h(b))-\tfrac12P^{\nohit(h)}_{a,b^-}(u,h(b)),\label{eq:Phitbdry}
\end{equation}
where in the equality we used the identities $P^{\hit(h)}_{a,b}(u,h(b))-P^{\hit(h)}_{a,b^-}(u,h(b))=P^{\nohit(h)}_{a,b^-}(u,h(b))$ and $P^{\hit(h)}_{a,b}(u,h(b)\pm1)=P^{\hit(h)}_{a,b^-}(u,h(b)\pm1)$.
The first term on the right hand side of \eqref{eq:Phitbdry} corresponds to the contribution of the corresponding term in $\SL_{\gdet}$ ignoring the boundary effect, and comes together with the rest of $\SL$ to produce $2[P^{\hit(h)}_{a,b},\nabla](u,v)$ as above.
The second term on the right hand side of \eqref{eq:Phitbdry}, on the other hand, yields the last term on the right hand side of \eqref{eq:LP}.
A similar computation handles the case $u=h(a)$ when $h$ has a down jump at $a$, and yields the additional boundary term in \eqref{eq:LP}.
\end{proof}

Before turning to the proof of the backward equation we need two additional results.
The first one will be useful in handling the creation part of the generator:

\begin{lem}\label{lem:cr-rk1}
Let $h\in\UC$, $t\geq0$, and $x_1<\dotsm<x_n$, and let $K(h)$ denote the PNG kernel \eqref{eq:PNG-kernel-two-sided-hit} defined using the given $x_i$'s, $t$, and $h_0=h$.
Then for any fixed $z\in\rr$, $K(h+\uno{z})-K(h)$ is rank one.
\end{lem}

\begin{proof}
In order to compare the two kernels we need to look at the hit probabilities involving $h$ and $h+\uno{z}$.
We have
\[P^{\hit(h+\uno{z})}_{x_i-t,x_j+t}=P^{\hit(h)}_{x_i-t,x_j+t}+P^{\nohit(h)}_{x_i-t,z}\delta_{h(z)+1}P^{\nohit(h)}_{z,x_j+t},\]
where for convenience we set $P^{\nohit(h)}_{a,b}=0$ if $a>b$ (so that this covers both the case when $z$ is in $[x_i-t,x_j+t]$ and when it is not).
Then $K(h+\uno{z})(x_i,\cdot;x_j,\cdot)-K(h)(x_i,\cdot;x_j,\cdot)$ equals
\begin{align}
&e^{-2t\nabla-x\Delta}\big(\J^{h+\uno{z}}_{x_i-t,x_j+t}-\J^{h}_{x_i-t,x_j+t}\big)e^{2t\nabla+x'\Delta}
=e^{-2t\nabla-t\Delta}\big(P^{\hit(h+\uno{z})}_{x_i-t,x_j+t}-P^{\hit(h)}_{x_i-t,x_j+t}\big)e^{2t\nabla-t\Delta}\\
&\hspace{1.6in}=e^{-2t\nabla-t\Delta}P^{\nohit(h)}_{x_i-t,z}\delta_{h(z)+1}P^{\nohit(h)}_{z,x_j+t}e^{2t\nabla-t\Delta}
=f_{i}\otimes g_{j}
\end{align}
with $f_{i}(u)=e^{-2t\nabla-t\Delta}P^{\nohit(h)}_{x_i-t,z}(u,h(z)+1)$ and $g_{j}(v)=P^{\nohit(h)}_{z,x_j+t}e^{2t\nabla-t\Delta}(h(z)+1,v)$.
\end{proof}

The second thing we need before proving Thm. \ref{thm:Kolmogorov} is an abstract result about Fredholm determinants.
To motivate it, consider the case $r_i>r_0(t,x)$ for each $i$ and suppose for a moment that we knew that the kernel $K=\P_rK^{\ext}\P_r$ is such that $I-K$ is invertible; this fact will actually follow from Thm. \ref{thm:PNG-fred}, because its Fredholm determinant equals $F(t,x_1,\dotsc,x_n,r_1,\dotsc,r_n)$, which is positive under our assumption on the $r_i$'s (see the argument in Sec. \ref{sec:pfMain}).  However, the proof of Thm. \ref{thm:PNG-fred} depends on Thm. \ref{thm:Kolmogorov}.
Assuming that $I-K$ is invertible we have the standard identity
\[\p_t\det(I-K)=-\det(I-K)\tr((I-K)^{-1}\p_tK),\]
provided $\p_tK$ exists in trace class (in our setting this fails at the shock set, but we ignore this issue in this explanation).
Similarly $\SL_{\gdet}\det(I-K)=-\det(I-K)\tr((I-K)^{-1}\SL_{\gdet}K)$, because $\SL_{\gdet}$ is a first order differential operator.
The key is that the same holds for the creation part thanks to the last lemma and another standard identity: if $A$ is rank one and $I-K$ is invertible, then
\begin{equation}\label{eq:rk-1}
\det(I-K+A)-\det(I-K)=-\det(I-K)\tr((I-K)^{-1}A).
\end{equation}
What is crucial here is that $\p_t$, $\SL_{\gdet}$ and $\SL_{\gcr}$ all give rise to expressions of the same form, leading to
\begin{align}
    (\p_t-\hat\SL)\det(I-K)=(\p_t-\SL)\det(I-K)&=-\det(I-K)\tr((I-K)^{-1}(\p_t-\SL)K)\\
    &=-\det(I-K)\tr((I-K)^{-1}(\p_t-\hat\SL)K),
\end{align}
where we have dropped and then added back the dependence on the absolute height following the argument in the paragraph after \eqref{ellhit}.

In order to achieve this when we do not know a priori that $I-K$ is invertible, we proceed as follows.
Let $\mathcal{H}$ be a separable Hilbert space and let $\mathcal{B}_1$ be the space of trace class operators on $\mathcal{H}$.
Define $\mathsf{D}\!:\mathcal{B}_1\longrightarrow\cc$ through
\begin{equation}
\mathsf{D}(K)=\det(I-K),\label{eq:D-fred}
\end{equation}
where the Fredholm determinant is computed on $\mathcal{H}$.

\begin{lem}\label{lem:d-rk1}
$\mathsf{D}$ is Fr\'echet differentiable with Fr\'echet derivative $\mathsf{D}'(K)$.
In particular, if $K_a\in\mathcal{B}_1$ depends on a parameter $a\in\rr$ and the derivative $\p_aK_a$ exists in trace norm, then
\begin{equation}\label{eq:lemD1}
\p_a\det(I-K_a)=\mathsf{D}'(K_a)(\p_aK_a).
\end{equation}
On the other hand, if $K\in\mathcal{B}_1$ and $A$ is a rank one operator acting on $\mathcal{H}$, then
\begin{equation}\label{eq:lemD2}
\det(I-K-A)-\det(I-K)=\mathsf{D}'(K)(A).
\end{equation}
\end{lem}

The main point of the result (which is simple, but which we did not find stated in the literature) is that \eqref{eq:lemD2} holds without any assumption on $K$.

\begin{proof}
That $\mathsf{D}$ is Fr\'echet differentiable is standard, see e.g. \cite[Cor. 5.2]{simon}, while \eqref{eq:lemD1} follows from that and the chain rule.
The same result in \cite{simon} shows that, in the special case when $I-K$ is invertible, the Fr\'echet derivative of $\mathsf{D}$ has the following form: for each $L\in\mathcal{B}_1$,
\begin{equation}
\mathsf{D}'(K)(L)=-\mathsf{D}(K)\tr((I-K)^{-1}L).\label{eq:rk-2}
\end{equation}
Now fix $K,A\in\mathcal{B}_1$ with $A$ rank one and, for $\mu\in\cc$, define
\[\mathsf{R}(\mu)=\mathsf{D}(\mu K+A)-\mathsf{D}(\mu K)-\mathsf{D}'(\mu K)(A).\]
The three objects on the right hand side are analytic in $\mu$ (for $\mathsf{D}(\mu K)$ this is standard, see e.g. \cite[Lem. 3.3]{simon}, while for the other two it is implicit in Cor. 5.2 of the same book), so $\mu\longmapsto\mathsf{R}(\mu)$ is an analytic function.
On the other hand, if $\mu^{-1}$ is not in the spectrum of $K$ then, by \eqref{eq:rk-1} and \eqref{eq:rk-2},
\[\mathsf{D}(\mu K+A)-\mathsf{D}(\mu K)=-\mathsf{D}(\mu K)\tr((I-\mu K)^{-1}A)=\mathsf{D'}(\mu K)(A).\]
Hence $\mathsf{R}(\mu)=0$ whenever $\mu^{-1}$ is not in the spectrum of $K$.
$K$ is compact, so its spectrum is discrete.
Hence we deduce that $\mathsf{R}(\mu)=0$ for all $\mu\in\cc$ and thus, in particular, for $\mu=1$.
\end{proof}

\begin{proof}[Proof of Thm. \ref{thm:Kolmogorov}]
The initial condition (i) was already verified.
The second point of (ii) is straightforward from the definition \eqref{eq:Fth} of $\FFF$, while (iii) follows directly from 1 in (ii) and integration by parts.
To get the estimate on $|\FFF(s,h)|$ in the third point of (ii) we use Prop. \ref{prop:trcl} together with the bound $|\tsm\det(I-A)|\leq e^{\|A\|_1}$ for any trace class operator $A$ on a given (separable) Hilbert space, with $\|A\|_1$ its trace norm \cite[Lem. 3.3]{simon}.
They give $|\FFF(s,h)|\leq C(1+e^{c\tts\max_i(r_0(s,x_i)-r_i)})\uno{r_i\geq r_0(s,x_i),\ts i=1,\dotsc,n}$ with constants $c,C>0$ valid for all $s\in[0,t]$, which is clearly bounded uniformly as needed (note the key role played here by the truncation).
The estimate on $|(\p_s-\hat\SL)\FFF|(s,h)$ follows from 1 of (ii) off the shock set, and will be proved in Case 3b of this proof on $\mathcal{S}(\vec x)$.

The rest of the proof will be devoted to showing the two remaining conditions: that $(\p_t-\hat\SL)\FFF=0$ off $\mathcal{S}(\vec x)$ and that $|(\p_t-\hat\SL)\FFF(t,h)|$ is bounded uniformly on $\mathcal{S}(\vec x)$ for $t$ in a compact interval and $h\in\UC$.  Throughout we use $\SL$
instead of $\hat\SL$, i.e. we drop the dependence on the absolute height, following the argument in the paragraph after \eqref{ellhit} which tells us that $(\hat{\SL}-\SL)\FFF=0$, $(\hat{\SL}-\SL)K=0$ and $(\hat{\SL}-\SL)P^{\hit(h)}_{a,b}=0$.

We divide the proof in several cases, according to the position of the parameters $r_i$ and whether we are not at a shock.
Throughout the proof all determinants and traces are computed on $\ell^2(\{x_1,\dots,x_n\}\times\zz)$.

\vskip3pt\noindent\emph{Case 1}: Suppose $r_k\geq r_0(t,x_k)$ for each $k$ and write (recalling from \eqref{eq:matrixKPNG} that $K^{\uptext{ext}}(x_i,u;x_j+v)=(K_r)_{ij}(u-r_i,v-r_j)$)
\[K=\P_rK^{\uptext{ext}}\P_r,\qquad\uptext{so that}\quad \FFF=\det(I-K).\]
$K$ only depends on the jumps of $h$ inside $[\min\{x_1,\dotsc,x_n\}-t,\max\{x_1,\dotsc,x_n\}+t]$, so $\SL_{\gdet}$ acts as a sum of finitely many derivatives.
$\p_t-\SL_{\gdet}$ thus acts as a first order differential operator, and from the arguments in Appdx. \ref{sec:trcl-diff} it follows that (the limit which defines) $(\p_t-\SL_{\gdet})K$ exists in trace norm.
Then by \eqref{eq:lemD1} we have
\begin{equation}\label{eq:F-time-deriv}
(\p_t-\SL_{\gdet})\FFF = \mathsf{D}'(K)((\p_t-\SL_{\gdet})K).
\end{equation}
Now we consider ${\SL}_{\gcr}$.
Using \eqref{eq:L-cr} and the notation from Lem. \ref{lem:cr-rk1} we have
\[\textstyle{\SL}_{\gcr}\det(I-K(h))=2\int_{-\infty}^\infty\tsm\big(\tsm\det(I-K(h+\uno{z}))-\det(I-K(h))\big)\tts\d z,\]
and since $K(h+\uno{z})$ is a rank one perturbation of $K(h)$ by the same lemma, Lem. \ref{lem:d-rk1} now shows that the difference in the integrand on the right hand side equals $\mathsf{D}'(K)(K(h+\uno{z})-K(h))$.
This gives
\begin{equation}
\textstyle{\SL}_{\gcr}\det(I-K)=\mathsf{D}'(K)\!\left(2\int_{-\infty}^\infty(K(h+\uno{z})-K(h))\,\d z\right)=\mathsf{D}'(K)(\SL_{\gcr}K).
\end{equation}
From this and \eqref{eq:F-time-deriv} we get $(\p_t-\SL_{\gdet}-\SL_{\gcr})\FFF=\mathsf{D}'(K)((\p_t-\SL_{\gdet}-\SL_{\gcr})K)$, or, in other words,  $(\p_t-\SL)\FFF=\mathsf{D}'(K)((\p_t-\SL)K)$.  

Thus we have shown
\begin{equation}\label{eq:cu0}
(\p_t-{\SL})\FFF=\mathsf{D}'(K)((\p_t-{\SL})K).
\end{equation}
Therefore (recalling $K$ includes the projections $\P_r$) $(\p_t-{\SL})\FFF=0$ will follow in this case if we prove that, for $K_{ij}=K(x_i,\cdot;x_j,\cdot)$, $(\p_t-{\SL})K_{ij}(u,v)=0$ for all $u>r_i$, $v>r_j$ or slightly more generally, due to the assumption $r_i\geq r_0(t,x_i)$, $r_j\geq r_0(t,x_j)$, that
\begin{equation}\label{eq:ptLK22}
(\p_t-{\SL})K_{ij}(u,v)=0\qquad\forall\,u>h(x_i-t),~v>h(x_j+t).
\end{equation}
Note that in this case we are getting $(\p_t-{\SL})\FFF=0$ irrespective of whether we are or not at a shock.

By definition of $K_{ij}$ and \eqref{eq:PNG-kernel-two-sided-hit} we have
\begin{equation}\label{eq:K-time-deriv}
\p_t K_{ij} = e^{- 2 \ft \nabla - \fx_i \Delta} \Bigl( 2 [ \J^{h}_{\fx_i - \ft, \fx_j + \ft}, \nabla] + \p_\ft \J^{h}_{\fx_i - \ft, \fx_j + \ft}\Bigr) e^{2 \ft \nabla + \fx_j \Delta}
\end{equation}
(here we have used the fact that $\nabla$ and $\Delta$, and hence also their exponentials, commute).
Now take $u>h(x-t_i)$, $v>h(x_j+t)$ as in \eqref{eq:ptLK22}.
Lem. \ref{lem:petazeta4} shows that in this case
\[e^{-2\ft\nabla-\fx_i\Delta}\big(\p_\ft \J^{h}_{\fx_i - \ft, \fx_j + \ft}\big)e^{2 \ft \nabla + \fx_j \Delta}(u,v)
=-\uno{u=h(x-t_i)+1}\tts\gamma^1_{t,x_i,x_j}(v)-\uno{v=h(x_j+t)+1}\tts\gamma^2_{t,x_i,x_j}(u),\]
where the boundary terms are given by
\begin{align}
    \gamma^1_{t,x_i,x_j}(v)&=\tfrac12P^{\nohit(h)}_{(x_i-t)^+,x_j+t}e^{2t+2t\nabla-t\Delta}(h(x_i-t),v),\\
    \gamma^2_{t,x_i,x_j}(u)&=\tfrac12e^{2t-2t\nabla-t\Delta}P^{\nohit(h)}_{x_i-t+,(x_j+t)^-}(u,h(x_j+t)),
\end{align}
Using this and the definition of $\J^{h}_{a,b}$ in \eqref{eq:K-time-deriv} we get
\begin{multline}
\p_t K_{ij}(u,v) =  e^{- 2 \ft \nabla - \ft \Delta}\big(2[ P^{\hit(h)}_{\fx_i - \ft, \fx_j + \ft}, \nabla]\big) e^{2 \ft \nabla - \ft \Delta}(u,v)\label{eq:ptKij}\\
-\uno{u=h(x-t_i)+1}\tts\gamma^1_{t,x_i,x_j}(v)-\uno{v=h(x_j+t)+1}\tts\gamma^2_{t,x_i,x_j}(u),
\end{multline}
and we need to prove that the right hand side matches $\SL K_{ij}(u,v)$.
To see this we use the definition of $\J^{h}_{a,b}$ again to write
\begin{equation}
{\SL}K_{ij}=e^{- 2 \ft \nabla - \fx_i \Delta}\big({\SL}\J^{h}_{\fx_i - \ft, \fx_j + \ft}\big)e^{2 \ft \nabla + \fx_j \Delta}=e^{- 2 \ft \nabla - \ft \Delta}\big({\SL}P^{\hit(h)}_{\fx_i - \ft, \fx_j + \ft}\big)e^{2 \ft \nabla - \ft \Delta}.\label{eq:LKij}
\end{equation}
By Lem. \ref{lem:ett} the kernel $e^{2 \ft \nabla - \ft \Delta}$ is upper triangular, and similarly $e^{-2 \ft \nabla - \ft \Delta}=(e^{2 \ft \nabla - \ft \Delta})^*$ is lower triangular, so the restriction on $u$ and $v$ in \eqref{eq:ptLK22} means that in the last expression we only need to consider ${\SL}P^{\hit(h)}_{\fx_i - \ft, \fx_j + \ft}(y,z)$ with $y>h(x_i-t)$, $z>h(x_j+t)$.
But for such $(y,z)$ Prop. \ref{prop:LP} yields 
\begin{multline}
{\SL}P^{\hit(h)}_{\fx_i - \ft, \fx_j + \ft}(y,z)=2[P^{\hit(h)}_{\fx_i - \ft, \fx_j + \ft},\nabla](y,z)
-\tfrac12\uno{y=h(\fx_i-\ft)+1}P^{\nohit(h)}_{(\fx_i-\ft)^+,\fx_j+\ft}(h(\fx_j-\ft),z)\\
-\tfrac12P^{\nohit(h)}_{\fx_i-t,(\fx_j+\ft)^-}(y,h(\fx_j+\ft))\uno{z=h(\fx_j+\ft)+1},
\end{multline}
and hence from \eqref{eq:LKij} and the definitions of $\gamma^1_{t,x_i,x_j}$ and $\gamma^2_{t,x_i,x_j}$ we get that
\begin{multline}
    {\SL}K_{ij}(u,v)=e^{-2\ft\nabla-\ft\Delta}\big(2[P^{\hit(h)}_{\fx_i-\ft,\fx_j+\ft},\nabla]\big)e^{2\ft\nabla-\ft\Delta}(u,v)\\
    -e^{-2\ft\nabla-\ft\Delta}(u,h(\fx_i-\ft)+1)e^{-2t}\gamma^1_{t,x_i,x_j}(v)
    -e^{-2t}\gamma^2_{t,x_i,x_j}(u)e^{2t\nabla-t\Delta}(h(x_j+t)+1,v).
\end{multline}
We are computing this only for $u>h(\fx_i-\ft)$ and $v>h(\fx_j+\ft)$, so by the upper/lower triangularity of $e^{\pm2t\nabla-t\Delta}$ the second and third terms on the right hand side only survive for $u=h(\fx_i-\ft)+1$ and $v=h(\fx_j+\ft)-1$ respectively, and then since $e^{\pm2t\nabla-t\Delta}(\eta,\eta)=e^{2t}$ by \eqref{eq:Scontour}, the whole expression equals the right hand side of \eqref{eq:ptKij}, as desired.

\vskip3pt\noindent\emph{Case 2}: Suppose now that $r_k<r_0(t,x_k)$ for some $k$ and that we are not at a shock in the $x_k$ variable (i.e., that $h$ does not have a jump at $x_k+t$ or $x_k-t$).
Then $r_k<r_0(s,x_k)$ for all $s$ in a neighborhood of $t$, and for all such $s$ we then have $\FFF(s,h)=0$, so $\p_t\FFF(t,h)=0$.
Now consider ${\SL}\FFF$.
The deterministic part of the generator only moves the positions of the jumps of $h$, which lie away from $x_k\pm t$, while the creation part only moves $h$ up, which means that the inequality $r_k<r_0(t,x_k)$ (and hence the fact that $\FFF=0$) is preserved through the action of the generator.
Hence ${\SL}\FFF=0$, which gives us what we want.

\vskip3pt\noindent\emph{Case 3}: Next suppose that $r_k<r_0(t,x_k)$ for some $k$ and we are at a shock in the $x_k$ variable.
We split in three subcases.

\vskip3pt\noindent\emph{Case 3a}: Suppose $r_j<r_0(t,x_j)$ for some $j\neq k$ for which we are not at a shock in the $x_j$ variable.
Then the argument of case 2 applies, and we get again $(\p_t-{\SL})\FFF=0$.

\vskip3pt\noindent\emph{Case 3b}: Suppose now that  $r_j\geq r_0(t,x_j)$ for all $j\neq k$.
We will only consider the case when the $x_k$ shock comes from a jump of $h$ at $x_k+t$; a jump at the other edge, or at both edges, are handled similarly.
Additionally, we assume first for simplicity that there is only one jump at $x_k+t$, and explain at the end the extension to the case when there is a $(+1,-1)$ pair at that location.
Let $(\vec y,\vec\sigma)$ be the configuration of jump locations and signs associated to $h$.
By relabeling, we may assume that the (unique) jump at $x_k+t$ has label $0$, i.e. $\y_0=x_k+t$.
If $\sigma_0=-1$, i.e. there is a down jump at $x_k+t$, then by upper semi-continuity of $h$ we have that $r_0(s,z)$ is constant for $s\in[t-\ep,t+\ep]$, $z\in[x_k-\ep,x_k+\ep]$ and small $\ep$, so $\p_tF=0$ as above, and for the same reason ${\SL}\FFF=0$ (note here we need to consider the movement of the $\y_0$ jump, but it also does not affect $r_0(t,x_k)$).

We turn now to the case $\sigma_0=1$.
In particular, in this case we have $r_k<r_0(t+\ep,x_k)$, and hence $\FFF(t+\ep,h)=0$, for small $\ep>0$.
Note that, since we are assuming that the jump of $h$ at $x_k+t$ is up by one, $\lim_{s\nearrow t}r_0(s,x_k)=r_0(t,x_k)-1$.
If $r_k<r_0(t,x_k)-1$ then we also have $r_k<r_0(t-\ep,x_k)$ for small $\ep>0$, and as in the previous cases we have $\p_t\FFF={\SL}\FFF=0$.

The interesting situation is thus when $\sigma_0=1$ and $r_k=r_0(t,x_k)-1$.
This is the case for which $(\p_t-{\SL})\FFF$ is not $0$; we will show that it is bounded as needed.
In this case we have as above $\FFF(t+\ep,h)=0$, so
\[\FFF(t+\ep,h)-\FFF(t-\ep,h)=-\FFF(t-\ep,h).\]
Now we study ${\SL}\FFF$.
Once again, since the creation part of the generator only moves $h$ up (and we are assuming $r_k<r_0(t,x_k)$), we have $\SL_{\gcr}\FFF=0$.
On the other hand, it is clear that the deterministic part of the generator restricted to the jumps of $h$ other than the one at $\y_0=x_k+t$ cannot change the inequality $r_k<r_0(t,x)$, so that part of the action of the generator on $F$ vanishes.
This tells us that the only non-zero contribution to $\hat{\SL}F$ comes from the limit as $\ep\to0$ of $(2\ep)^{-1}$ times
\[\FFF(t,h)|_{\y_0=x_k+t-\ep}-\FFF(t,h)|_{\y_0=x_k+t+\ep},\]
i.e. the difference between computing $\FFF$ with the $\y_0$ jump moved to $x_k+t-\ep$ and to $x_k+t+\ep$ (recall that up steps move to the left).
We may assume that $\ep$ is small enough so that the only jump of $h$ in $[x_k+t-\ep,x_k+t+\ep]$ is the one at $x_k+t$.
When the $\y_0$ jump is moved to $x_k+t-\ep$ the inequality $r_k<r_0(t,x_k)$ remains (since $r_0(t,x_k)$ does not change) but, under our assumptions, when it is moved to $x_k+t+\ep$, $r_0(t,x_k)$ changes to $r_0(t,x_k)-1=r_k$, which means that the above difference equals $-\FFF(t,h)|_{\y_0=x_k+t+\ep}$.
The conclusion from all this is that, in the current situation,
\[(\p_t-{\SL})\FFF=\lim_{\ep\to0}(2\ep)^{-1}\big(-\FFF(t-\ep,h)+\FFF(t,h)|_{\y_0=x_k+t+\ep}\big).\]
Under our assumptions, when evaluating $\FFF$ at $t-\ep$ with $\y_0=x_k+t$ or at $t$ with $\y_0=x_k+t+\ep$, we are in the regime where $r_i\geq r_0(t,x_i)$ for each $i$, so that the $\FFF$'s on the right hand side are given by the corresponding Fredholm determinants, i.e.
\begin{equation}
(\p_t-{\SL})\FFF=\lim_{\ep\to0}(2\ep)^{-1}\big(\tsm\det(I-K(t,h)|_{\y_0=x_k+t+\ep})-\det(I-K(t-\ep,h))\big),\label{eq:cu}
\end{equation}
where we are including the dependence of $K$ on $t$, $h$ and $\y_0$ in the same way as for $F$.
Now we claim that there is a constant $C<\infty$ which is independent of $\ep$ and $h$ in the support of $\nu_\rho$ and bounded uniformly for $t$ in a compact interval such that
\begin{equation}\label{eq:detKth-bd}
\big|\!\det(I-K(t,h)|_{\y_0=x_k+t+\ep})-\det(I-K(t-\ep,h))\big|\leq C\tts\ep.
\end{equation}
We prove this in Appdx. \ref{sec:pf-detKth-bd}.
It shows that \noeqref{eq:detKth-bd}
$
    |(\p_t-{\SL})\FFF|\leq C
$
uniformly in $t$ in a finite set and $h\in\UC$, which is what we wanted.

The only situation left to consider in this case is when there is a $(+1,-1)$ pair at $x_k+t$, but it is easy to see that the above argument still works in this case: when the generator is applied to $\FFF$, the up jump moves left while the down jump moves right, but all the necessary inequalities as well as the estimates in Appdx. \ref{sec:pf-detKth-bd} remain valid.

\vskip3pt\noindent\emph{Case 3c}: The last possibility is that there are several indices $j_1,\dotsc,j_\ell$ (including $k$) for which $r_{j_i}<r_0(t,x_{j_i})$ and we are at a shock in each of those variables.
The proof in this case is basically the same.
When computing the limit defining $(\p_t-{\SL})\FFF$ there will be terms as above coming from the $t$ derivative for each of the $j_i$'s, and matching terms coming from ${\SL}\FFF$ coming from the deterministic movement of the jumps of $h$ at the corresponding points.
Most of these terms will vanish as above, and those which do not will lead a sum of at most $\ell$ differences as in \eqref{eq:cu}, which can be estimated separately.
\end{proof}

\appendix

\settocdepth{section}

\section{Initial data for non-Abelian Toda}\label{sec:multipt-ini}

In this appendix we compute the initial condition for the non-Abelian 2D Toda equation for PNG.
We focus on the equation \eqref{2a'} in terms of $U_r$ and $V_r$, and give expressions for these two matrices at $t=0$.
We stress that this is far from providing the whole boundary data for the equation, which we do not attempt here.
In everything that follows we work under the assumption $r_i\geq r_0(t,x_i)$, which for $t=0$ means we take, for each $i$,
\begin{equation}
r_i\geq h_0(x_i).\label{eq:rassum2}
\end{equation}

The formula \eqref{eq:KPNG0} for $K^{\uptext{ext}}\big|_{t=0}$ has the same structure as the (slightly formal, in that case) formula derived in \cite{kp}, in the setting of the KPZ fixed point, for the Brownian scattering operator at $t=0$.
Thus the same basic argument applies in our case, and it will lead to a formula for $Q_r\big|_{t=0}$ in \eqref{eq:Qrt0} below, providing a discrete analog of the formula derived in \cite{kp}.
However, in our setting we need to calculate $Q^{-1}_r\big|_{t=0}$ and $\p_\eta Q_r\big|_{t=0}$ as well.
This will require an additional computation, leading to more complicated formulas.
To express them we will need some:

\begin{notation}
For $i<j$ and $i\leq k\leq m\leq j$, let
\begin{multline}
P_{i,j}^{\geq h_0,\leq\check{\mathfrak{d}}_{k,m}}(u,v)
=\pp_{\fN(x_i)=u}\big(\fN(y)\geq h_0(y)~\forall\,y\in[x_i,x_j],\\
\fN(x_\ell)\leq r_\ell\uptext{ for each }x_\ell\in[x_k,x_m]\cap(x_i,x_j),\,\fN(x_j)=v\big).
\end{multline}
Note that in the probability we ask $\fN$ to stay above $h_0$ inside the whole interval $[x_i,x_j]$, and to stay below the heights $r_\ell$ only at those points $x_\ell$ among the $x_1<\dotsm<x_n$ with indices between $k$ and $m$ but excluding, if necessary, the two endpoints $x_i$ and $x_j$.
The notation $\check{\mathfrak{d}}_{k,m}$ represents an upside-down, multiple narrow wedge, defined as $\check{\mathfrak{d}}_{k,m}(x)=r_\ell$ if $x=x_\ell$ for some $x_\ell\in\{x_k,\dotsc,x_m\}$ and $\infty$ otherwise (so the event with upper bound inside the probability is the same as asking for $\fN(y)\leq\check{\mathfrak{d}}_{k,m}(y)$ for $y\in[x_i,x_j]$).
We extend this notation in several ways:

\noindent
 (i) If $k=m$ then we write $\check{\mathfrak{d}}_{k}$ instead of $\check{\mathfrak{d}}_{k,k}$.
 (ii) If we change, remove or add inequalities in the superscript on the left hand side, then all we do is change, remove or add the corresponding inequalities in the probability on the right hand side.
 (iii) When $u=r_i$ and $v=r_j$, we simply write
  \begin{equation}
  \bar P_{i,j}^{\geq h_0,\leq\check{\mathfrak{d}}_{i,j}}=P_{i,j}^{\geq h_0,\leq\check{\mathfrak{d}}_{i,j}}(r_i,r_j),
  \end{equation}
 and similarly if any of the inequalities are modified.
 (iv) We turn the last definition into an $n\times n$ matrix $\bar P^{\geq h_0,<\check{\mathfrak{d}}}_{\vec r}$ defined through (here for convenience we make the dependence on $\vec r=(r_1,\dots,r_n)$ explicit)
  \begin{equation}\label{eq:lastbPnotation}
  (\bar P_{\vec r}^{\geq h_0,<\check{\mathfrak{d}}})_{ij}=\uno{i<j}\bar P^{\geq h_0,<\check{\mathfrak{d}}_{i,j}}_{i,j}
  \end{equation}
 for each $i,j$ (note that this is using implicitly the convention that each $\bar P_{ij}=0$ for $i\geq j$). 
 (v) We let $\geqdot h_0$ denote the event that $\fN$ stays $\geq h_0$ but hits exactly $h_0$ along the way, i.e.
  \begin{equation}\label{eq:geqdot}
  P^{\geqdot h_0}_{i,j}=P^{\geq h_0}_{i,j}-P^{>h_0}_{i,j}.
  \end{equation}
 We extend this notation as before if we need to add further restrictions in the paths of $\fN$, and also to the case $\leqdot\check{\mathfrak{d}}$.
 (vi) Finally, we use the notation $\geqdot h_0\tts\scalebox{0.65}{$\curvearrowright$}\tts\leqdot\check{\mathfrak{d}}_{i,j}$ if we require the $\check{\mathfrak{d}}_{i,j}$ touching to occur after the first $h_0$ touching, and we also write
  \begin{equation}\label{eq:geqdotplusr}
  \big(\bar P^{\geqdot h_0,\leq\check{\mathfrak{d}}}_{+,\vec r}\big)_{ij}=P^{\geqdot h_0,\leq\check{\mathfrak{d}}_{i,j}}_{i,j}(r_i+1,r_j)
  \qqand\big(\bar P^{\geqdot h_0\tts\scalebox{0.65}{$\curvearrowright$}\tts\leqdot\check{\mathfrak{d}}}_{+,\vec r}\big)_{ij}=P^{\geqdot h_0\tts\scalebox{0.65}{$\curvearrowright$}\tts\leqdot\check{\mathfrak{d}}_{i,j}}_{i,j}(r_i+1,r_j).
  \end{equation}
\end{notation}

Additionally, we introduce the shift operators
\[\theta_rf(u)=f(u+r),\qquad\theta^*_rf(u)=f(u-r).\]
The action of $\theta_r$ is meant as taking place in all of $\ell^2(\zz)$ (which is why we are distinguishing $\theta_1$ and $\sigma$).
In terms of our formulas, where all (scalar) kernel products take place in $\ell^2(\zz_{>0})$, it is better to think of $\theta_r$/$\theta^*_r$ simply as a convenient notation, which modifies the kernel to its right/left, rather than as the action of an operator.
The point is that for a given kernel $A$ we have $\theta_r A\theta^*_r(u,v)=A(u+r,v+r)$ so we can, for instance, rewrite \eqref{eq:KPNG0} (using also \eqref{eq:rassum2}) as
\begin{equation}\label{eq:Kij3}
(K_r)_{ij}=-\uno{i<j}\theta_{r_i}\P_{r_i}P^{>h_0}_{i,j}\P_{r_j}\theta^*_{r_j}=-\uno{i<j}\theta_{r_i}P^{>h_0}_{i,j}\theta^*_{r_j},
\end{equation}
since $\theta_r\P_r=\P_0\theta_r$, $\P_r\theta_r^*=\theta_r^*\P_0$ and our kernels act on $\ell^2(\zz_{>0})$ so the $\P_0$ factors may be omitted.

We turn now to the computation.
We begin with a  formula for $Q_r\big|_{t=0}$:

\begin{prop}\label{prop:Qrt0}
Under the condition \eqref{eq:rassum2} we have
\begin{equation}\label{eq:Qrt0}
Q_r\big|_{t=0}=I-\bar P_{\vec r}^{\geq h_0,<\check{\mathfrak{d}}}\qqand
Q^{-1}_r\big|_{t=0}=I+\bar P_{\vec r}^{\geq h_0,\leq\check{\mathfrak{d}}}.
\end{equation}
\end{prop}

\begin{proof}
In view of \eqref{eq:Kij3}, $K_r$ is strictly upper triangular, so $(R_r)_{ij}=\sum_{m=0}^{j-i}(K^m)_{ij}$ and we have $(R_r)_{ii}=I$ and $(R_r)_{ij}=0$ for $i>j$.
For $i<j$, on the other hand, we decompose $(K^m_r)_{ij}$ as $\sum_{\pi}\prod_{\ell=1}^{m}(K_r)_{\pi(\ell)\pi(\ell+1)}$, where the sum is over strictly increasing paths $\pi=(\pi(1),\dotsc,\pi(m+1))$ going from $i$ to $j$ along indices $\{1,\dotsc,n\}$.
Applying the inclusion-exclusion principle and using the fact that we are acting on $\ell^2(\zz_{>0})$ and $\theta^*_r\P_0\theta_r=\P_r$ in $\ell^2(\zz)$, we deduce that
\begin{equation}\label{eq:Rij3}
(R_{r})_{ij}=\uno{i=j}-\uno{i<j}\theta_{r_i}P^{>h_0,\leq\check{\mathfrak{d}}_{i,j}}_{i,j}\theta^*_{r_j}
\end{equation}
(see \cite[Eqn. (A.6)]{kp} for a similar argument), and then evaluating at $(1,1)$ gives
\[Q_{ij}=\uno{i=j}I-\uno{i<j}P^{>h_0,\leq\check{\mathfrak{d}}_{i,j}}(r_i+1,r_j+1)=\uno{i=j}I-\uno{i<j}P^{\geq h_0,<\check{\mathfrak{d}}_{i,j}}(r_i,r_j),\]
where in the second equality we used the translation invariance of $\fN$.
This is the first identity in \eqref{eq:Qrt0}.

For $Q^{-1}_r\big|_{t=0}$ it is now enough to check that $(I-\bar P_{\vec r}^{\geq h_0,<\check{\mathfrak{d}}})(I+\bar P_{\vec r}^{\geq h_0,\leq\check{\mathfrak{d}}})=I$, which is equivalent to $\bar P_{\vec r}^{\geq h_0,\leq\check{\mathfrak{d}}}-\bar P_{\vec r}^{\geq h_0,<\check{\mathfrak{d}}}=\bar P_{\vec r}^{\geq h_0,\leq\check{\mathfrak{d}}}\bar P_{\vec r}^{\geq h_0,<\check{\mathfrak{d}}}$.
The left hand side equals $\bar P_{\vec r}^{\geq h_0,\leqdot\check{\mathfrak{d}}}$, using the notation (v) introduced above.
The right hand side is trivially equal to this at or below the diagonal, while a simple computation using the upper triangularity of each factor shows that, above the diagonal, the product provides a decomposition of $\bar P_{\vec r}^{\geq h_0,\leqdot\check{\mathfrak{d}}}$ according to the last time the walk hits $\check{\mathfrak{d}}$.
\end{proof}

Next we need to compute $\p_\eta Q_r|_{t=0}$.

\begin{prop}
With the notation introduced above, and under condition \eqref{eq:rassum2}, we have
\begin{equation}
\p_\eta Q_{r}\big|_{t=0}=\bar P^{\geqdot h_0,\leq\check{\mathfrak{d}}}_{+,\vec r}-\bar P^{\geqdot h_0\tts\scalebox{0.65}{$\curvearrowright$}\tts\leqdot\check{\mathfrak{d}}}_{+,\vec r}.\label{eq:deltaQrrr}
\end{equation}
\end{prop}

\begin{proof}
In all that follows we set $t=0$.
We have
\begin{equation}
\p_\eta Q_r=\p_\eta R_r(1,1)=R_r\p_\eta K_rR_r(1,1).\label{eq:detaQr}
\end{equation}
$R_r$ is upper triangular while $\p_\eta K_r$ is strictly upper triangular, so the same holds for $\p_\eta K_rR_r$ and $R_r\p_\eta K_rR_r$.
In view of this, in the computation that follows it is enough to consider $i<j$.

Recall $\p_\eta K_r(u,v)=-K_r(u+1,v)+K_r(u,v-1)$, so using \eqref{eq:Kij3} we have
\[(\p_\eta K_r)_{ij}=\theta_{r_i+1}P^{>h_0}_{i,j}\theta^*_{r_j}-\theta_{r_i}P^{>h_0}_{i,j}\theta^*_{r_j-1}=\theta_{r_i}P^{\geq h_0}_{i,j}\theta^*_{r_j-1}-\theta_{r_i}P^{>h_0}_{i,j}\theta^*_{r_j-1}=\theta_{r_i}P^{\geqdot h_0}_{i,j}\theta^*_{r_j-1},\]
where we used translation invariance and the notation introduced in \eqref{eq:geqdot}.
Using now \eqref{eq:Rij3} we get that 
\begin{equation}
\textstyle(R_r\p_\eta K_r)_{ij}=\theta_{r_i}P^{\geqdot h_0}_{i,j}\theta^*_{r_j-1}
-\sum_{i<k<j}\theta_{r_i}P^{>h_0,\leq\check{\mathfrak{d}}_{i+1,k-1},>\check{\mathfrak{d}}_k,\geqdot h_0}_{i,j}\theta^*_{r_j-1},
\end{equation}
where in the second $P_{i,j}$ operator the sequence of inequalities means in particular that $\fN>h_0$ in $[x_i,x_k]$ and $\fN\;\geqdot\; h_0$ in $[x_k+1,x_j]$ (there is no ambiguity at $x_k$ since, under \eqref{eq:rassum2}, the condition $>\check{\mathfrak{d}}_k$ implies $\fN(x_k)>h_0(x_k)$ anyway), and we used again that $\theta^*_r\P_0\theta_r=\P_r$ in $\ell^2(\zz)$.

Next we compute $(R_r\p_\eta K_rR_r)_{ij}$.
Using the last formula and \eqref{eq:Rij3}, it is given by
\begin{multline}
\textstyle\theta_{r_i}P^{\geqdot h_0}_{i,j}\theta^*_{r_j-1}
-\sum_{i<k<j}\theta_{r_i}P^{>h_0,\leq\check{\mathfrak{d}}_{i+1,k-1},>\check{\mathfrak{d}}_k,\geqdot h_0}_{i,j}\theta^*_{r_j-1}
-\sum_{i<k<j}\theta_{r_i}P^{\geqdot h_0}_{i,k}\theta^*_{r_k-1}\theta_{r_k}P^{>h_0,\leq\check{\mathfrak{d}}_{k,j}}_{k,j}\theta^*_{r_j}\\
\textstyle+\sum_{i<k<\ell<j}\theta_{r_i}P^{>h_0,\leq\check{\mathfrak{d}}_{i+1,k-1},>\check{\mathfrak{d}}_k,\geqdot h_0}_{i,\ell}\theta^*_{r_\ell-1}\theta_{r_\ell}P^{>h_0,\leq\check{\mathfrak{d}}_{\ell,j}}_{\ell,j}\theta^*_{r_j}.\label{eq:crazy5}
\end{multline}
Consider the factor $\theta^*_{r_k-1}\theta_{r_k}P^{>h_0,\leq\check{\mathfrak{d}}_{k,j}}_{k,j}\theta^*_{r_j}$.
If we make explicit the fact that everything acts on $\ell^2(\zz_{>0})$ except for the $\theta_{r}$'s to their right and the $\theta^*_r$'s to their left, we can write this as
\begin{equation}
\theta^*_{r_k-1}\P_0\theta_{r_k}P^{>h_0,\leq\check{\mathfrak{d}}_{k,j}}_{k,j}\theta^*_{r_j}
=\P_{r_k-1}\theta_{1}P^{>h_0,\leq\check{\mathfrak{d}}_{k,j}}_{k,j}\theta^*_{r_j}
=\P_{r_k-1}P^{\geq h_0,<\check{\mathfrak{d}}_{k,j}}_{k,j}\theta^*_{r_j-1},\label{eq:crazy6}
\end{equation}
where once again we used translation invariance.
Let us for a moment replace the projection $\P_{r_k-1}$ on the left by $\P_{r_k}$, and use this in the last two terms of the above formula.
We get
\begin{multline}
\textstyle\theta_{r_i}P^{\geqdot h_0}_{i,j}\theta^*_{r_j-1}
-\sum_{i<k<j}\theta_{r_i}P^{>h_0,\leq\check{\mathfrak{d}}_{i+1,k-1},>\check{\mathfrak{d}}_k,\geqdot h_0}_{i,j}\theta^*_{r_j-1}
-\sum_{i<k<j}\theta_{r_i}P^{\geqdot h_0,>\check{\mathfrak{d}}_k,<\check{\mathfrak{d}}_{k
+1,j-1},\geq h_0}_{i,j}\theta^*_{r_j-1}\\
\textstyle+\sum_{i<k<\ell<j}\theta_{r_i}P^{>h_0,\leq\check{\mathfrak{d}}_{i+1,k-1},>\check{\mathfrak{d}}_{k},\geqdot h_0,>\check{\mathfrak{d}}_\ell,<\check{\mathfrak{d}}_{\ell-1,j},\geq h_0}_{i,j}\theta^*_{r_j-1}.
\end{multline}
The second and third terms correspond to $\fN$ paths which stay $\geq h_0$ and which at some point touch $h_0$ and at some other point go strictly above one of the $r_k$'s; the second sum corresponds to a decomposition of the event that the path goes above some $r_k$ before it first touches $h_0$, while the third sum corresponds  to the event that the path goes above some $r_k$ after it last touches $h_0$.
Now we note that the sum in the last term is precisely a decomposition of the probability of the intersection of these two events.
Thus the last three terms above correspond to minus the probability that $\fN\;\geqdot\;h_0$ and it goes strictly above one of the $r_k$'s.
Using this in the last expression and evaluating at $(1,1)$ yields $P^{\geqdot h_0,\leq\check{\mathfrak{d}}_{i,j}}_{i,j}(r_i+1,r_j)$.

What remains is to account for the difference between $\P_{r_k-1}$ and $\P_{r_k}$ in \eqref{eq:crazy6}.
This difference is just $\uno{\{r_k\}}$, which corresponds to evaluation at $r_k$.
So, upon evaluating at $(1,1)$, the remaining terms we get from \eqref{eq:crazy5}/\eqref{eq:crazy6} amount to
\begin{multline}
\textstyle-\sum_{i<k<j}P^{\geqdot h_0}_{i,k}(r_i+1,r_k)P^{\geq h_0,<\check{\mathfrak{d}}_{k,j}}_{k,j}(r_k,r_j)\\
\textstyle+\sum_{i<k<\ell<j}P_{i,\ell}^{>h_0,\leq\check{\mathfrak{d}}_{i+1,k-1},>\check{\mathfrak{d}}_k,\geqdot h_0}(r_i+1,r_\ell)P_{\ell,j}^{\geq h_0,<\check{\mathfrak{d}}_{\ell,j}}(r_\ell,r_j).
\end{multline}
The first sum is a decomposition of the probability that $\fN\;\geqdot\; h_0$ and it touches some of the $r_k$'s at some point after it first touches $h_0$, remaining below $\check{\mathfrak{d}}$ after that.
The second sum corresponds to something similar, except that we also ask for $\fN$ to go strictly above some of the $r_k$'s, excluding $k=i$, before touching $h_0$.
The difference of the two sums then corresponds to the event that $\fN\;\geqdot\;h_0$, $\fN\;\leqdot\;\check{\mathfrak{d}}$ (with the obvious meaning), and that there is a $\check{\mathfrak{d}}$ touching after the first $h_0$ touching.
In the notation introduced above, we have proved that
$(\p_\eta Q_{r}|_{t=0})_{ij}=P^{\geqdot h_0,\leq\check{\mathfrak{d}}_{i,j}}_{i,j}(r_i+1,r_j)-P^{\geqdot h_0\tts\scalebox{0.65}{$\curvearrowright$}\tts\leqdot\check{\mathfrak{d}}_{i,j}}_{i,j}(r_i+1,r_j)$.
This yields \eqref{eq:deltaQrrr}.
\end{proof}

Finally we can write the initial data for the non-Abelian Toda equations \eqref{2a'}.  Like the Rosetta stone, all  notations can be deciphered from the preceding explanations.

\begin{prop}
Assuming $r_i\geq h_0(x_i)$ for each $i$,
\begin{equation}
U_r\big|_{t=0}=\big(I-\bar P_{\vec r}^{\geq h_0,<\check{\mathfrak{d}}}\big)\!\big(I+\bar P_{\vec r}^{>h_0,\leq\check{\mathfrak{d}}}\big),\qquad
V_r\big|_{t=0}=\big(\bar P^{\geqdot h_0\tts\scalebox{0.65}{$\curvearrowright$}\tts\leqdot\check{\mathfrak{d}}}_{+,\vec r}-\bar P^{\geqdot h_0,\leq\check{\mathfrak{d}}}_{+,\vec r}\big)\!\big(I+\bar P_{\vec r}^{\geq h_0,<\check{\mathfrak{d}}}\big).
\end{equation}
\end{prop}

\begin{proof}
This follows directly from \eqref{eq:Qrt0} and \eqref{eq:deltaQrrr}, together with the identity $Q^{-1}_{r-1}\big|_{t=0}=I+\bar P_{\vec r}^{>h_0,\leq\check{\mathfrak{d}}}$, which in turn follows from \eqref{eq:Qrt0} and translation invariance.
\end{proof}

\section{Trace class estimates}\label{sec:tr-cl-estimate}

\subsection{Continuity}

The PNG kernel and transition probabilities are not continuous in the initial data  $h_0$ in $\UC$.
Consider, for simplicity, the one-point case.
The issue arises at the shock set $\mathcal{S}(x)$.
As an example, take $h_\ep(z)= \log(\uno{z\geq t+\ep})$, converging in $\UC$ to $h_0(z)= \log(\uno{z\geq t})$.
But $P^{\hit(h_0)}_{-t,t}=e^{2t\Delta}\bP_0$ while $P^{\hit(h_0^\ep)}_{-t,t}=0$. Writing $K_{h}$ for $K^\uptext{ext}_{h}(0,\cdot;0,\cdot)$ we get 
$K_{h_0}=e^{-2t\nabla+t\Delta}\bP_0e^{2t\nabla-t\Delta}$ and $K_{h_\ep}=0$.
Moreover $\pp_{h_0}(h(t,0)\leq-1)=0$ while $\pp_{h_\ep}(h(t,0)\leq-1)=1$. 

In the above example, $h_\ep\leq h_0$.
We will see that if we restrict instead to approximations of $h_0$ from above, continuity does hold.
To express this more precisely, we say that \emph{$h_n\longrightarrow h$ from above} in $\UC$ if $h_n\longrightarrow h$ in $\UC$ and $h_n(x)\geq h(x)$ for all $x\in\rr$ and all $n\in\nn$.  A function $F\!:\UC\longrightarrow\mathbb{R}$, is \emph{continuous from above} if $F(h_n)\longrightarrow F(h)$ whenever $h_n\longrightarrow h$ from above.

\begin{prop}\label{prop:cont1}
\leavevmode
\mbox{}\vskip-18pt\mbox{}
\begin{enumerate}[label=\uptext{(\roman*)}]
    \item For each $a<b$ in $\rr$ and $u,v$ in $\zz$, $P_{a,b}^{\hit(h)}(u,v)$ is continuous from above as a function of $h\in \UC$.
    \item For each $h\in \UC$, $x\in \rr$, $u,v\in \zz$, $P^{\hit(h)}_{x-t,x+t}(u,v)$ is right continuous in $t\ge 0$.
\end{enumerate}
\end{prop}

\begin{proof}  (i)  It is not hard to see that if $h_n$ are converging to $h$ from above in $\UC$, then the Lebesgue measure of the subset $A_n$ of $[a,b]$ on which $h$ is strictly less than $h_n$ is going to zero.
Since $P_{a,b}^{\hit(h_n)}(u,v)-P_{a,b}^{\hit(h)}(u,v)\geq0$ and it is bounded above by the probability that the walk $\fN$ has a downward jump in $A_n$, which goes to $0$, we conclude that the desired property holds.  (ii)  Let $t_\ep\searrow t$.  Since $h\in \UC$, for $\ep$ sufficiently small $h(x+t_\ep)\le h(x+t)$ and $h(x-t_\ep)\le h(x-t)$.  So it is not hard to see that the difference between $P^{\hit(h)}_{x-t_\ep,x+t_\ep}(u,v)$ and $P^{\hit(h)}_{x-t,x+t}(u,v)$ is controlled by the probability of a jump in $[x-t_\ep,x-t]\cup[x+t,x+t_\ep]$, which is of order $t_\ep-t$.
\end{proof}

\subsection{The kernel}\label{sec:trcl-ker}

For fixed $x_1<\dotsc<x_n$ we introduce the multiplication operators
\begin{equation}\label{eq:theta-conj}
\vartheta f(x_i,u)=\vartheta_if(x_i,\cdot)(u)\coloneqq\vartheta_i(u)f(x_i,u)\qquad\uptext{with}\qquad \vartheta_i(u)=(1+u^2)^{i}.
\end{equation}

\begin{prop}\label{prop:trcl}
For fixed $t>0$, $x_1<\dotsc<x_n$, $r_1,\dotsc,r_n\in\zz$ and $h_0\in\UC$, the conjugated kernel $\vartheta\P_rK^{\uptext{ext}}_{h_0}\P_r\vartheta^{-1}$ on $\ell^2(\{x_1,\dotsc,x_n\}\times\zz)$ is:
\begin{enumerate}[label=\uptext{(\roman*)}]
\item Trace class, with a bound on the trace norm of the form $C(1+e^{c\max_i(r_0(t,x_i)-r_i)})$, with $r_0$ as in \eqref{eq:r0}, for some $c,C<\infty$ which can be chosen uniformly in bounded sets of $t$.
\item Continuous from above on $\UC$ in trace norm.
\item Right continuous in $t$ in trace norm.
\end{enumerate}
\end{prop}

It is important in (i) to get a bound on the trace norm depending on $h_0$ only through the $r_0(t,x_i)$'s.
The reason is that in our proof of the backward equation in Sec. \ref{sec:backw} we need a bound on $\FFF(t,h_0)=\det(I-K_r)_{\oplus_n\ell^2(\zz_{>0})}\uno{r_i\geq r_0(t,x_i),\ts i=1,\dotsc,n}$ in terms of the supremum of $h_0$ on the region of dependence. The  abstract bound available for the absolute value of the Fredholm determinant is  $e^{\|K_r\|_1}$.
So thanks to the truncation coming from the indicator function in the definition of $\FFF$, the bound on the trace norm provided by the proposition in fact yields a \emph{uniform} bound on $|\FFF|$.

\begin{proof}
(i) It is enough to show that each entry $\P_{r_i}\vartheta_iK^{\uptext{ext}}(x_i,\cdot;x_j,\cdot)\vartheta^{-1}_j\P_{r_j}$ of the kernel is trace class on $\ell^2(\zz)$.  The main estimate we will use is the following very simple bound, valid for any $t,x\in\rr$:
\begin{equation}\label{eq:Spng-est}
\big|e^{2t\nabla+x\Delta}(u,v)\big|\leq C\tts e^{c(|t|+|x|)}\lambda^{|u-v|}
\end{equation}
for any fixed $\lambda\in(0,1)$ and constants $c,C<\infty$ which depend only on  $\lambda$.
This estimate follows directly from \eqref{eq:Scontour} after choosing $\gamma_0$ as a circle of radius $\lambda$ when $u>v$ and $1/\lambda$ otherwise.

Consider the first term on the right hand side of \eqref{eq:PNG-kernel-two-sided-hit}, which (after conjugating and projecting) equals
\begin{equation}
\P_{r_i}\vartheta_i e^{(x_j-x_i)\Delta}\vartheta^{-1}_j\P_{r_j}=\big(\P_{r_i}\vartheta_i e^{\frac12(x_j-x_i)\Delta}\vartheta^{-1/2}_{i+j}\big)\big(\vartheta_{i+j}^{1/2}e^{\frac12(x_j-x_i)\Delta}\vartheta^{-1}_j\P_{r_j}\big)\label{eq:fQdec}
\end{equation}
for fixed $i<j$, where the operators $\vartheta^{\pm1/2}_k$ correspond to multiplication by $\vartheta_k(u)^{\pm1/2}$.
Using the bound $\|AB\|_1\le \|A\|_2\|B\|_2$ is enough to show that each of the two factors is Hilbert-Schmidt.
By \eqref{eq:Spng-est} we have  
\[\textstyle\|\P_{r_i}\vartheta_i e^{\frac12(x_j-x_i)\Delta}\vartheta^{-1/2}_{i+j}\|_2^2\leq C^2\tts e^{c(x_j-x_i)}\sum_{u>r_i}(1+u^2)^{2i}\sum_{v\in\zz}\lambda^{2|u-v|}(1+v^2)^{-i-j}.\]
The sum in $v$ restricted to $|v|\leq|u|/2$ is bounded by $2\kappa|u|\lambda^{|u|}$ while the sum over $|v|>|u|/2$ is bounded by $C\tts 2^{2(i+j)}(1+u^2)^{-i-j}$ for some $C>0$ (depending on $\lambda)$, so
\[\textstyle\|\P_{r_i}\vartheta_i e^{\frac12(x_j-x_i)\Delta}\vartheta^{-1/2}_{i+j}\|_2^2\leq C\tts e^{c(x_j-x_i)}\sum_{u>r_i}((1+u^2)^{2i}|u|\lambda^{2(1-\kappa)|u|}+(1+u^2)^{-(j-i)})<\infty,\]
since $j-i\geq1$.
The second factor on the right hand side of \eqref{eq:fQdec} can be estimated in the same way.  

Now we turn to the second term in \eqref{eq:PNG-kernel-two-sided-hit}.
Fix $i,j$ and write $\ux=x_i-t$, $\ox=x_j+t$.
The kernel we are interested in equals $\P_{r_i}\vartheta_i e^{-2t\nabla-t\Delta}P^{\hit(h_0)}_{\ux,\ox}e^{2t\nabla-t\Delta}\vartheta^{-1}_j\P_{r_j}$.
If $\ux>\ox$ then $P^{\hit(h_0)}_{\ux,\ox}$ vanishes and there is nothing to prove, so assume that $\ux\leq\ox$.
$h_0$ is upper semi-continuous so we may define
\[\textstyle m_0=\max_{z\in[\ux,\min\{x_i+t,\ox\}]}h_0(z),\qquad x_0=\argmax_{z\in[\ux,\min\{x_i+t,\ox\}]}h_0(z),\]
where in the argmax we take, say, the rightmost point in the interval where the maximum is attained.
Write $e^{-2t\nabla-t\Delta}P^{\hit(h_0)}_{\ux,\ox}e^{2t\nabla-t\Delta}$, using the inclusion-exclusion principle, as
\begin{equation}\label{eq:bmSMdecomp}
 e^{-2t\nabla-t\Delta}\big(P^{\hit(h_0)}_{\ux,x_0}e^{(\ox-x_0)\Delta}+e^{(x_0-\ux)\Delta}P^{\hit(h_0)}_{x_0,\ox}-P^{\hit(h_0)}_{\ux,x_0}P^{\hit(h_0)}_{x_0,\ox}\big) e^{2t\nabla-t\Delta}.
\end{equation}
The reason we choose this decomposition, splitting the hit probabilities at $x_0$, is that it will allow us to get a bound on the trace norm that only depends on $h_0$ through $r_0(t,x_i)$, its maximum on the interval $[x_i-t,x_i+t]$ (which is contained in $[\ux,x_0]$).
This dependence will come from estimates on the factors involving $P^{\hit(h_0)}_{\ux,x_0}$ in this decomposition.
As we will see, the factors $P^{\hit(h_0)}_{x_0,\ox}$ are a bit simpler to handle, basically because in the conjugation they end up next to decaying factor $\vartheta_j^{-1}$ (as opposed to the factor $\vartheta_i$ on the left, which is blowing up).
In fact, we will only estimate the trace norm of the third term coming from the above expression after conjugating and projecting; the bound for the other two
is slightly simpler, and can be adapted easily from the one we present here.

Note that if $m_0=-\infty$ (as it may be) then $P^{\hit(h_0)}_{\ux,x_0}=0$ and there is nothing to prove.
So we assume $m_0>-\infty$.
Our goal is to estimate the trace norm of $\P_{r_i}\vartheta_ie^{-2t\nabla-t\Delta}P^{\hit(h_0)}_{\ux,x_0}P^{\hit(h_0)}_{x_0,\ox}e^{2t\nabla-t\Delta}\vartheta_j^{-1}\P_{r_j}$.
We can write the kernel of this operator as an average of rank one kernels $A_{s_1}(u,\eta_1)B_{s_2}(\eta_2,v)$, 
\begin{equation}\label{eq:bmSMhitdecomp}
 \textstyle \int_{\ux}^{x_0}\int_{x_0}^{\ox}\sum_{\eta_1,\eta_2\in\zz} \alpha(\eta_1,\d s_1,\eta_2,\d s_2) A_{s_1}(u,\eta_1)B_{s_2}(\eta_2,v)
 \end{equation}
 where $A_{s}=\P_{r_i}\vartheta_i e^{-2t\nabla+(x_0-s-t)\Delta} $, $B_s= e^{2t\nabla+(\ox-s_2-t)\Delta}\vartheta_j^{-1}\P_{r_j}$ 
 and $\alpha(\eta_1,\d s_1,\eta_2,\d s_2)$ is given by $\sum_{\xi\in\zz}\hat p_{\ux,\xi}(\eta_1,\d s_1) p_{x_0,\xi}(\eta_2,\d s_2)$, with $p_{x,\xi}(\eta,\d s)=\pp_{\fN(x)=\xi}(\tau\in\d s,\,\fN(\tau)=\eta)$ and $\hat p_{x,\xi}(\eta,\d s)=\pp_{\fN(x)=\xi}(\hat\tau\in\d s,\,\fN(\hat\tau)=\eta)$. Here $\tau$ and $\hat\tau$ are the hitting times by $\fN$ (restricted to $[x_0,\ox]$ and to $[\ux,x_0]$) of $\hypo(h_0)$ and $\hypo(\hat h_0)$, with $\hat h_0$ the reversed barrier given by $h_0(s)=h_0(\ux+x_0-s)$; the first one is a straightforward decomposition using the hitting time $\tau$ while the second one is similar after using the reversibility of the walk $\fN$.
We note that in the case $x_0=\ox$, the above is interpreted as forcing $\tau=x_0$ and $\eta=y$ provided $y\leq h_0(x_0)$, and $0$ otherwise, and similarly in the case $x_0=\ux$.

The trace norm of the average can be bounded by the average of the trace norms of the rank one operators.
But the square of the trace norm of a rank one operator is just the product of $\ell^2$ norms.
Each $\ell^2$ norm can be estimated by \eqref{eq:Spng-est}. Thus the trace norm of $\big\|\P_{r_i}\vartheta_ie^{-2t\nabla-t\Delta}P^{\hit(h_0)}_{\ux,x_0}P^{\hit(h_0)}_{x_0,\ox} e^{2t\nabla-t\Delta}\vartheta_j^{-1}\P_{r_j}\big\|_1$ is bounded by a constant depending on $\lambda$, $t$, $\ux$ and $\ox$ times
 \begin{multline}
\textstyle\sum_{\eta_1,\eta_2,\xi}\pp_{\fN(\ux)=\xi}(\hat \tau\in[\ux,x_0],\fN(\hat\tau)=\eta_1)\pp_{\fN(x_0)=\xi}(\tau\in[x_0,\ox],\fN(\tau)=\eta_2)\label{eq:cbu78}\\
\times(\vartheta_i(\eta_1)^2+\vartheta_i(r_i)^2)^{1/2}\lambda^{(r_i-\eta_1)\vee0}\lambda^{(r_j-\eta_2)\vee0}.
\end{multline}
If $\xi\leq m_0$,
then $\hat\tau=\ux$ and $\fN(\tau)=\xi$, which forces $\eta_1=\xi$.
Thus the part of the sum over $\xi\leq m_0$ is bounded by
\begin{equation}
\textstyle\sum_{\xi\leq m_0,\eta_2\in\zz}\pp_{\fN(x_0)=\xi}(\tau\in[x_0,\ox],\fN(\tau)=\eta_2)(\vartheta_i(\xi)^2+\vartheta_i(r_i)^2)^{1/2}\lambda^{(r_i-\xi)\vee0}.
\end{equation}
We can sum over $\eta_2$ in the probabilities, then drop the resulting $\pp_{\fN(x_0)=\xi}(\tau\in[x_0,\ox])$.
The resulting expression is summable, with a result which can be bounded by $C(1+ e^{c(m_0-r_i)})$.
On the other hand, if  $\xi>m_0$, we use the fact that $\fN(\hat\tau)\leq m_0$ if $\hat\tau\in[\ux,x_0]$, so we may restrict the $\eta_1$ sum to $\eta_1\leq m_0$. So this part of the sum is less than or equal to
\begin{equation}
    \textstyle\sum_{\eta_1\leq m_0,\xi>m_0}\pp_{\fN(\ux)=\xi}(\hat \tau\in[\ux,x_0],\fN(\hat\tau)=\eta_1)(\vartheta_i(\eta_1)^2+\vartheta_i(r_i)^2)^{1/2}\lambda^{(r_i-\eta_1)\vee0}.
\end{equation}
Summing in $\eta_1$ gives an expression which is bounded by $C(1+ e^{c(m_0-r_i)})$ again, and it only remains to estimate 
$\sum_{\xi>m_0}\pp_{\fN(\ux)=\xi}(\hat \tau\in[\ux,x_0])$.
Now $\pp_{\fN(\ux)=\xi}(\tau\in[x_0,\ox])\leq\pp_{\fN(x_0)=\xi}\!\left(\inf_{y\in[\ux,x_0]}\fN(x)\leq m_0\right)$ since $m_0$ maximizes $\hat h_0$ inside $[\ux,x_0]$.
We can compute the latter by the reflection principle, and bound it by $C\tts\lambda^{\xi-m_0}$, and summing over $\xi>m_0$ yields a new constant.
The conclusion is that \eqref{eq:cbu78} is bounded by $C(1+ e^{c(m_0-r_i)})$ for some $c,C<\infty$, as desired.  This proves (i). 

To prove (ii) and (iii), first of all if $h_\ep\to h$ in $\UC$ then there is a global constant $\bar h<\infty$ such that $h_\ep(x)\le \bar h$ for $x$ in the finite region of support $[x_1-t,x_n+t]$ up to time $t$, for  small $\ep>0$. 
The kernel $\vartheta\P_rK^{\uptext{ext}}_{h_0}\P_r\vartheta^{-1}$ is a composition of kernels each of which is continuous from above in $\UC$ or right continuous in $t$, pointwise, by Prop. \ref{prop:cont1} in the case of the hit kernel, or from \eqref{eq:Scontour} in the case of the conjugating operators $e^{\pm 2t\nabla\pm x\Delta}$.
If we look at one of the differences for which we have to control the trace norm, say $\P_r\vartheta K^{\uptext{ext}}_{h_\ep}\P_r\vartheta^{-1}-\vartheta\P_rK^{\uptext{ext}}_{h_0}\vartheta^{-1}\P_r$ (the other one is similar), and consider its $(i,j)$ entry, we may factor it as $\P_{r_i}\vartheta_ie^{-2t\nabla-t\Delta}(P^{\hit(h_\ep)}_{x_i-t,x_j+t}-P^{\hit(h)}_{x_i-t,x_j+t})e^{2t\nabla-t\Delta}\vartheta^{-1}_j\P_{r_j}$
We follow the estimation procedure from the proof of (i), keeping the absolute value of this difference.  
Replacing the absolute value of the difference by the sum, and the $h_\ep$ by the constant function $\bar h$ (which can only increase the hit probability), we obtain a term that by the estimation procedure of (i) is bounded, and serves as a dominating function for the difference.
Hence we can use the dominated convergence theorem to conclude the trace norm convergence from the pointwise convergence of the individual kernels.
\end{proof}

\subsection{Differentiability}\label{sec:trcl-diff}

Now we turn to checking that the (conjugated and projected) PNG kernel satisfies (2) and (3) from Thm. \ref{thm:nonAbelianToda-general} with the derivatives holding in trace class.
The arguments are very similar to the ones in Sec. \ref{sec:trcl-ker}, and we only write them in the case of the $\eta$ derivative, the $\zeta$ derivative being completely analogous.
As above, it is enough to check this separately for each entry $K^{\ext}(x_i,\cdot;x_j,\cdot)$, after conjugating and adding the projections.
We fix $i,j$ and let $K_{\eta,\zeta}$ denote this entry, where we have made the dependence on $\eta$ and $\zeta$ explicit.
Then we need to prove that
\[\big\|\P_{r_i}\vartheta_i\Big(\delta^{-1}\big(K_{\eta+\delta,\zeta}-K_{\eta,\zeta}\big)-\p_{\eta}K_{\eta,\zeta}\Big)\vartheta_j^{-1}P_{r_j}\big\|_1\xrightarrow[\delta\to0]{}0\]
with $\p_{\eta}K_{\eta,\zeta}(u,v)=K_{\eta,\zeta}(u,v-1)-K_{\eta,\zeta}(u+1,v)$ under the assumption $r_i>r_0(t,x_i)$, $r_j>r_0(t,x_j)$.
Using \eqref{eq:PNG-kernel-two-sided-hit} we can write $K_{\eta,\zeta}=-\uno{x_i<x_j}e^{(x_j-x_i)\Delta}+e^{-2(\eta_i-\zeta_i)\nabla-(\eta_i+\zeta_i)\Delta}\J^{h_0}_{2\eta_i,2\zeta_j}e^{2(\eta_j-\zeta_j)\nabla+(\eta_j+\zeta_j)\Delta}$,
with $\eta_k=\frac12(x_k+t)$, $\zeta_k=\frac12(x_k-t)$.
Then
\begin{multline}\label{eq:Kdelta}
K_{\eta+\delta,\zeta}-K_{\eta,\zeta}=(e^{-2\delta\nabla-\delta\Delta}-I)K_{\eta,\zeta}+e^{-2\delta\nabla-\delta\Delta}K_{\eta,\zeta}(e^{2\delta\nabla+\delta\Delta}-I)\\
+e^{-2(\eta_i+\delta-\zeta_i)\nabla-(\eta_i+\delta+\zeta_i)\Delta}\big(\J^{h_0}_{2\eta_i+2\delta,2\zeta_j}-\J^{h_0}_{2\eta_i,2\zeta_j}\big)e^{2(\eta_j+\delta-\zeta_j)\nabla+(\eta_j+\delta+\zeta_j)\Delta}.
\end{multline}
Focus on the first term on the right hand side.
After dividing by $2\delta$, the factor $e^{-2\delta\nabla-\delta\Delta}-I$ converges to $-(\nabla+\frac12\Delta)$, so we expect that the whole kernel, divided by $\delta$ and evaluated at $(u,v)$, should converge to $K_{\eta,\zeta}(u,v)-K_{\eta,\zeta}(u+1,v)$.
This convergence has to be proved in trace norm, after properly conjugating and adding the projections.
Following the arguments of Sec. \ref{sec:trcl-ker} we see that this essentially depends on estimating the $\ell^2$ norm of
\[\P_{r_i}\vartheta_i(u)\big((2\delta)^{-1}\big(e^{-2(t+\delta)\nabla-(x_i+\delta)\Delta}(u,\eta)-e^{2t\nabla+x_i\Delta}(u,\eta)\big)
-e^{2t\nabla+x_i\Delta}(u,\eta)+e^{2t\nabla+x_i\Delta}(u+1,\eta)\big)\]
for fixed $\eta\in\zz$, as a function of $u$.
Now \eqref{eq:Scontour} shows that
$\tfrac12\big(e^{-2(t+\delta)\nabla-(x+\delta)\Delta}(u,\eta)-e^{2t\nabla+x\Delta}(u,\eta)\big)=\frac1{2\pi\I}\oint_{\gamma}\d z\,\frac{e^{t(z^{-1}-z)-x(z+z^{-1})+2x}}{z^{\eta-u+1}}\frac{e^{-2\delta z+2\delta}-1}2=\frac{\delta}{2\pi\I}\oint_{\gamma}\d z\,\frac{e^{t(z^{-1}-z)+x(z+z^{-1})+2x}}{z^{\eta-u+1}}(1-z+\mathcal{O}(\delta))$,
where the $\mathcal{O}(\delta)$ term can be bounded uniformly in $z$ and $\eta$.
Using this in the above expression we see that the desired $\ell^2$ norm is bounded by $C\tts\delta\|\P_{r_i}\vartheta_ie^{-2t\nabla-t\Delta}(\cdot,\eta)\|_2$, and using this in the estimates of Sec. \ref{sec:trcl-ker} (see e.g. \eqref{eq:bmSMhitdecomp} and \eqref{eq:cbu78}) shows the desired convergence of $(2\delta^{-1})(e^{-2\delta\nabla-\delta\Delta}-I)K^{\eta,\zeta}$.

Going back to \eqref{eq:Kdelta}, the second term converges in essentially the same way as above, after dividing by $2\delta$, to $K_{\eta,\zeta}(u,v-1)-K_{\eta,\zeta}(u,v)$.
The last term in \eqref{eq:Kdelta} involves what in Sec. \ref{sec:pfMain} was defined as $\W^{(\eta)}_{t,x_i,x_j}$ (here there is an extra conjugation by $e^{-2t\nabla-t\Delta}$, which is of no consequence).
In Lem. \ref{lem:petazeta4} we showed that $\W^{(\eta)}_{t,x_i,x_j}(u,v)$ vanishes for $u>r_0(t,x_i)+1$, $v>r_0(t,x_j)+1$.
This is precisely our setting, because we are surrounding by projections $\P_{r_i}$ and $\P_{r_j}$ and assuming $r_i>r_0(t,x_i)$, $r_j>r_0(t,x_j)$.
What is left is to upgrade this pointwise computation to a convergence in trace norm, but this can be done without too much difficulty by using the arguments of this section in the proof of Lem. \ref{lem:petazeta4}; we leave the details to the reader.
This finishes the proof of the desired trace norm convergence of the $\zeta$ derivative.

\subsection{Proof of (\ref{eq:detKth-bd})}\label{sec:pf-detKth-bd}
Recall the setting of Case 3 in the proof of Thm. \ref{thm:Kolmogorov}.
We are given a $\UC$ function $h$ with locally finite $\pm1$ jumps and parameters $t$, $x_1,\dotsc,x_n$ and $r_1,\dotsc,r_n$ such that $h$ has an up jump of size one at $x_k+t$, no jumps at $x_k-t$ nor at $x_j\pm t$ for each $j\neq k$, and $r_j\geq r_0(t,x_j)$ for each $j\neq k$ while $r_0(t^-,x_k)=r_k=r_0(t,x_k)-1$.
$(\vec y,\vec\sigma)$ is the configuration of jump locations and signs associated to $h$, and as a convention we assume that the jump at $x_k+t$ has label $0$, i.e. $y_0=x_k+t$ (and $\sigma_0=1$).
Letting $K(t,h)$ denote the PNG extended kernel with these choices, which we will abbreviate as $K_t^0$, and $K(t,h)|_{\y_0=x_k+t+\ep}$ denote the same kernel with the jump at $y_0=x_k+t$ moved to $x_k+t+\ep$, which we will abbreviate as $K_t^\ep$, we want to show that
\begin{equation}
\big|\tsm\det(I-\P_rK_t^\ep\P_r)-\det(I-\P_rK_{t-\ep}^0\P_r)\big|\leq C\label{eq:cu22}
\end{equation}
for some $C>$ which does not depend on $\ep$ and $h$ and which is bounded uniformly in choices of $t$ on a compact set.
We will put tildes on top of kernels when they are conjugated by $\vartheta$ and multiplied by projections $\P_r$ on both sides as in the previous subsections.
Prop. \ref{prop:trcl} shows that both $\tilde K_{t-\ep}^0$ and $\tilde K_t^{\ep}$ are trace class, with trace norms which are bounded uniformly in $t$ in a compact set and $h\in\UC$, and in fact from the argument it is not hard to see that these are also bounded uniformly in $\ep\in(0,1)$.
Then, using the bound $|\ttsm\det(I-A)-\det(I-B)|\leq\|A-B\|_1e^{1+\|A\|_1+\|B\|_1}$ for trace class operators $A$ and $B$ we deduce that, for some suitably uniformly chosen constant $C$,
\begin{equation}\label{eq:cu23}
\big|\tsm\det(I-\tilde K_t^{\ep})-\det(I-\tilde K_{t-\ep}^0)\big|\leq C\big\|\tilde K_{t-\ep}^0-\tilde K_t^{\ep}\big\|_1.
\end{equation}

In order to estimate the norm on the right hand side of \eqref{eq:cu23} it is enough to do the same for each of the $i,j$ entries $\Delta_{ij}^{\ep}\coloneqq K_{t-\ep}^0(x_i,\cdot;x_j,\cdot)-K_t^{\ep}(x_i,\cdot;x_j,\cdot)$, which we decompose as
\begin{multline}\label{eq:cu65}
\Delta_{ij}^\ep=e^{-2(t-\ep)\nabla-(t-\ep)\Delta}P^{\hit(h)}_{x_i-t,x_j+t}e^{2(t-\ep)\nabla-(t-\ep)\Delta}-e^{-2t\nabla-t\Delta}P^{\hit(h)}_{x_i-t,x_j+t}e^{2t\nabla-t\Delta}\\
+e^{-2t\nabla-t\Delta}\big(P^{\hit(h)}_{x_i-t,x_j+t}-P^{\hit(h^{\ep})}_{x_i-t,x_j+t}\big)e^{2t\nabla-t\Delta},
\end{multline}
where $h^{\ep}$ denotes $h$ with the jump at $y_0=x_k+t$ moved to $x_k+t+\ep$.
The difference on the right hand side of the first line can be estimated in trace norm, after conjugating and projecting, using the arguments of Sec. \ref{sec:trcl-ker}, and it is not hard to obtain a bound of the form $C\tts\ep$; we omit the details.
Hence, in order to finish the proof we need to show that the same holds for the second line of \eqref{eq:cu65}, i.e. that
\begin{equation}\label{eq:cu167}
\big\|\vartheta_i\P_{r_i}e^{-2t\nabla-t\Delta}\big(P^{\hit(h)}_{x_i-t,x_j+t}-P^{\hit(h^{\ep})}_{x_i-t,x_j+t}\big)e^{2t\nabla-t\Delta}\vartheta_j^{-1}\P_{r_j}\big\|_1\leq C\tts\ep.
\end{equation}

To prove \eqref{eq:cu167} we need to separate cases according to where $x_k+t$ lies with respect to $x_i-t$ and $x_j+t$.
The simplest situation is $x_k+t\notin[x_i-t,x_j+t]$; in this case $h=h^\ep$ inside $[x_i-t,x_j+t]$ (for small enough $\ep$) so the left hand side of \eqref{eq:cu167} vanishes.
If $k=j$ then $\big(P^{\hit(h)}_{x_i-t,x_j+t}-P^{\hit(h^{\ep})}_{x_i-t,x_j+t}\big)(u,\xi)=P^{\hit(h)}_{x_i-t,x_j+t}(u,h(x_k+t)-1)\uno{\xi=h(x_k+t)-1}$, which gives $\big(P^{\hit(h)}_{x_i-t,x_j+t}-P^{\hit(h^{\ep})}_{x_i-t,x_j+t}\big)\vartheta_j^{-1}\P_{r_j}(u,v)=P^{\hit(h)}_{x_i-t,x_j+t}(u,h(x_k+t)-1)e^{2t\nabla-t\Delta}(h(x_k+t)-1,v)\vartheta_j(v)\uno{v>r_j}$, but in this case $r_j=r_k=r_0(t,x_k)-1\geq h(x_k+t)-1$, so by Lem. \ref{lem:ett} the kernel vanishes again.

Finally we consider \eqref{eq:cu167} when $x_k+t\in[x_i-t,x_j+t)$.
In particular, this implies that the interval is non-empty, and we may assume that it also includes $x_k+t+\ep$.
Note that the event inside the probabilities $P^{\hit(h)}_{x_i-t,x_j+t}$ and $P^{\hit(h^{\ep})}_{x_i-t,x_j+t}$ is the same if we restrict to paths where the random walk $\fN$ stays put between time $x_k+t$ and $x_k+t+\ep$, and event whose complement has probability $1-e^{-2\ep}\leq2\ep$.
The same argument as in the proof of Prop. \ref{prop:trcl} can now be employed to turn this into the desired bound \eqref{eq:cu167} (note that, $h$ and $h^\ep$ are, in fact, at Hausdorff distance $\ep$).

\vs

\noindent{\bf Acknowledgements.}
The authors thank Jinho Baik for comments on a draft of this article.
KM was partially supported by NSF grant DMS-1953859.
JQ was supported by the Natural Sciences and Engineering Research Council of Canada.
DR was supported by Centro de Modelamiento Matemático Basal Funds FB210005 from ANID-Chile and by Fondecyt Grants 1201914 and 1241974.
  
\printbibliography[heading=apa]

\end{document}